\documentclass[10pt]{amsart}
\usepackage{mathtools}
\usepackage{amssymb}
\usepackage{stmaryrd}
\usepackage{graphicx}
\usepackage[export]{adjustbox}
\usepackage{xcolor}
\usepackage{tikz}
\usepackage[colorlinks, linkcolor=blue, citecolor=blue, pagebackref]{hyperref}
\usepackage{amsrefs}
\usepackage{ytableau}
\usepackage{subcaption}
\usepackage{fullpage}
\usepackage[indent]{parskip}
\usepackage{enumitem}
\usepackage{setspace}
\usepackage{float}
\usepackage{halloweenmath}
\usepackage[scr=boondoxupr,scrscaled=1.1]{mathalfa}

\theoremstyle{plain}
\newtheorem{theorem}{Theorem}[section]
\newtheorem{lemma}[theorem]{Lemma}

\newtheorem*{thm*}{Theorem}

\theoremstyle{definition}

\newtheorem{rem}[theorem]{Remark}
\newtheorem{problem}[theorem]{Problem}
\newtheorem{algorithm}[theorem]{Algorithm}

\numberwithin{equation}{section}

\newcommand{\la}{\lambda}
\newcommand{\SYT}{{\rm SYT}}
\newcommand{\fS}{\mathfrak{S}}
\newcommand{\RS}{\operatorname{RS}}
\newcommand{\se}[1]{\mathsf{#1}}
\newcommand{\gT}{g_{_{\mathrm{T}}}}
\newcommand{\gH}{g_{_{\mathrm{H}}}}
\newcommand{\gV}{g_{_{\mathrm{V}}}}
\newcommand{\gR}{g_{_{\mathrm{R}}}}

\def\shortmset#1#2{\ensuremath{\left(\kern-.28em\left(\genfrac{}{}{0pt}{}
{#1}{#2}\right)\kern-.28em\right)}}

\def\tallmset#1#2{\ensuremath{\left(\kern-.4em\left(\genfrac{}{}{0pt}{}
{#1}{#2}\right)\kern-.4em\right)}}

\def\tallballot#1#2{\ensuremath{\left\langle\kern-.45em \left\langle \genfrac{}{}{0pt}{}
{#1}{#2}%
   \right\rangle \kern-.45em \right\rangle}}

\def\shortballot#1#2{\ensuremath{\left\langle\kern-.32em \left\langle \genfrac{}{}{0pt}{}
{#1}{#2}%
   \right\rangle \kern-.32em \right\rangle}}

\begin{document}

\title{Explicit marginal distributions for permutations \\ with prescribed Robinson--Schensted shape}

\author{William Q.~Erickson}
\address{
University of Tennessee at Chattanooga \\ 
615 McCallie St.\\
Chattanooga, TN 37403}
\email{william-erickson01@utc.edu}



\begin{abstract}

Given a permutation $\sigma$, the Robinson--Schensted correspondence determines a certain partition called the \emph{shape} of $\sigma$.
Famously, the shape measures the longest unions of increasing and decreasing subsequences, thus giving global information about $\sigma$.
In this paper, by contrast, we ask how prescribing a shape collectively controls \emph{local} behavior:
if $\sigma$ is a random permutation of shape $\la$, then what is $P^\la_{ij} \coloneqq$ the probability that $\sigma(i) = j$?
Using tableau-theoretic methods, we derive explicit formulas for $P^\la_{ij}$ when $\la$ is a hook, two-row, or rectangular shape.
We use these formulas to depict and analyze the intricate diffraction-like patterns in the matrices $(P^\la_{ij})$.
As a surprising application, we show that for both hook and two-row shapes, as the largest part of $\la$ tends to infinity with the remaining parts fixed (summing to $m$), the expected proportion of fixed points in $\sigma$ approaches the Wallis integral $\int_0^{\pi/2} \sin^{2m+1} x \: dx = (2m)!! / (2m+1)!!$.
\end{abstract}

\maketitle

\ytableausetup{centertableaux,smalltableaux}

\section{Introduction}

Let $\fS_n$ denote the symmetric group on $n$ letters.
In this paper, we solve the following problem in the case where the partition $\la$ is a hook or two-row or rectangular shape:

\begin{problem}
  \label{problem:Pij}
  Let $\sigma \in \fS_n$ be chosen uniformly at random, and let $\lambda$ be a partition of $n$.
  For all $1 \leq i,j \leq n$, find an explicit formula for the conditional probability
  \[
    P^\la_{ij} \coloneqq \mathbb{P} \Big( \sigma(i) = j \; \Big| \; \sigma \text{ has Robinson--Schensted shape } \la \Big).
  \]
\end{problem}

Programming our solution has led us to realize that the doubly stochastic matrix $P^\la = (P^\la_{ij})$ exhibits intricate ``diffraction''-like patterns that imitate certain aspects of the Young diagram of $\la$;
see Figures~\ref{fig:n15 hooks}--\ref{fig:rectangles}, which show the heat maps of $P^\la$ for hook, two-row, and rectangular shapes, respectively.
Our $P^\la_{ij}$ formulas not only describe these patterns explicitly, but also enable us to generate their pictures in the first place, since a brute-force approach becomes computationally prohibitive even for relatively small $n$.
We are hopeful that the results and methods in the paper might provide a launch point in solving Problem~\ref{problem:Pij} for other families of shapes, or ultimately for arbitrary~$\la$.

\subsection*{Motivation}

At the risk of understatement, the Robinson--Schensted (RS) correspondence is a fundamental concept in algebraic combinatorics.
In itself, the RS correspondence is a bijection from $\fS_n$ to the set of ordered pairs of same-shaped standard Young tableaux with $n$ boxes;
as an example where $n=7$, we have
\begin{equation}
  \label{example intro}
  \RS : (2, 7, 4, 6, 3, 1, 5) \longmapsto \left( \ytableaushort{124,37,5,6}_{\!\!\textstyle,} \;\; \ytableaushort{135,26,4,7} \right),
\end{equation}
with the intermediate details given in Figure~\ref{fig: ex of RSK}.
At a less granular level, however, when we forget the tableau entries, the RS correspondence simply associates to each permutation in $\fS_n$ a certain \emph{shape}, namely the partition of $n$ given by the row lengths of either tableau.
For example, the permutation in~\eqref{example intro} has shape $(3,2,1,1)$. 
The shape of a permutation $\sigma$ encodes global information about $\sigma$:
Greene's theorem~\eqref{Greene} states that the partial sums of the shape's row (resp., column) lengths give the lengths of the longest unions of increasing (resp., decreasing) subsequences in $\sigma$.
Thus, since prescribing the shape of a random permutation~$\sigma$ controls the lengths of its monotone subsequences, this should presumably affect the expected \emph{local} behavior of $\sigma$, meaning the marginal distribution of $\sigma(i)$ for each $1 \leq i \leq n$.
The effect of shape on the marginal distributions seems not to have been studied in the literature;
our desire to understand and quantify it led us to pose Problem~\ref{problem:Pij}, the central problem in this project.

\ytableausetup{smalltableaux}

\begin{figure}[t]
\captionsetup[subfigure]{labelformat=empty}
  \centering
     \begin{subfigure}[b]{0.1\textwidth}
         \centering
         \includegraphics[frame, width=\textwidth]{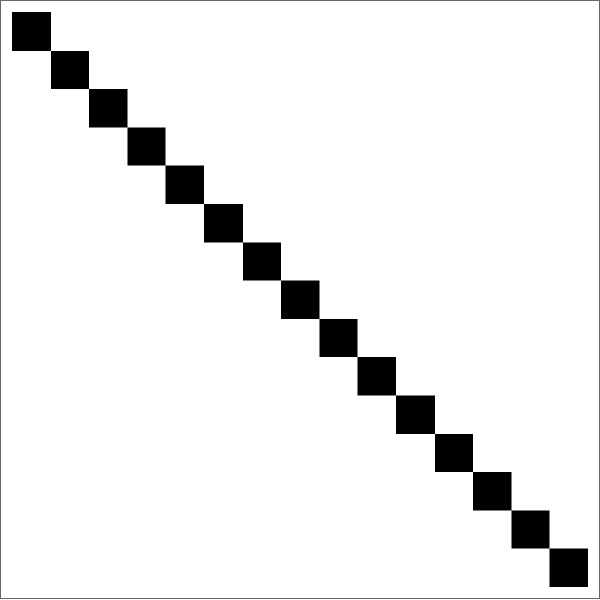}
         \caption{$\lambda\!=\!(15,1^0)$}
         \end{subfigure}
     \hfill
     \begin{subfigure}[b]{0.1\textwidth}
         \centering
         \includegraphics[frame, width=\textwidth]{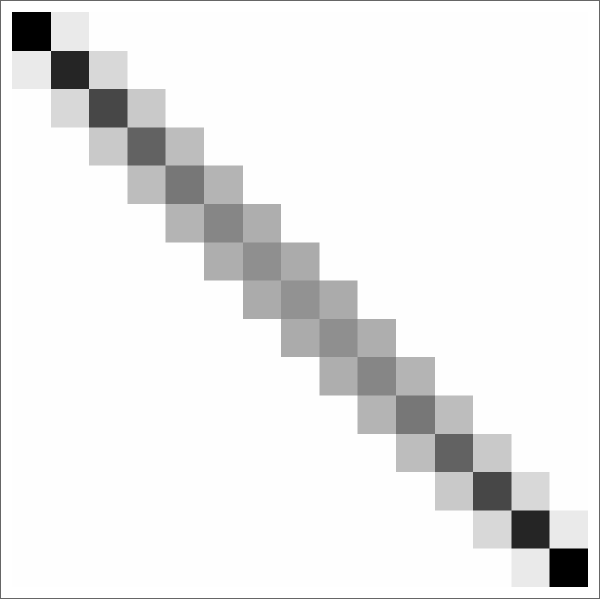}
         \caption{$\lambda\!=\!(14,1^1)$}
         \end{subfigure}
     \hfill\begin{subfigure}[b]{0.1\textwidth}
         \centering
         \includegraphics[frame, width=\textwidth]{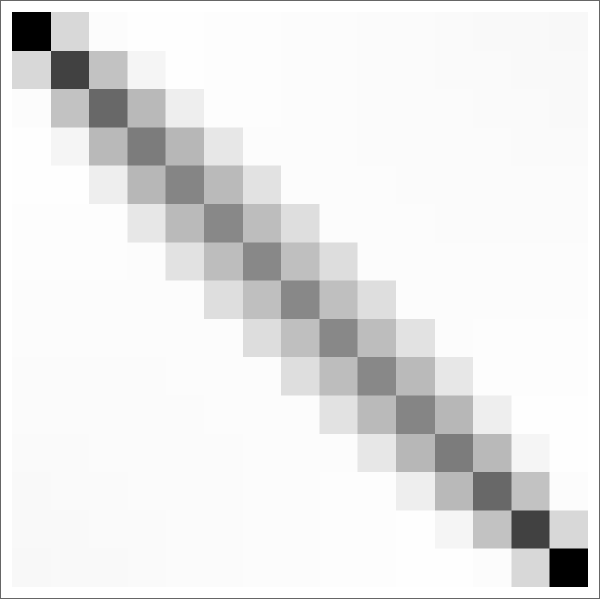}
         \caption{$\lambda\!=\!(13,1^2)$}
         \end{subfigure}
     \hfill\begin{subfigure}[b]{0.1\textwidth}
         \centering
         \includegraphics[frame, width=\textwidth]{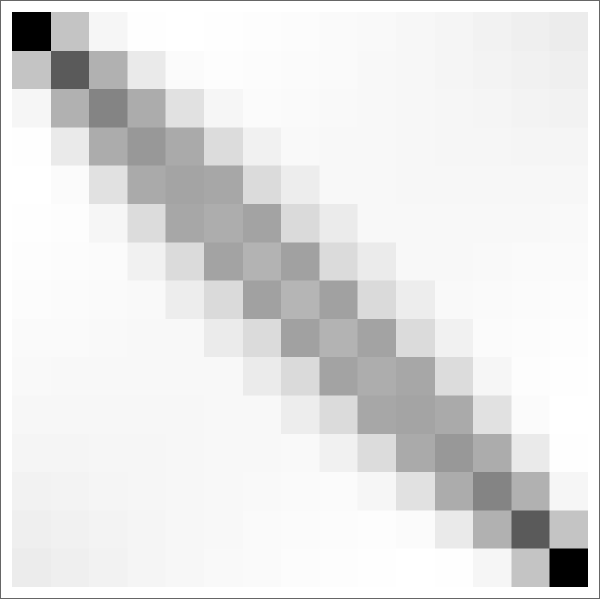}
         \caption{$\lambda\!=\!(12,1^3)$}
         \end{subfigure}
     \hfill\begin{subfigure}[b]{0.1\textwidth}
         \centering
         \includegraphics[frame, width=\textwidth]{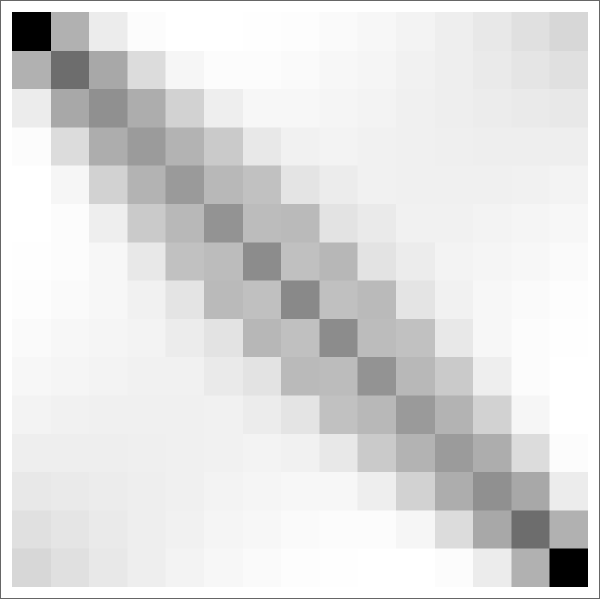}
         \caption{$\lambda\!=\!(11,1^4)$}
         \end{subfigure}
     \hfill\begin{subfigure}[b]{0.1\textwidth}
         \centering
         \includegraphics[frame, width=\textwidth]{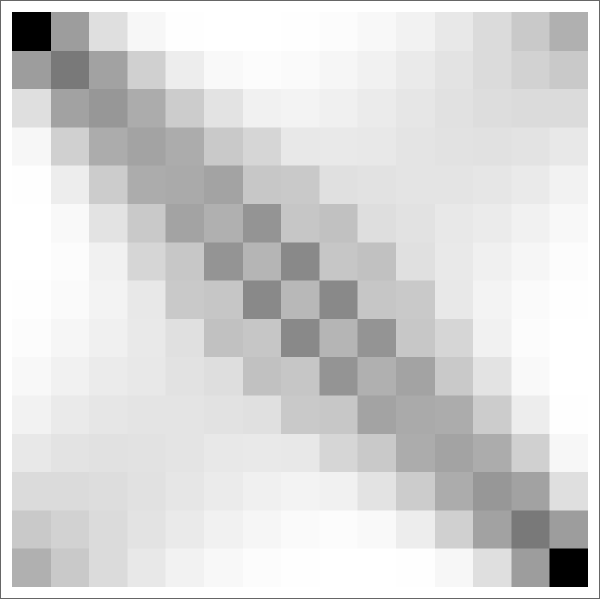}
         \caption{$\lambda\!=\!(10,1^5)$}
         \end{subfigure}
     \hfill\begin{subfigure}[b]{0.1\textwidth}
         \centering
         \includegraphics[frame, width=\textwidth]{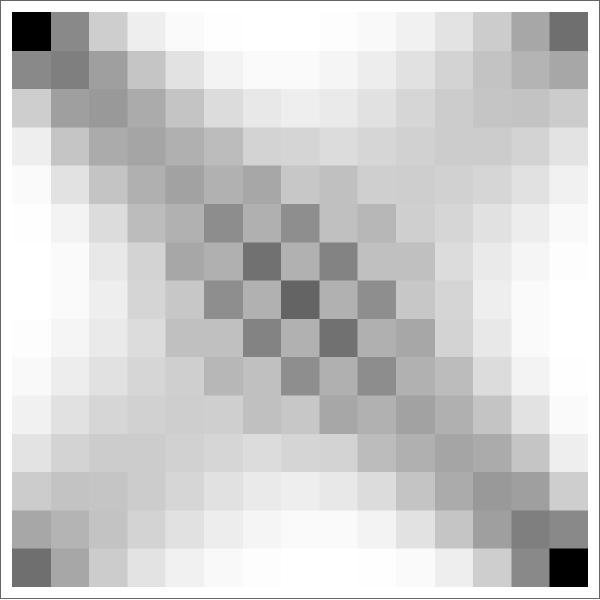}
         \caption{$\lambda\!=\!(9,1^6)$}
         \end{subfigure}
     \hfill\begin{subfigure}[b]{0.1\textwidth}
         \centering
         \includegraphics[frame, width=\textwidth]{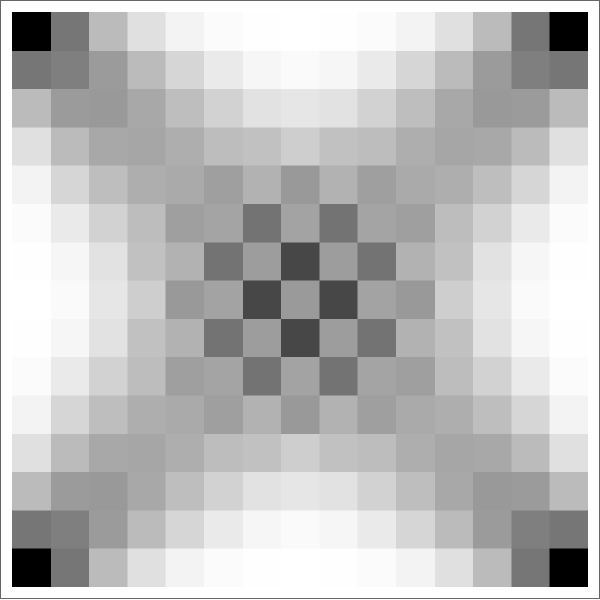}
         \caption{$\lambda\!=\!(8,1^7)$}
         \end{subfigure}
     \hfill

        \caption{Heat maps of $P^\la$ where $\la = (\ell, 1^{n-\ell})$ is a hook shape, for $n=15$.
        In each heat map, the entry $P^\la_{ij}$ gives the probability that $\sigma(i) =j$ given that a uniform random~$\sigma$ has shape $\la$;
        white represents 0, and black represents the maximum matrix entry.
        These were generated in Mathematica using our explicit formula for $P^\la_{ij}$ in Theorem~\ref{thm:hooks in intro}.
        In light of Lemma~\ref{lemma:symmetries}, the cases $\ell = 8, 7, \ldots, 2, 1$ are obtained by holding the page up to a mirror.}
        \label{fig:n15 hooks}
\end{figure}

\begin{figure}[t]
\captionsetup[subfigure]{labelformat=empty}
  \centering
     \begin{subfigure}[b]{0.1\textwidth}
         \centering
         \includegraphics[frame, width=\textwidth]{Heatmaps15/a15.pdf}
         \caption{$\lambda\!=\!(15, 0)$}
         \end{subfigure}
     \hfill
     \begin{subfigure}[b]{0.1\textwidth}
         \centering
         \includegraphics[frame, width=\textwidth]{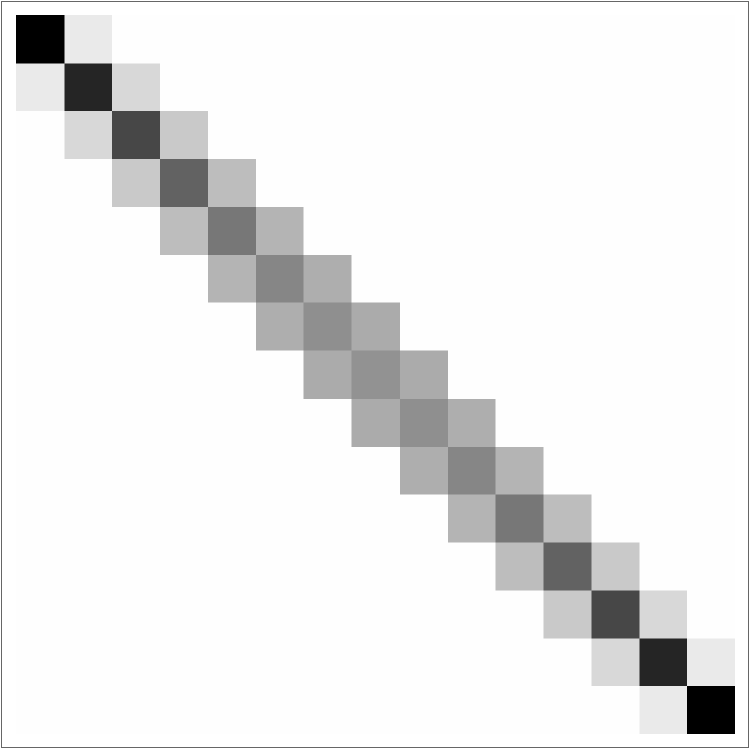}
         \caption{$\lambda\!=\!(14, 1)$}
         \end{subfigure}
     \hfill\begin{subfigure}[b]{0.1\textwidth}
         \centering
         \includegraphics[frame, width=\textwidth]{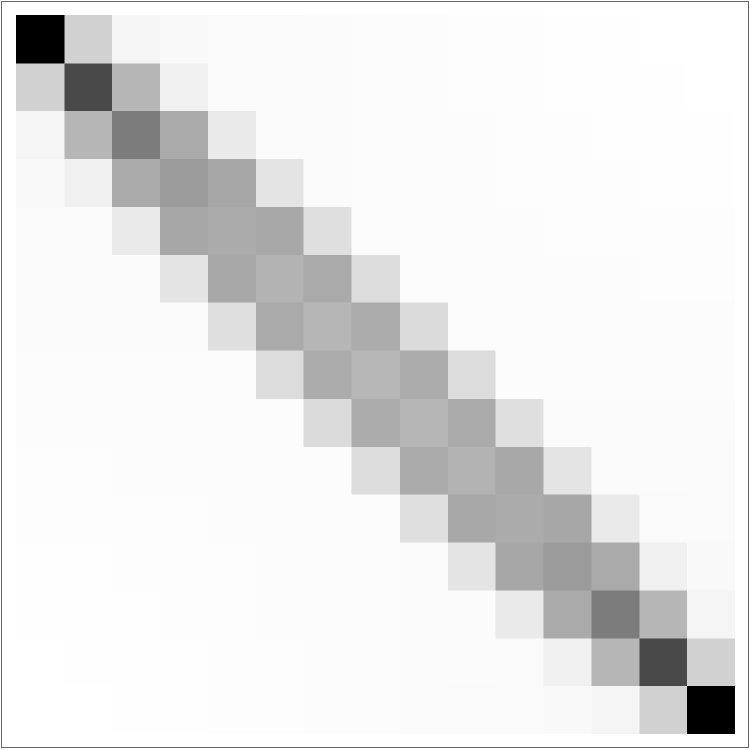}
         \caption{$\lambda\!=\!(13, 2)$}
         \end{subfigure}
     \hfill\begin{subfigure}[b]{0.1\textwidth}
         \centering
         \includegraphics[frame, width=\textwidth]{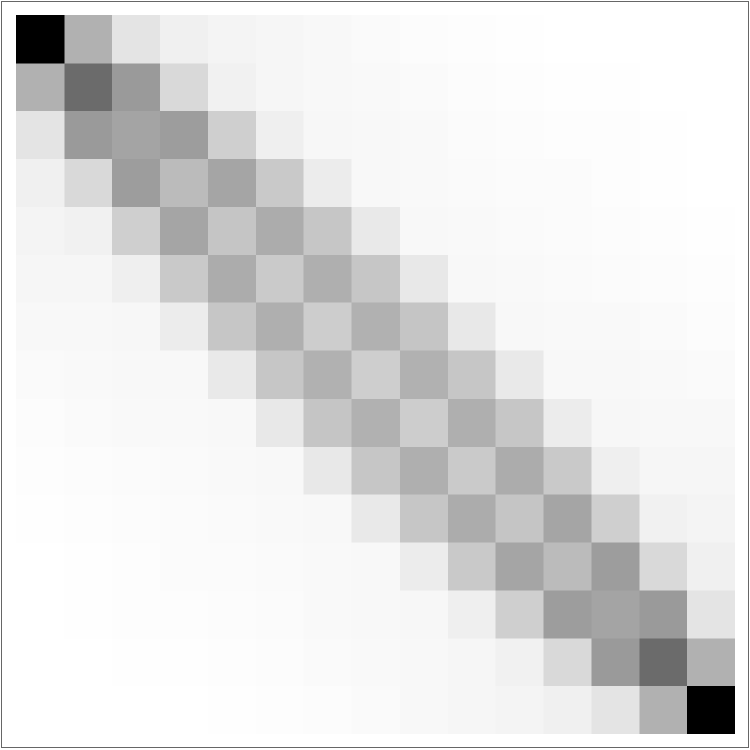}
         \caption{$\lambda\!=\!(12, 3)$}
         \end{subfigure}
     \hfill\begin{subfigure}[b]{0.1\textwidth}
         \centering
         \includegraphics[frame, width=\textwidth]{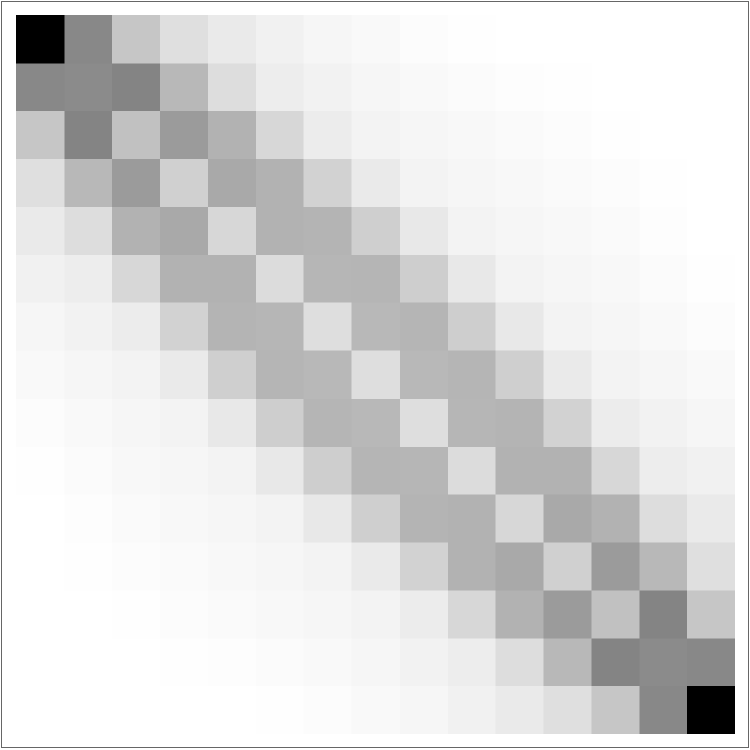}
         \caption{$\lambda\!=\!(11, 4)$}
         \end{subfigure}
     \hfill\begin{subfigure}[b]{0.1\textwidth}
         \centering
         \includegraphics[frame, width=\textwidth]{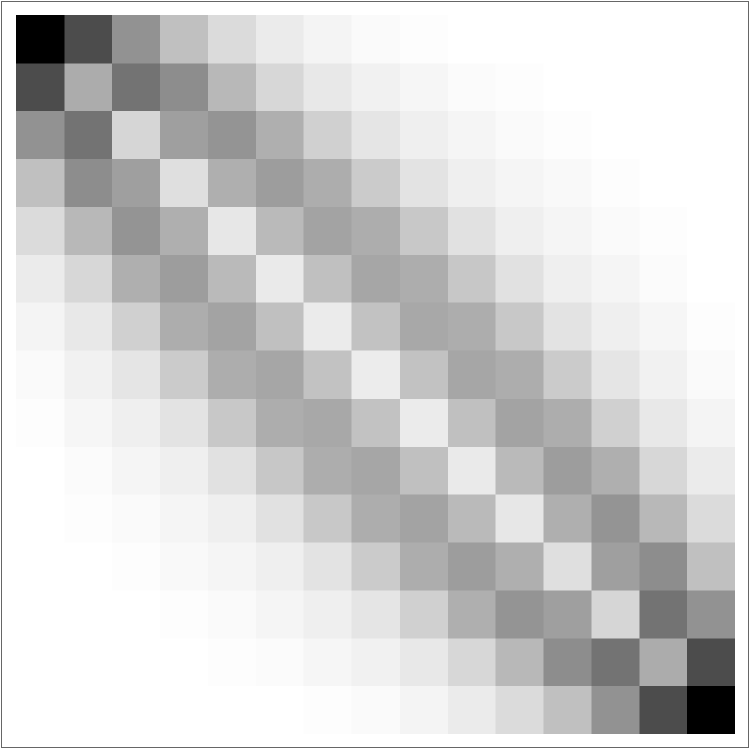}
         \caption{$\lambda\!=\!(10, 5)$}
         \end{subfigure}
     \hfill\begin{subfigure}[b]{0.1\textwidth}
         \centering
         \includegraphics[frame, width=\textwidth]{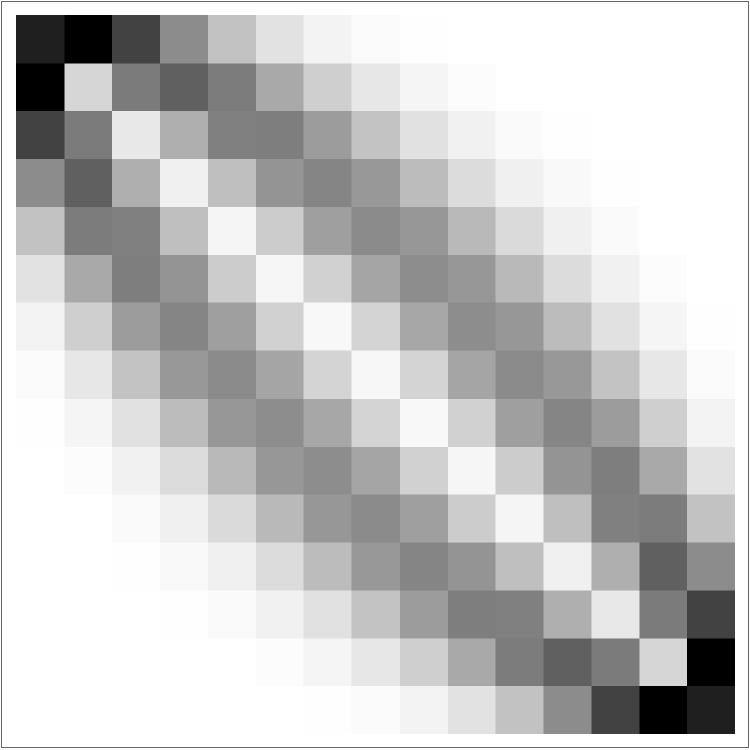}
         \caption{$\lambda\!=\!(9, 6)$}
         \end{subfigure}
     \hfill\begin{subfigure}[b]{0.1\textwidth}
         \centering
         \includegraphics[frame, width=\textwidth]{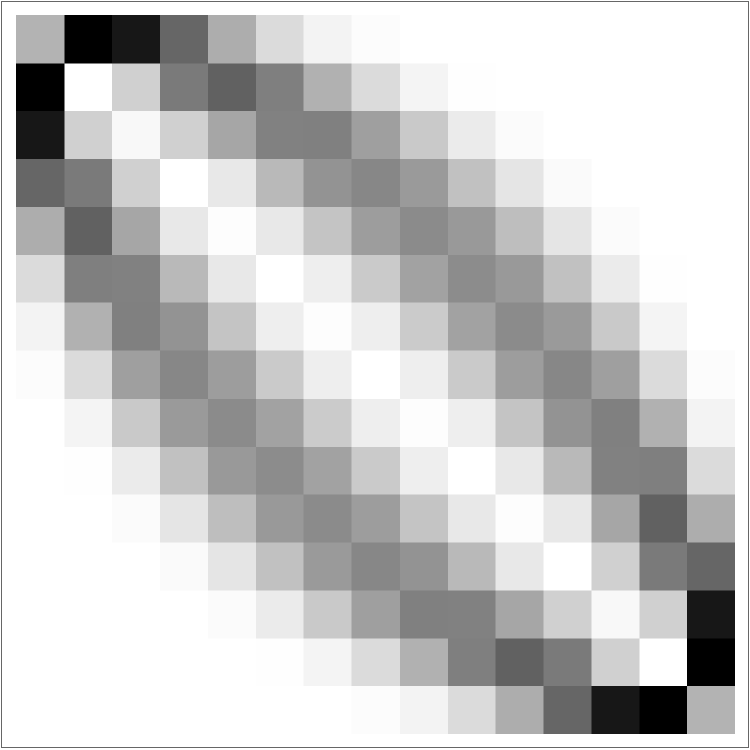}
         \caption{$\lambda\!=\!(8, 7)$}
         \end{subfigure}
     \hfill

        \caption{Heat maps of $P^\la$ where $\la = (\ell, n-\ell)$ is a two-row shape, for $n=15$.
        These were generated in Mathematica using our formula for $P^\la_{ij}$ in Theorem~\ref{thm:two-row in intro}.
        In light of Lemma~\ref{lemma:symmetries}, a horizontal (or vertical) reflection takes the heat map of $\la$ to that of the two-column shape~$\la'$.}
        \label{fig:n15 two-row}
\end{figure}

\begin{figure}[t]
\captionsetup[subfigure]{labelformat=empty}
  \centering
     \begin{subfigure}[b]{0.1\textwidth}
         \centering
         \includegraphics[frame, width=\textwidth]{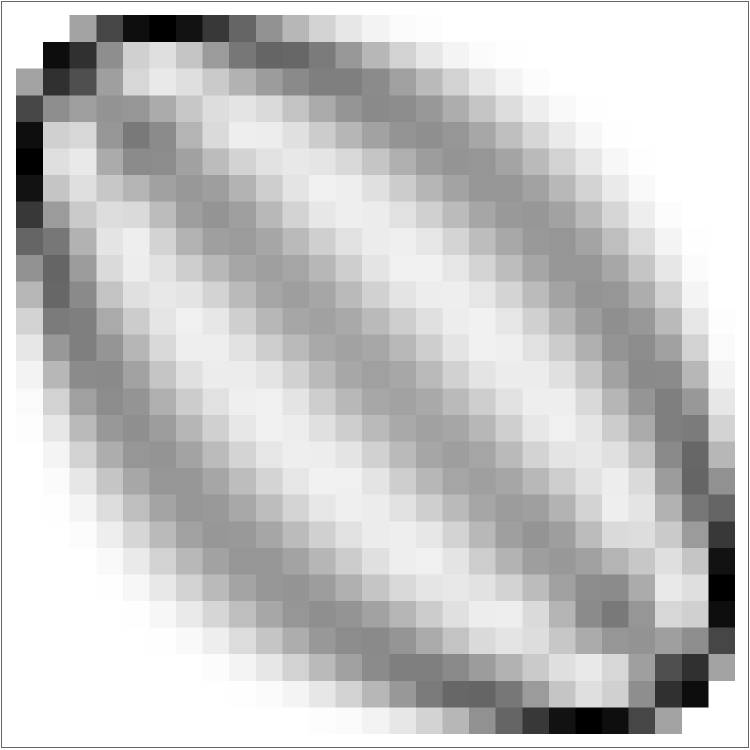}
         \caption{$\lambda=(9^3)$}
         \end{subfigure}
     \hfill
     \begin{subfigure}[b]{0.1\textwidth}
         \centering
         \includegraphics[frame, width=\textwidth]{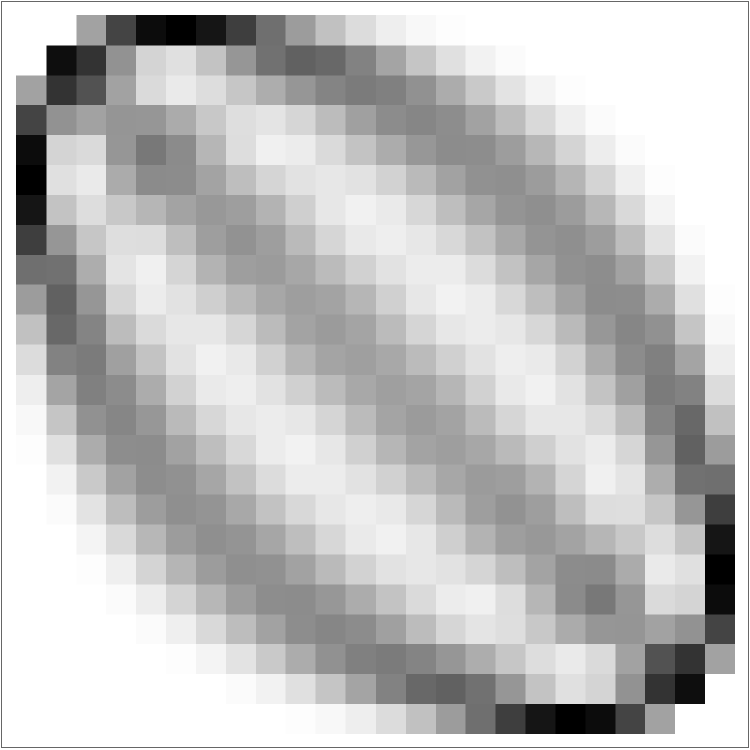}
         \caption{$\lambda=(8^3)$}
         \end{subfigure}
     \hfill\begin{subfigure}[b]{0.1\textwidth}
         \centering
         \includegraphics[frame, width=\textwidth]{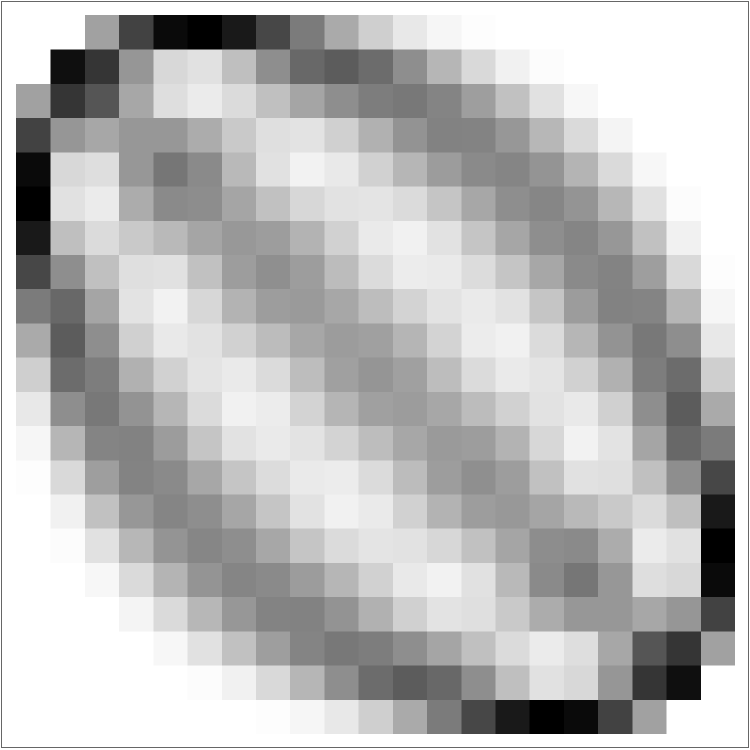}
         \caption{$\lambda=(7^3)$}
         \end{subfigure}
     \hfill\begin{subfigure}[b]{0.1\textwidth}
         \centering
         \includegraphics[frame, width=\textwidth]{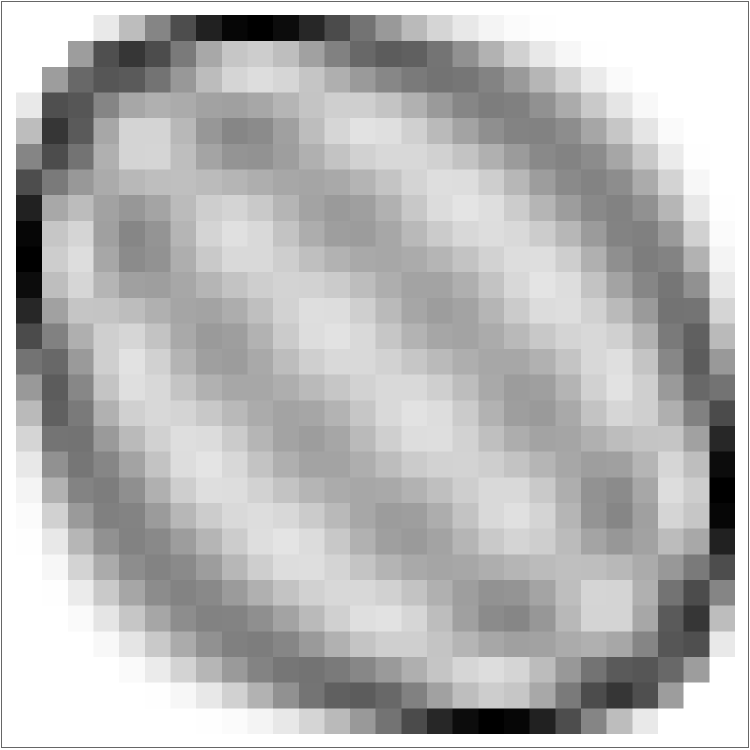}
         \caption{$\lambda=(7^4)$}
         \end{subfigure}
     \hfill\begin{subfigure}[b]{0.1\textwidth}
         \centering
         \includegraphics[frame, width=\textwidth]{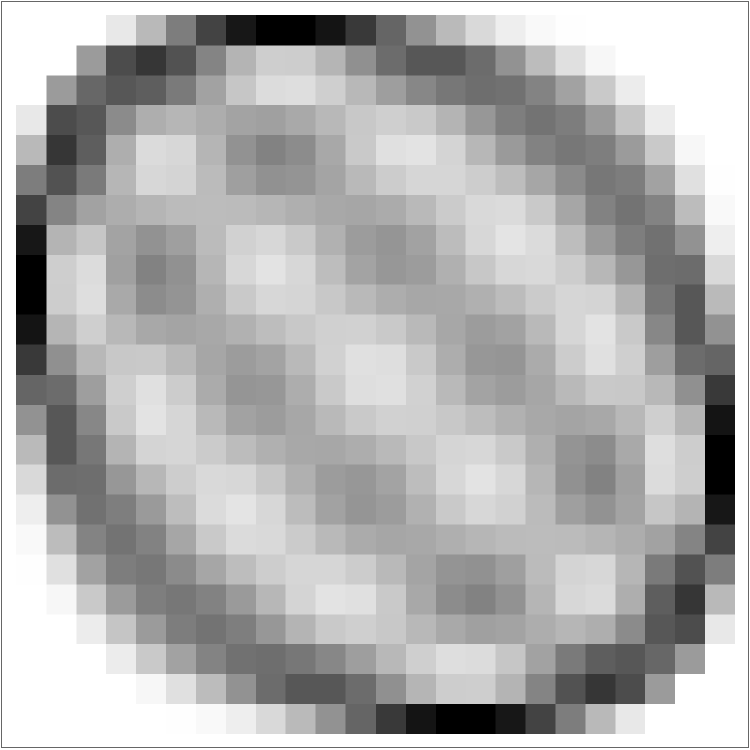}
         \caption{$\lambda=(6^4)$}
         \end{subfigure}
     \hfill\begin{subfigure}[b]{0.1\textwidth}
         \centering
         \includegraphics[frame, width=\textwidth]{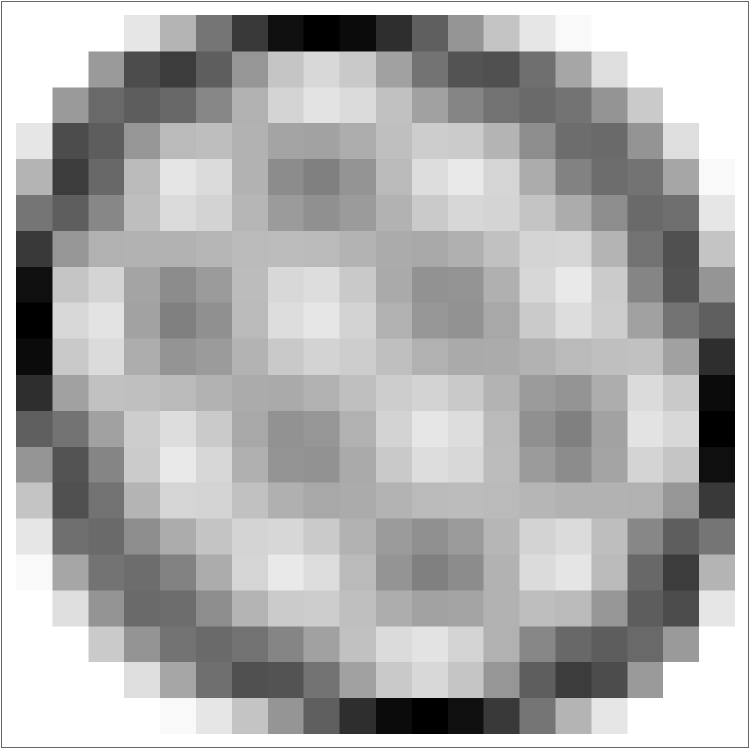}
         \caption{$\lambda=(5^4)$}
         \end{subfigure}
     \hfill\begin{subfigure}[b]{0.1\textwidth}
         \centering
         \includegraphics[frame, width=\textwidth]{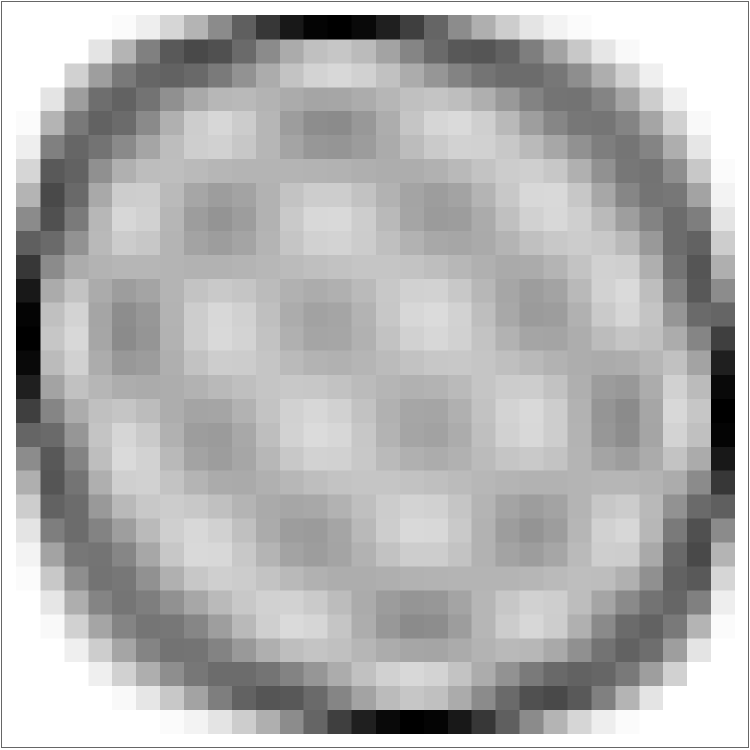}
         \caption{$\lambda=(6^5)$}
         \end{subfigure}
     \hfill\begin{subfigure}[b]{0.1\textwidth}
         \centering
         \includegraphics[frame, width=\textwidth]{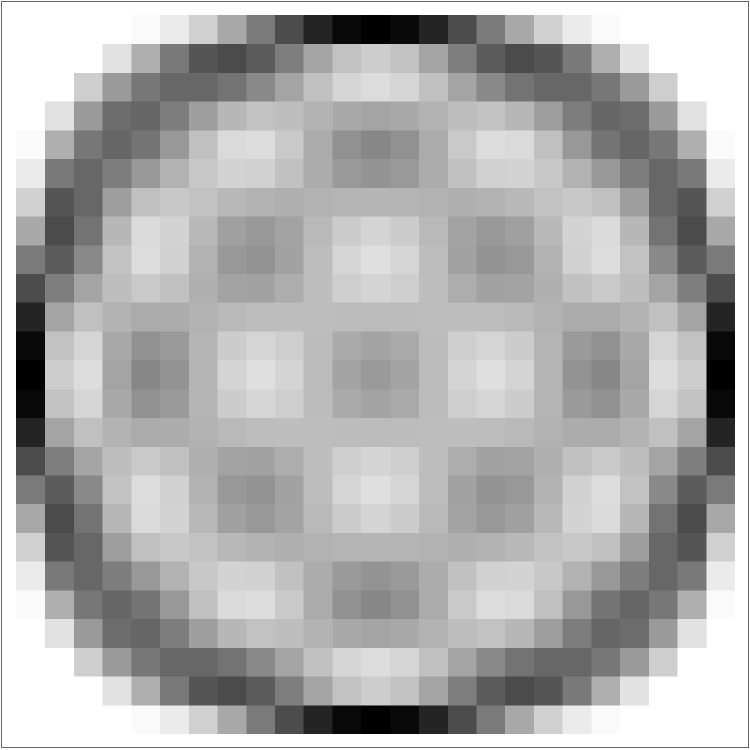}
         \caption{$\lambda=(5^5)$}
         \end{subfigure}
     \hfill

        \caption{Heat maps of $P^\la$ where $\la = (\ell^m)$ is a rectangular shape, showing most cases (up to symmetry) where $20 \leq n \leq 30$ with more than two rows.
        These were generated in Mathematica using our formula for $P^\la_{ij}$ in Theorem~\ref{thm:rectangles in intro}.
        The diffraction-like pattern forms a grid with $\ell$ dark antidiagonal bands and $m$ dark diagonal bands, imitating the (rotated) Young diagram of $\la$.}
        \label{fig:rectangles}
\end{figure}

Problem~\ref{problem:Pij} admits trivial solutions in the two extreme cases where $\la = (n)$ is a single row or $\la = (1^n) \coloneqq (1, \ldots, 1)$ is a single column.
In the one-row case, there is only one permutation of shape $\la = (n)$, namely the identity $(1,2,3,\ldots,n)$, and so we have $P^\la_{ij} = \delta_{ij}$.
Likewise, in the one-column case, there is only one permutation of shape $\la = (1^n)$, namely the exchange permutation $(n,\ldots, 3,2,1)$, and so we have $P^\la_{ij} = \delta_{i, n+1-j}$.
All other shapes interpolate between these two extreme cases in a deeply complicated way.

To gain some intuition for this, it is helpful to view the \emph{heat map} of the matrix $P^\la$, where each entry $P^\la_{ij}$ is given a shade between white (0) and black (the maximum matrix entry).
% Heat maps therefore show the way in which the distribution of mass among the marginal probabilities $P^\la_{ij}$ varies with $\la$.
See Figure~\ref{fig:heat maps intro} for the heat maps of all matrices $P^\la$ in the case $n=7$;
note the complicated way in which the mass spreads away from the diagonal, and then back toward the antidiagonal, as $\la$ varies between the two extremes $(n)$ and $(1^n)$.
Even at such a small value $n=7$, one can see the very coarse beginnings of the diffraction-like patterns in Figures~\ref{fig:n15 hooks}--\ref{fig:rectangles}.
Because $n=7$ is so small, one can easily generate the heat maps in Figure~\ref{fig:heat maps intro} by brute force, by averaging the permutation matrices of all permutations with shape~$\la$.
By contrast, the last heat map shown in Figure~\ref{fig:rectangles} would have required taking the average of approximately $5 \times 10^{17}$ permutation matrices.

A different perspective on Problem~\ref{problem:Pij} comes from the representation theory of $\fS_n$, where the group algebra $\mathbb{C}[\fS_n]$ acts on its defining representation $\mathbb{C}^n$, with each $\sigma \in \fS_n$ acting via its corresponding permutation matrix.
Then $P^\la$ is the matrix of the element in $\mathbb{C}[\fS_n]$ obtained by averaging over all permutations of shape~$\la$.
(These shape-equivalence classes are precisely the two-sided Kazhdan--Lusztig cells in $\fS_n$.)
It may be interesting in future work to investigate the determinant, eigenvalues, eigenvectors, etc.~of $P^\la$ and to study their representation-theoretic interpretations.

\ytableausetup{boxsize=.3em}

\begin{figure}[t]
\captionsetup[subfigure]{labelformat=empty}
  \centering
     \begin{subfigure}[b]{0.1\textwidth}
         \centering
         \includegraphics[frame, width=\textwidth]{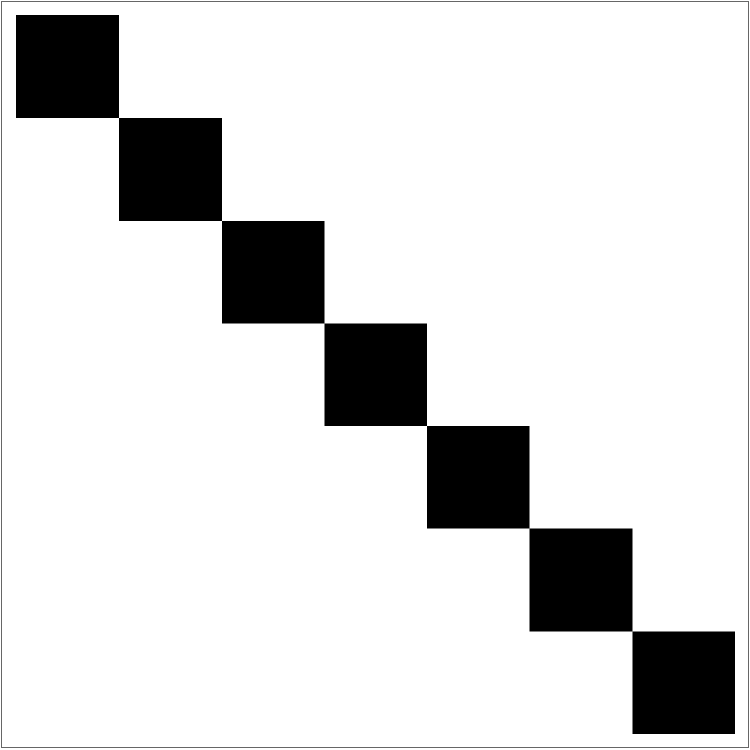}
         \caption{$\la = \ydiagram{7}$}
         \end{subfigure}
     \hfill
     \begin{subfigure}[b]{0.1\textwidth}
         \centering
         \includegraphics[frame, width=\textwidth]{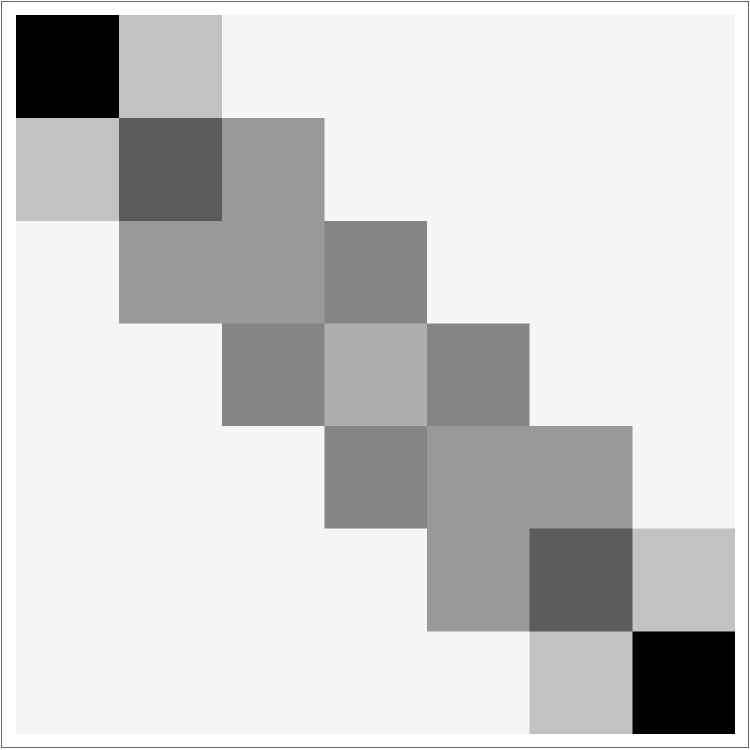}
         \caption{$\la=\ydiagram{6,1}$}
         \end{subfigure}
     \hfill\begin{subfigure}[b]{0.1\textwidth}
         \centering
         \includegraphics[frame, width=\textwidth]{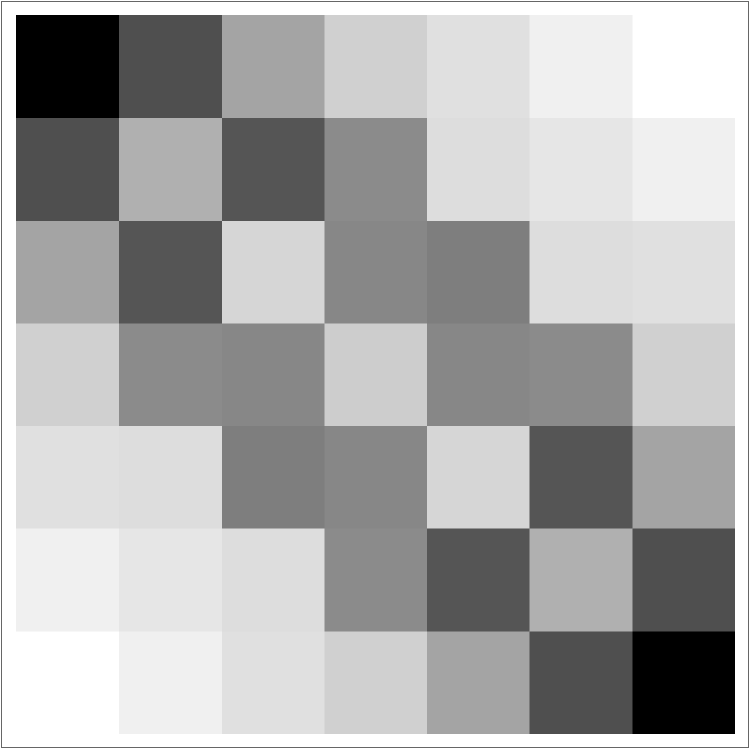}
         \caption{$\la = \ydiagram{5,2}$}
         \end{subfigure}
     \hfill\begin{subfigure}[b]{0.1\textwidth}
         \centering
         \includegraphics[frame, width=\textwidth]{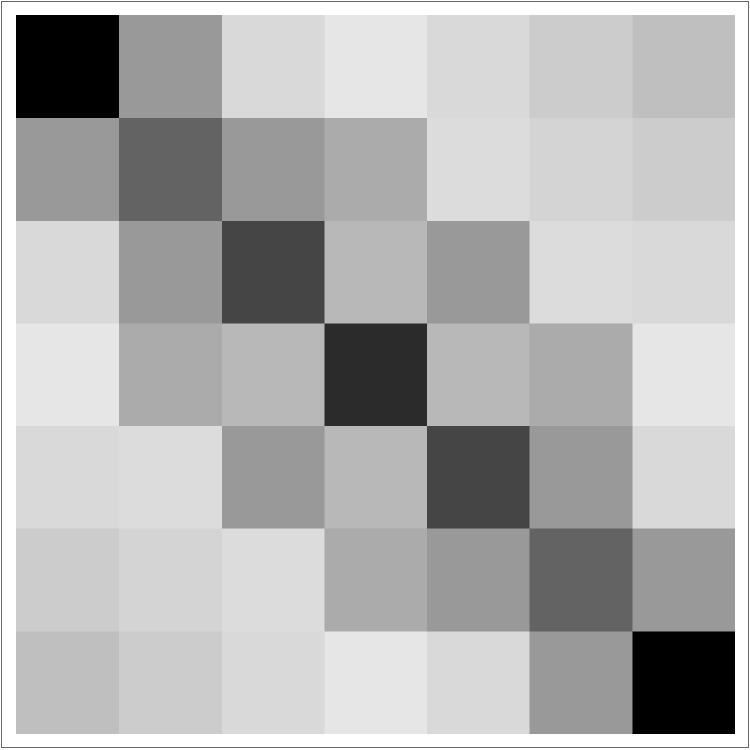}
         \caption{$\la = \ydiagram{5,1,1}$}
         \end{subfigure}
     \hfill\begin{subfigure}[b]{0.1\textwidth}
         \centering
         \includegraphics[frame, width=\textwidth]{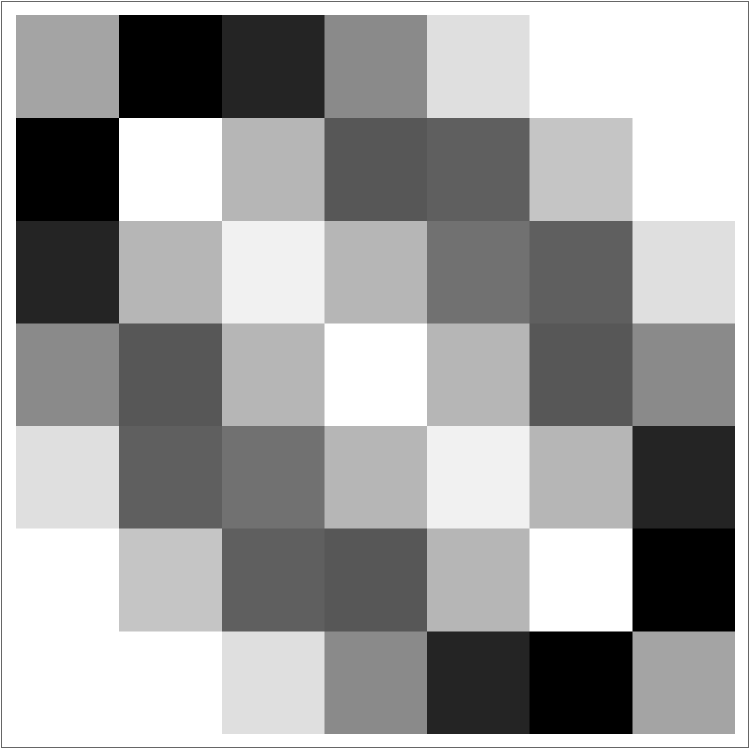}
         \caption{$\la = \ydiagram{4,3}$}
         \end{subfigure}
     \hfill\begin{subfigure}[b]{0.1\textwidth}
         \centering
         \includegraphics[frame, width=\textwidth]{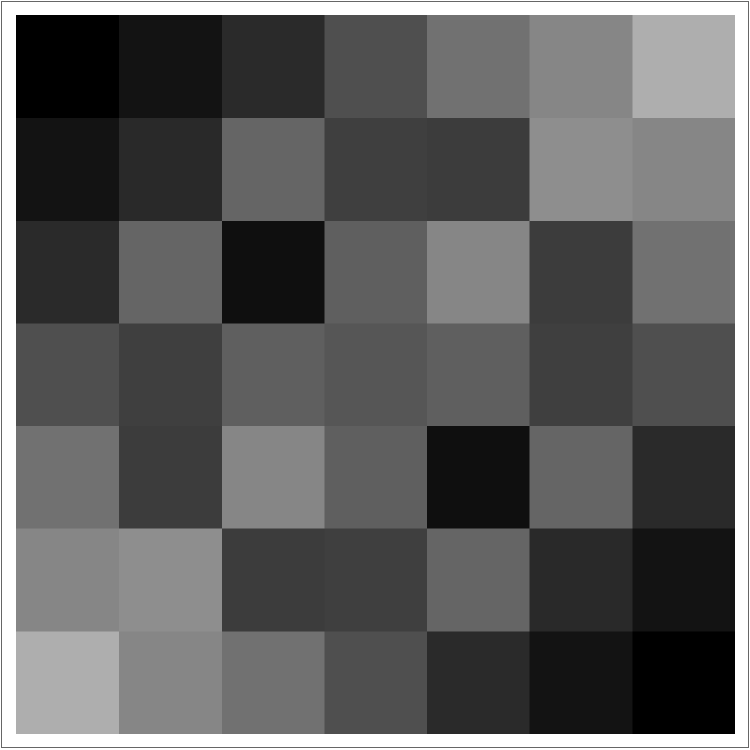}
         \caption{$\la = \ydiagram{4,2,1}$}
         \end{subfigure}
     \hfill\begin{subfigure}[b]{0.1\textwidth}
         \centering
         \includegraphics[frame, width=\textwidth]{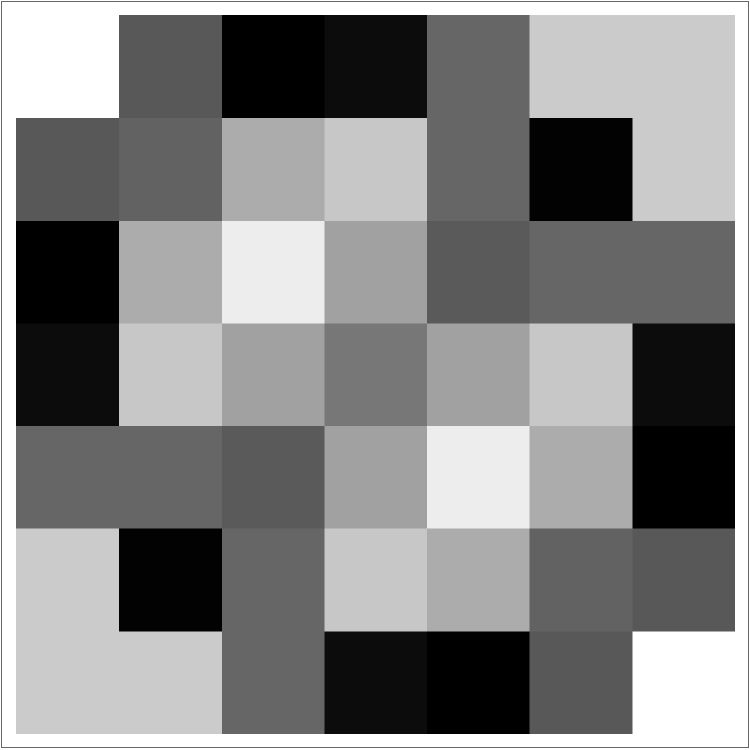}
         \caption{$\la = \ydiagram{3,3,1}$}
         \end{subfigure}
     \hfill\begin{subfigure}[b]{0.1\textwidth}
         \centering
         \includegraphics[frame, width=\textwidth]{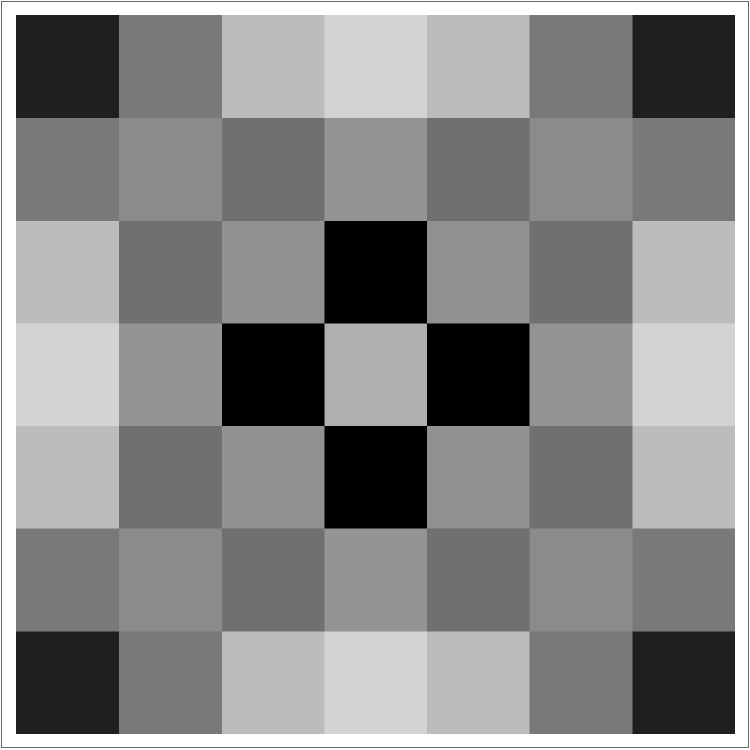}
         \caption{$\la = \ydiagram{4,1,1,1}$}
         \end{subfigure}
     \hfill

     \vspace{.5cm}

     \begin{subfigure}[t]{0.1\textwidth}
         \centering
         \includegraphics[frame, width=\textwidth]{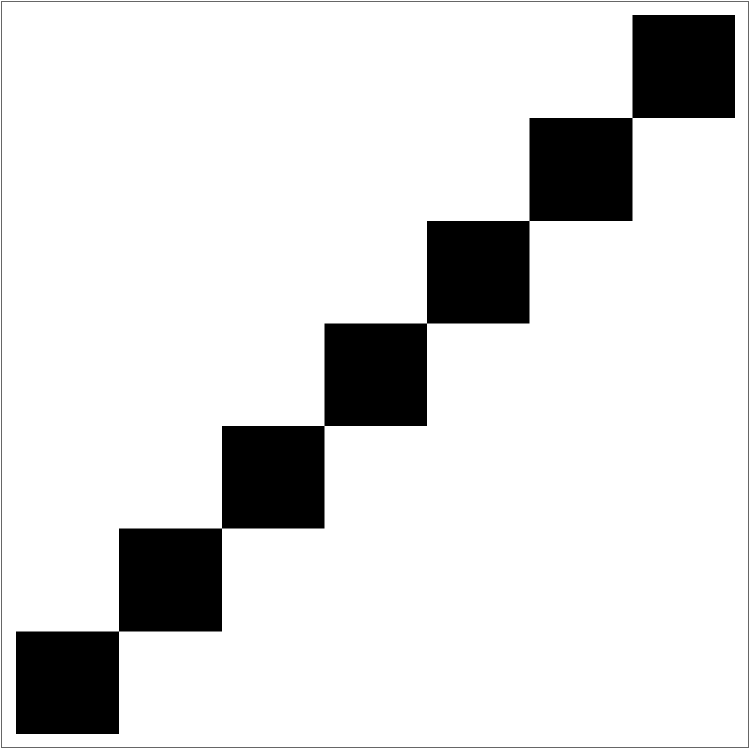}
         \caption{$\la = \ydiagram{1,1,1,1,1,1,1}$}
         \end{subfigure}
     \hfill\begin{subfigure}[t]{0.1\textwidth}
         \centering
         \includegraphics[frame, width=\textwidth]{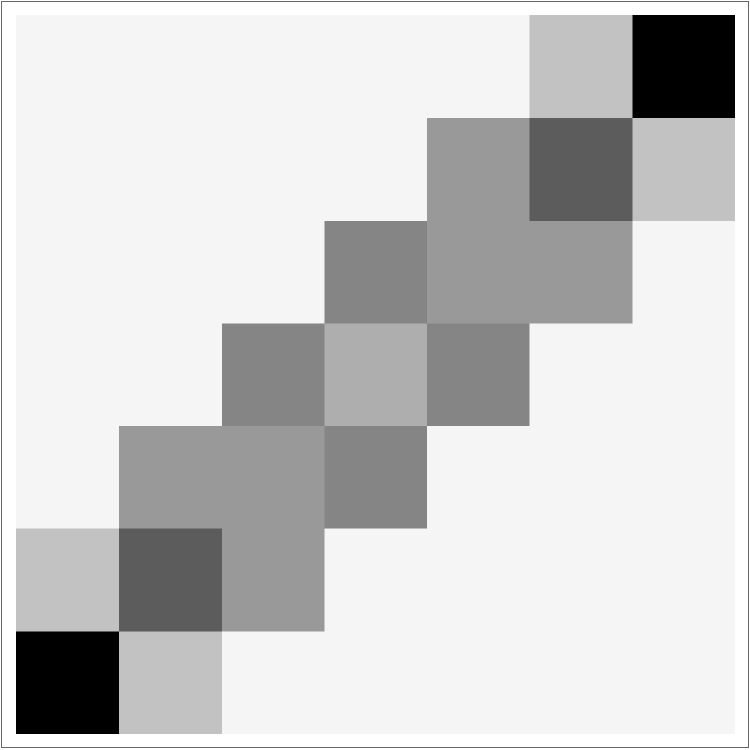}
         \caption{$\la = \ydiagram{2,1,1,1,1,1}$}
         \end{subfigure}
     \hfill\begin{subfigure}[t]{0.1\textwidth}
         \centering
         \includegraphics[frame, width=\textwidth]{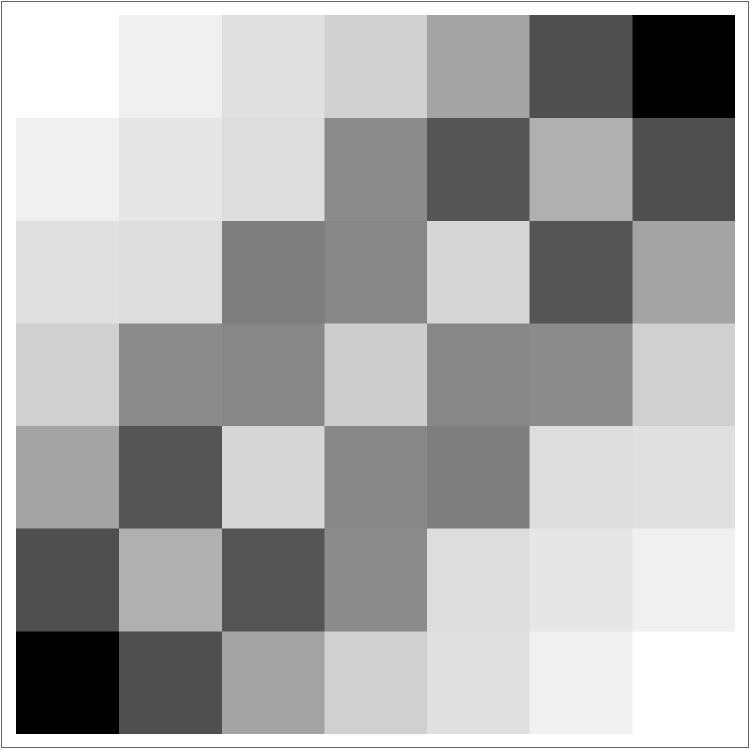}
         \caption{$\la = \ydiagram{2,2,1,1,1}$}
         \end{subfigure}
     \hfill\begin{subfigure}[t]{0.1\textwidth}
         \centering
         \includegraphics[frame, width=\textwidth]{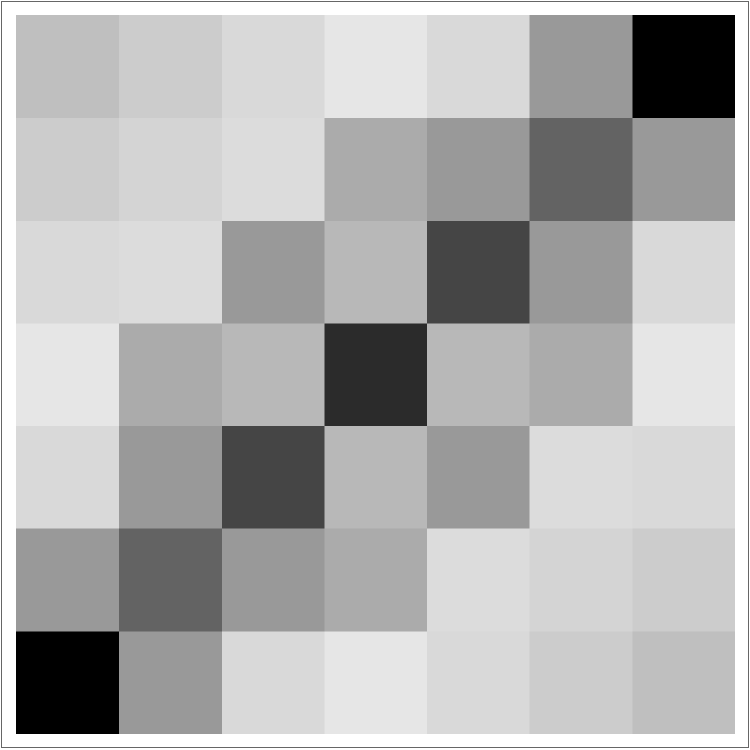}
         \caption{$\la = \ydiagram{3,1,1,1,1}$}
         \end{subfigure}
     \hfill\begin{subfigure}[t]{0.1\textwidth}
         \centering
         \includegraphics[frame, width=\textwidth]{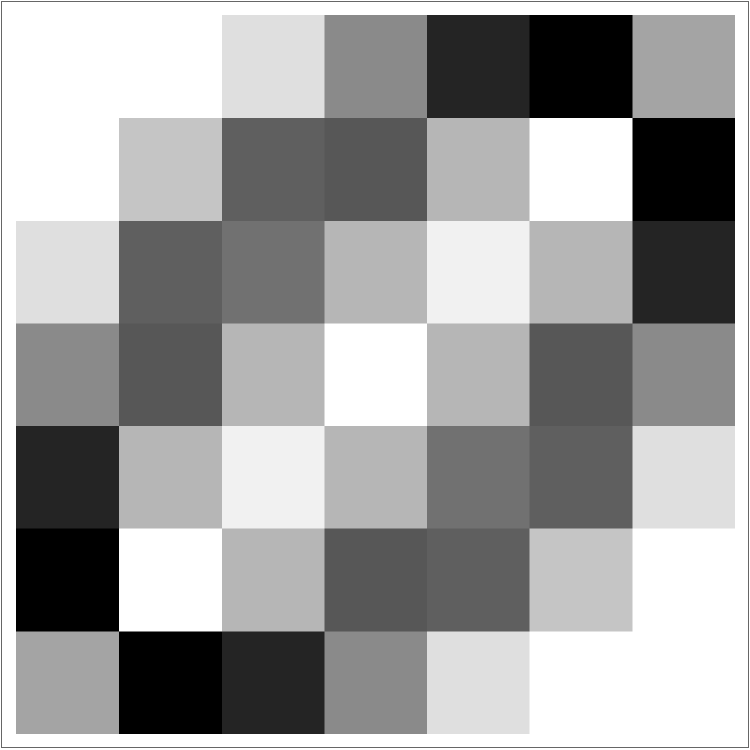}
         \caption{$\la = \ydiagram{2,2,2,1}$}
         \end{subfigure}
     \hfill\begin{subfigure}[t]{0.1\textwidth}
         \centering
         \includegraphics[frame, width=\textwidth]{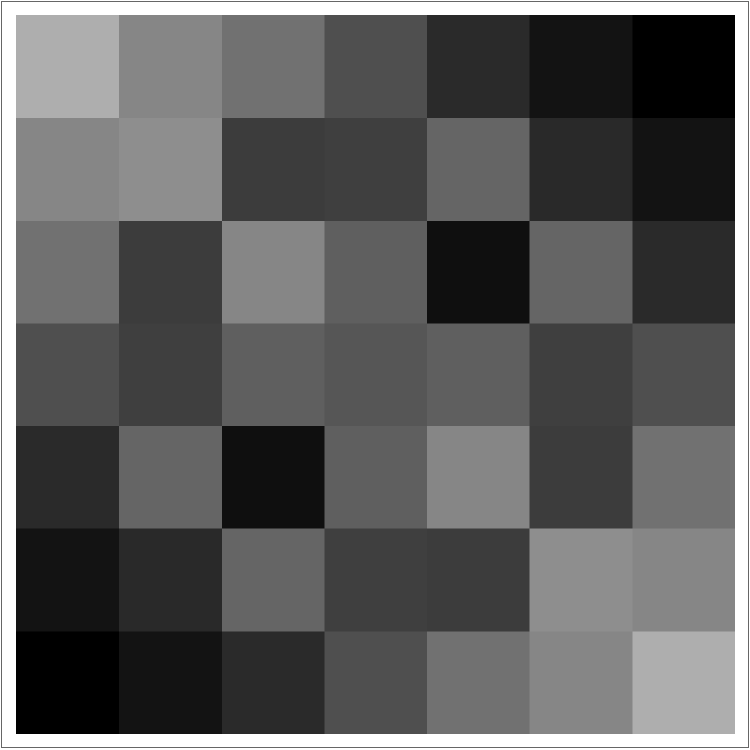}
         \caption{$\la = \ydiagram{3,2,1,1}$}
         \end{subfigure}
     \hfill\begin{subfigure}[t]{0.1\textwidth}
         \centering
         \includegraphics[frame, width=\textwidth]{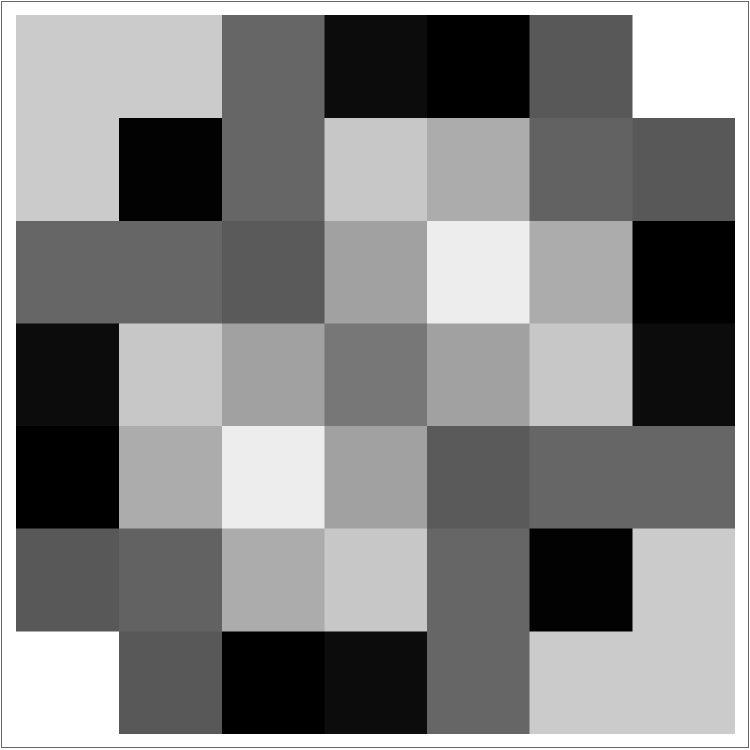}
         \caption{$\la = \ydiagram{3,2,2}$}
         \end{subfigure}
     \hfill\begin{subfigure}[t]{0.1\textwidth}
         \centering
         \caption{}
         \end{subfigure}
     \hfill
        \caption{Heat maps of the matrices $P^\la$, for all partitions $\la$ of $n = 7$.}
        \label{fig:heat maps intro}
\end{figure}

\subsection*{Main result: \texorpdfstring{$P^\la_{ij}$}{P lambda ij} formulas}

In this paper we focus on three families of shapes $\la$:
namely, \emph{hook} shapes $(\ell, 1^{m})$, \emph{two-row} shapes $(\ell, m)$, and \emph{rectangular} shapes $(\ell^m) \coloneqq (\ell, \ldots, \ell)$.
(Due to the inherent symmetries of the RS correspondence (Lemma~\ref{lemma:symmetries}), our solution for two-row shapes automatically yields the solution for two-\emph{column} shapes as well, by reflecting $P^\la$ horizontally or vertically.)
We chose these three families in part because they exhibit a natural and interesting contrast with respect to the complements of their first rows:
for hooks and two-row shapes, the complement is as thin as possible (a single column or row), whereas for rectangular shapes it is as thick as possible.
Immediately below, we preview our three main theorems, for hook, two-row, and rectangular shapes, respectively;
the reader will observe that the three $P^\la_{ij}$ formulas differ significantly from each other, despite our best efforts to render them in a uniform manner.
This is a reflection of the various symmetries in the RS correspondence when restricted to certain shapes.
Recall that $f^\mu$ denotes the number of standard Young tableaux of shape $\mu$.

\bigskip

\begin{theorem}[see Theorem~\ref{thm:hooks}]
\label{thm:hooks in intro}
  Let $\la = (\ell, \: 1^{n-\ell})$ be a hook shape.
  For all $1 \leq i,j \leq n$,
  \[
  P^\la_{ij} = \frac{X(\ell,i,j) + Y(\ell, i, j) +  Y(\bar{\ell}, i, \bar{\jmath}) + Y(\bar{\ell}, \bar{\imath}, j) + Y(\ell, \bar{\imath}, \bar{\jmath})}{\tallmset{\ell}{n-\ell}^2},
  \]
  where
  \begin{align*}
    \textstyle \shortmset{a}{b} & \coloneqq f^{(a, 1^b)} =  \textup{the multiset number } \textstyle \binom{a+b-1}{b}, \\[1ex]
    X(\ell,i,j) & \coloneqq \textstyle 
    \tallmset{K}{i - K} \tallmset{K}{j - K} \tallmset{\ell-K+1}{n - \ell + K - i} \tallmset{\ell - K + 1}{n - \ell + K - j},
    \\[1ex]
    Y(\ell,i,j) & \coloneqq \sum_{1 < k < K} \textstyle \tallmset{k-1}{i-k} \tallmset{k-1}{j-k} \tallmset{\ell - k+1}{n - \ell + k - i} \tallmset{\ell - k+1}{n - \ell + k - j}, \\[1ex] 
    K \coloneqq K(\ell,i,j) & \coloneqq (i+j+\ell-n)/2, \quad \textup{and} \quad \bar{\imath} \coloneqq n+1-i.
  \end{align*}.
\end{theorem}

\begin{theorem}[see Theorem~\ref{thm:two-row}]
\label{thm:two-row in intro}
  Let $\la = (\ell, \: n-\ell)$ be a two-row shape.
  For all $1 \leq i \leq j \leq n$,
  \[
    P^\la_{ij} = P^\la_{ji} = \frac{V(\ell,i,j) + W(\ell,i,j)}{\tallballot{\ell}{n-\ell}^2},
  \]
  where
  \begin{align*}
    \textstyle \shortballot{a}{b} & \coloneqq f^{(a,b)} = \textup{the ballot number } \textstyle \binom{a+b}{b} - \binom{a+b}{b-1} \textup{ for $a \geq b$}, \\[1ex]
    \textstyle \shortballot{a \rightarrow c}{b} & \coloneqq f^{(a+c,b)/(c)} = \textup{the generalized ballot number } \textstyle \binom{a+b}{b} - \binom{a+b}{b-c-1} \textup{ for $a+c \geq b$}, \\[1ex]
    V(\ell,i,j) & \coloneqq 
        \sum_{k=1}^{\ell} \textstyle \tallballot{k-1}{i-k} \tallballot{k-1}{j-k} \tallballot{\ell - k \: \rightarrow j - i}{n - \ell + k - i} \tallballot{\ell - k}{n - \ell + k - j} , \\
    W(\ell,i,j) & \coloneqq \sum_{\mathclap{(p,q) \in \mathscr{K}}} \textstyle \; \tallballot{p - 1}{i - p} \tallballot{q}{j - q - 1} 
    \left( \!\!
      \begin{array}{l}
      \phantom{-} \tallballot{\ell - p \: \rightarrow \: j-i + p - q}{n - \ell + p - i} \tallballot{\ell - q \: \rightarrow \: q - p}{n - \ell + q - j} \\[2ex]
      - \tallballot{\ell - p \: \rightarrow \: j-i + p - q - 1}{n - \ell + p - i} \tallballot{\ell - q \: \rightarrow \: q - p - 1}{n - \ell + q - j}
    \end{array} 
  \right), \\[1ex]
    \mathscr{K} & \coloneqq \Big\{ (p,q) \in [\ell] \times [\ell] : p \leq i \textup{ and } p + q \geq j \textup{ and } 0 \leq q-p < j-i \Big\}.
  \end{align*}
\end{theorem}

\begin{theorem}[see Theorem~\ref{thm:rectangles}]
  \label{thm:rectangles in intro}
  Let $\la = (\ell^m)$ be a rectangular shape, with $n \coloneqq \ell m$.
  For all $1 \leq i,j \leq n$,
  \[
    P^\la_{ij} = \frac{Z(\ell,m,i,j)}{(f^\la)^2},
  \]
  where
  \begin{align*}
    Z(\ell,m,i,j) & \coloneqq
    \sum_{\substack{1 \leq r \leq m, \\ 1 \leq k \leq \ell} \phantom{,}}
    \Bigg[
      \sum_{\substack{\mu,\nu, \\ \pi, \rho \phantom{,}}}
      f^{(k + \mu; k-1; \nu)} f^{(k + \pi; k-1; \rho)} f^{(\ell - \overscriptleftarrow{\nu}; \ell-k; \ell - k - \overscriptleftarrow{\mu})} f^{ ( \ell - \overscriptleftarrow{\rho}; \ell-k; \ell - k - \overscriptleftarrow{\pi})}
    \Bigg], \\[1ex]
    (\mu, \nu) &\in \operatorname{Par}\big((r-1) \times (\ell-k)\big) \times \operatorname{Par}\big((m-r) \times (k-1) \big) \textup{ such that } |\mu| + |\nu| = i - kr, \\ 
    (\pi,\rho) & \in \operatorname{Par}\big( (m-r) \times (\ell-k)\big) \times \operatorname{Par}\big( (r-1) \times (k-1) \big) \textup{ such that } 
    |\pi| + |\rho| = j - k(m+1-r), \\ 
    \operatorname{Par}(a \times b) & \coloneqq \Big\{ \textup{partitions whose Young diagram fits inside an $a \times b$ rectangle} \Big\}, \\
    \ell - \overscriptleftarrow{\nu} & \coloneqq (\ell- \nu_{m-r}, \ldots, \ell - \nu_2, \ell-\nu_1), \textup{ and semicolons denote concatenation}.
\end{align*}
\end{theorem}

Certain features of the heat maps in Figures~\ref{fig:n15 hooks}--\ref{fig:rectangles} are explained by a glance at the theorems above.
For instance, in Figure~\ref{fig:n15 hooks}, the central ``checkerboard'' phenomenon is explained by the fact that in Theorem~\ref{thm:hooks in intro} the term $X(\ell,i,j)$ vanishes when $K$ is a half-integer. 
Likewise, in Figure~\ref{fig:n15 two-row}, the ``double slit'' phenomenon is explained by the fact that in Theorem~\ref{thm:two-row in intro} the index set $\mathscr{K}$ is empty along the diagonal $i=j$, and is generally smaller when $i$ or $j-i$ or $\ell-j$ is small.
(In Lemma~\ref{lemma:size SV SW} we express the functions $V$ and $W$ in a way that makes the $i \leftrightarrow j$ symmetry more apparent.)
We have verified Theorems~\ref{thm:hooks in intro}--\ref{thm:rectangles in intro} by computer up to $n=13$ (which seems to be the threshold at which the runtime on the author's personal laptop passes from hours into days).

\subsection*{Method of proof}

Sections~\ref{sec:hooks}, \ref{sec:two-row}, and~\ref{sec:rectangles}, devoted to hook and two-row and rectangular shapes, respectively, share a parallel four-part development:

\begin{enumerate}
  \item Lemma~\ref{lemma:size Slij hooks} (resp.,~\ref{lemma:Slij count two-row} and~\ref{lemma:RS rectangles}) is the theoretical heart of the argument.  
  We decompose the set of $\sigma$'s of shape $\la$ according to the relative position of $\sigma(i)=j$:

  \begin{itemize} 

  \item For hook shapes, $\sigma$ forms an ``X'' pattern, meaning the union of a maximal increasing subsequence and a maximal decreasing subsequence.
  We decompose the $\sigma$'s according to whether $\sigma(i)$ lies in the ``crossroads'' of the ``X,'' meaning the intersection of the two maximal subsequences.
  If so, then $\sigma$ contributes 1 to the $X$ formula in Theorem~\ref{thm:hooks in intro};
  if not, then $\sigma$ contributes 1 to one of the $Y$ formulas, depending on which of the four legs of the ``X'' contains $\sigma(i)$.
  
  \item For two-row shapes, $\sigma$ is the disjoint union of a maximal increasing sequence and several descent pairs, determined by Viennot's ``light and shadow'' construction (see Algorithm~\ref{alg:Viennot}).
  We decompose the $\sigma$'s accordingly:
  if $\sigma(i)$ lies in the maximal increasing subsequence, then $\sigma$ contributes 1 to the $V$ formula in Theorem~\ref{thm:two-row in intro}, and otherwise to the $W$ formula.

  \item For rectangular shapes, which are especially amenable to the RS correspondence, there is no advantage in decomposing the $\sigma$'s into subtypes;
  rather, we immediately characterize all $\sigma$'s in terms of the positions of the entries $i$ and $j$ in the corresponding RS tableaux.

  \end{itemize}

  \item Lemma~\ref{lemma:size SX SY} (resp.,~\ref{lemma:size SV SW} and~\ref{lemma:size S lambda rectangles}) is the technical heart of the argument.
  For each subtype $\star$ of $\sigma$'s described above, we exhibit a bijection $\Phi_\star$ of the following form:
  \[
    \{ \text{$\sigma$'s of type $\star$} \} \xrightarrow{\;\;\RS\;\;} \left\{ \begin{array}{c} \text{certain tableau pairs } \\ \text{$(\se{P}, \se{Q} )$ of shape $\la$} \end{array} \right\} \xrightarrow{\;\; \Phi_\star \;\;} \left\{ \begin{array}{c} \text{certain tableau} \\ \text{quadruples $(\se{A}, \se{B}, \se{C}, \se{D})$} \end{array} \right\},
  \]
  where $\se{A}$ (resp., $\se{B}$) is the subtableau of $\se{P}$ (resp., $\se{Q}$) with entries up to $i$ (resp., $j$), and where $\se{C}$ (resp., $\se{D}$) is formed by rearranging and relabeling the remainder of $\se{P}$ (resp., $\se{Q})$.
  This explains all the fourfold products of $f^\mu$'s appearing in Theorems~\ref{thm:hooks in intro}--\ref{thm:rectangles in intro}.
  The five bijections $\Phi_{\mathrm{X}}$, $\Phi_{\mathrm{Y}}$, $\Phi_{\mathrm{V}}$, $\Phi_{\mathrm{W}}$, and $\Phi_{\mathrm{Z}}$ are defined in detail in~\eqref{SX bijection}, \eqref{SY bijection}, \eqref{SV bijection},~\eqref{SW bijection}, and~\eqref{SZ bijection}, respectively.

  \item Theorem~\ref{thm:hooks} (resp.,~\ref{thm:two-row} and~\ref{thm:rectangles}) is the main result of the paper.
  We combine the two previous lemmas to give the desired formula for $P^\la_{ij}$.
  
  \item Remark~\ref{rem:visualize hooks} (resp.,~\ref{rem:visualize two-row} and ~\ref{rem:rectangles}) presents a heuristic for visualizing the contribution toward $P^\la$ coming from each term in our main formulas.
  (The reader can quickly verify these observations by inspecting the main theorems to determine the $(i,j)$'s for which each term vanishes.)
  In particular, a two-row shape $(\ell, n-\ell)$ forces $\sigma$ to have bandwidth $\ell$, and a rectangular shape $(\ell^m)$ forces $\sigma$ to have bandwidth $\ell(m-1)$, whereas hook shapes impose no limitations whatsoever on bandwidth.
\end{enumerate}

Despite the uniform approach outlined above, the three families each required a surprising degree of special handling.
For instance, hooks (being conjugate to other hooks) admit symmetries via evacuation operators which two-row shapes do not have.
Moreover, for hooks and rectangular shapes it suffices to determine just the locations of $i$ and $j$ in the tableaux $\se{P}$ and $\se{Q}$, but for two-row shapes one must also verify certain inequalities among the entries $>i$ and $>j$;
this necessitates the subtraction in the $W$ formula, and for this reason the two-row proof is by far the most delicate of the three.
Having worked through the various shape-dependent details in this paper, we are increasingly curious whether Problem~\ref{problem:Pij} can admit anything close to a uniform solution while maintaining sufficient explicitness.

\subsection*{Fixed points and Wallis integrals}

Our formulas above led to an unexpected result concerning fixed point asymptotics.
Given a permutation $\sigma \in \fS_n$, a key object of interest is its \emph{fixed point set}
\[
  \operatorname{Fix}(\sigma) \coloneqq \Big\{ i : \sigma(i) = i \Big\}.
\]
Note that the trace of our matrix $P^\la$ equals the expected number of fixed points in a permutation of shape~$\la$.
Hence this expected value can be obtained directly from Theorems~\ref{thm:hooks in intro}--\ref{thm:rectangles in intro} simply by taking $\sum_i P^\la_{ii}$.
A more interesting problem is to determine an asymptotic for the \emph{proportion} of fixed points:
in particular, for either hook or two-row shapes $\la$, one can fix the complement of the first row of $\la$ at some size $m$, and then ask for the expected proportion of fixed points as the length $\ell$ of the first row approaches infinity.
(The question does not make sense for rectangular shapes, of course, since letting only the first row grow destroys the rectangular shape.)

Our trace formulas from Theorems~\ref{thm:hooks in intro} and~\ref{thm:two-row in intro} yield the following curious answer to this problem --- the same answer, in fact, for both hook and two-row shapes.
By way of background, the family of \emph{Wallis integrals} was introduced in 1656 by John Wallis~\cite{Wallis} in his development of an infinite product expression for~$\pi$.
In particular, the $k$th Wallis integral is given by
\begin{equation}
  \label{Wallis}
  W_k \coloneqq \int_0^{\pi/2} \sin^k \! x \: dx = \int_0^{\pi/2} \cos^k \! x \: dx, \quad k = 0, 1, 2, \ldots.
\end{equation}

\begin{theorem}
  \label{thm:fix in intro}

  Let $\la = (\ell, \: 1^m)$ or $(\ell, \: m)$, with $n \coloneqq \ell + m$, and let $\sigma$ with shape $\lambda$ be chosen uniformly at random.
  For $m \geq 0$ fixed,
    \[
      \lim_{\ell \rightarrow \infty} \mathbb{E} \! \left[ \frac{\left|\operatorname{Fix}(\sigma) \right|}{n} \right] = 
      \frac{(2m)!!}{(2m+1)!!} = W_{2m+1},
    \]
    where $W_{2m+1}$ is the $(2m+1)$th Wallis integral in~\eqref{Wallis}.
\end{theorem}

We devote Section~\ref{sec:fixed} to the proof of Theorem~\ref{thm:fix in intro}.
It is striking that the result applies to both hooks and two-row shapes, since (below the first row) these two cases are opposite extremes.
This raises the question of whether the result in Theorem~\ref{thm:fix in intro} is in fact shape-independent.

\begin{problem}
  Fix a positive integer $m$, and let $\sigma \in \fS_n$ be chosen uniformly at random from among those permutations whose longest increasing sequence has length $n - m$.
  Is it true that as $n \rightarrow \infty$, the expected proportion of fixed points in $\sigma$ equals $W_{2m+1}$?
\end{problem}

\begin{problem}
  Give a conceptual explanation (perhaps in terms of Beta functions?) for the appearance of the Wallis integral in Theorem~\ref{thm:fix in intro}.
\end{problem}

\subsection*{References and connections}

The RS correspondence is named for Robinson~\cite{Robinson} and later Schensted~\cite{Schensted}, before the landmark generalization by Knuth~\cite{Knuth} produced the RSK correspondence between biwords over $\{1, \ldots, n\}$ and pairs of \emph{semi}standard Young tableaux.
For background, we refer the reader to~\cite{Fulton}*{Ch.~4 and App.~A} or~\cite{Krattenthaler} or~\cite{Sagan}*{Ch.~3}.
Since its inception, many variations and analogues of RSK have been introduced into the literature, any of which could be taken up as a subject for Problem~\ref{problem:Pij}.
In fact, the paper~\cite{EK} can be viewed retrospectively as the solution to the version of Problem~\ref{problem:Pij} where $\fS_n$ is replaced by the set of biwords of length $d$ and where $\la = (d)$ is a one-row shape.

Also similar in spirit to the present paper is the work by Ayyer--Banerjee~\cite{AyyerBanerjee}, who studied the number of inversions in permutations with a prescribed shape.
In~\cite{EHMS} we studied the related question of which shapes can be attained by permutations of a given cycle type.
On the topic of longest increasing subsequences and related asymptotics and probability, see the \emph{Bulletin} article~\cite{AD} by Aldous--Diaconis, and the references therein.
Closely related to our Theorem~\ref{thm:fix in intro} is Diaconis--Fulman--Guralnick~\cite{DFG}, who generalized Montmort's original theorem~\cite{Montmort} giving the limiting distribution for the number of fixed points in  $\sigma \in \fS_n$ as $n \rightarrow \infty$.
Goldstein--Moews~\cite{GoldsteinMoews} solved the same problem in the setting of exchange permutations.

\subsection*{Acknowledgments}

We thank Jeb Willenbring for an early conversation regarding sums of permutation matrices which ultimately led to Problem~\ref{problem:Pij}.

\ytableausetup{smalltableaux}

\section{The RS correspondence and the matrix \texorpdfstring{$P^\la$}{P lambda}}
\label{sec:RS}

\subsection*{Partitions, Young diagrams, and tableaux}

A \emph{partition} is a finite weakly decreasing sequence $\la = (\la_1, \ldots, \la_t)$ of positive integers.
If $\sum_i \la_i = n$, we say that $\la$ is a partition of $n$, written as $\la \vdash n$.
The empty partition is denoted by 0.
We will ignore trailing zeros if they occur, meaning that $(\la_1, \ldots, \la_t, 0, \ldots, 0)$ is considered the same as the partition $(\la_1, \ldots, \la_t)$.
One typically identifies $\la$ with its \emph{Young diagram}, namely the left-justified arrangement of $n$ boxes such that the $i$th row from the top contains $\la_i$ boxes.
The \emph{conjugate} of $\la$, denoted by $\la'$, is the partition whose Young diagram is obtained by transposing that of $\la$ about the main diagonal.
For example:
\[
  \la = (3,2,1,1) = \ydiagram{3,2,1,1} \quad \leadsto \quad \la' = (4,2,1) = \ydiagram{4,2,1}.
\]
It is common to describe partitions in terms of the appearance of their Young diagrams.
Of special interest in the present paper, a partition $\la$ is said to be a \emph{hook} shape if $\la = (\ell, 1^{n-\ell})$ where $1 \leq \ell \leq n$; the shorthand $1^{n - \ell}$ denotes $1, \ldots, 1$ repeated $n - \ell$ times.
We view the single row $(n)$ and the single column $(1^n)$ as degenerate hook shapes.
Similarly, we say that $\la$ is a \emph{two-row} shape if it takes the form $\la = (\ell, n - \ell)$ where $n/2 \leq \ell \leq n$; we view the single row $(n)$ as a degenerate two-row shape.
Finally, $\la$ is said to be a \emph{rectangular} shape if $\la = (\ell^m) \coloneqq (\ell, \ldots, \ell)$ for positive integers $\ell$ and $m$.

A \emph{standard Young tableau of shape $\la \vdash n$} is obtained by assigning each element of $[n] \coloneqq \{1, \ldots, n\}$ to a unique box in the Young diagram of $\la$, in such a way that the entries increase from left to right, and from top to bottom.
(These conditions are referred to as \emph{row strictness} and \emph{column strictness}, respectively.)
We write $\SYT(\la)$ to denote the set of all standard Young tableaux of shape $\la$.
For example:
\begin{equation}
  \label{tableau example}
  \la = (3,2,1,1) \quad \leadsto \quad \ytableaushort{124,37,5,6} \in \SYT(\la).
\end{equation}
Note that if $\la=0$ is the empty partition, then $\SYT(0)$ contains a single element, namely the empty tableau.
On the other hand, if a tuple $\alpha$ is not a partition --- that is, if its coordinates are not weakly decreasing nonnegative integers --- then $\SYT(\alpha) = \varnothing$.
In this paper, we denote tableaux by sans serif capitals in order to distinguish them from matrices.
Given a tableau $\se{T}$, we write $\se{T}_{\!ij}$ for the entry in row $i$ and column $j$;
for example, if $\se{T}$ is the tableau in~\eqref{tableau example}, then $\se{T}_{1,3} = 4$ and $\se{T}_{4,1} = 6$.

More generally, let $\la$ and $\mu$ be two partitions such that the Young diagram of $\la$ contains that of $\mu$, when both are aligned at their upper-left corner.
Then $\la/\mu$ denotes the \emph{skew shape} obtained by removing from $\la$ those boxes belonging to $\mu$.
Just as above, a standard Young tableau of skew shape $\la/\mu$ is obtained by any row-strict and column-strict filling with entries $1, \ldots, |\la|-|\mu|$.
We write $\SYT(\la/\mu)$ for the set of all standard Young tableaux of skew shape $\la/\mu$.
In this paper we primarily consider skew two-row shapes of the form $(d,b)/(c)$, where $d \geq b,c$.
It will be useful to rewrite such a shape in the form $(a+c,b) / (c)$, to emphasize the fact that this shape consists of two rows of lengths $a$ and $b$, where the top row is offset by $c$ boxes; see Lemma~\ref{lemma:f special} below.
As an example where $a=3$ and $b=4$ and $c=2$,
\[
  \begin{ytableau}
  \none & \none & 2 & 4 & 7 \\ 
  1 & 3 & 5 & 6
  \end{ytableau} 
  \; \in \SYT( (5,4) / (2)) = \SYT((3+2, \: 4) / (2)).
\]

The cardinalities of tableau sets are denoted by
\begin{equation}
  \label{f}
  f^\la \coloneqq |\SYT(\la)|, \qquad f^{\la/\mu} \coloneqq |\SYT(\la/\mu)|.
\end{equation}
The number $f^\la$ is famously given by the hook-length formula~\cite{Fulton}*{p.~53}:
\begin{equation}
  \label{hook length formula}
  f^\la = \frac{n!}{\prod_{(i,j) \in \la}  h_\la(i,j)},
\end{equation}
where $(i,j)$ is viewed as the box in row $i$ and column $j$ of the Young diagram of $\la$, and where its hook length $h_\la(i,j)$ equals the number of boxes to the right of $(i,j)$ in its row, plus the number of boxes below $(i,j)$ in its column.
Since in this paper we largely focus on hook and two-row shapes, we will record those specializations in Lemma~\ref{lemma:f special} below;
these numbers will be the building blocks of the formulas in our main results.
As is typical in a purely combinatorial setting, we define the binomial coefficients
\begin{equation}
  \label{binomial definition}
  \textstyle \binom{a}{b} \coloneqq \frac{a!}{b! (a-b)!} \text{ if $a \geq b \geq 0$ are integers, otherwise $0$}.
\end{equation}
We will use the shorthand
\begin{align}
  \textstyle \shortmset{a}{b} & \coloneqq \textstyle \binom{a+b-1}{b}, \label{multiset definition}\\[1ex]
  \textstyle \shortballot{a \rightarrow c}{b} &\coloneqq \textstyle \binom{a+b}{b} - \binom{a+b}{b-c-1} \text{ if $a,b,c \geq 0$ with $a+c \geq b$, otherwise $0$}, \label{skew ballot definition} \\[1ex]
  \textstyle \shortballot{a}{b} &\coloneqq \textstyle \shortballot{a \rightarrow 0}{b}. \label{ballot definition}
\end{align}
The symbol $\shortmset{a}{b}$ agrees with the standard notation for the \emph{multiset numbers}~\cite{Stanley}*{p.~26}, so called because they give the number of $b$-element multisets that can be formed by taking elements from an $a$-element set.
The reason we stipulate in~\eqref{binomial definition} that $a$ and $b$ are integers is that one of our results (Lemma~\ref{lemma:size SX SY}) involves expressions $\shortmset{a}{b}$ in which $a$ or $b$ may be a half-integer, in which case $\shortmset{a}{b} = 0$.
The numbers $\shortballot{a}{b}$ are known as \emph{ballot numbers}~\cite{Riordan}*{p.~130}, since they count the number of \emph{ballot paths} from the origin to the point $(a,b)$; see also~\cite{Carlitz}.
The ballot numbers are sometimes denoted by $B(a,b)$, but in this paper we write $\shortballot{a}{b}$ to parallel $\shortmset{a}{b}$ and to conserve space in equations.

\begin{lemma}
  \label{lemma:f special}
  In the notation of~\eqref{f}--\eqref{ballot definition}, we have
  \[
    \textstyle f^{(a, 1^b)} = \shortmset{a}{b}, \qquad f^{(a,b)} = \shortballot{a}{b}, \qquad f^{(a+c, \: b)/(c)} = \shortballot{a \rightarrow c}{b}.
  \]
\end{lemma}

\begin{proof}
  For a hook shape, each tableau in $\SYT(a, 1^b)$ is uniquely determined by the entries in its first column below the 1.
  Thus $f^{(a, 1^b)}$ equals the number of ways to choose these $b$ entries from $\{2, \ldots, a+b\}$, namely $\binom{a+b-1}{b} = \shortmset{a}{b}$.

  Next, it is well known~\cite{KrattenthalerBook}*{Thm.~10.3.1} that $\shortballot{a \rightarrow c}{b}$ equals the number of lattice paths from $(0,-c)$ to $(a, b-c)$ inside the region $y \leq x$.
  We exhibit the following bijection between the set of all such paths and the set $\SYT((a+c, \: b)/(c))$.
  Given a path, label its steps $1,2,\ldots,a+b$, then put the labels corresponding to horizontal (resp., vertical) steps in the top (resp., bottom) row of the Young diagram of $(a+c, \: b)/(c)$;
  each column in the resulting tableau is increasing due to the $y \leq x$ condition on the path, and the construction is clearly invertible.
  Hence $f^{((a+c, \: b)/(c)} = \shortballot{a \rightarrow c}{b}$, and the special case $c=0$ yields $f^{(a,b)} = \shortballot{a}{b}$.
\end{proof}

The key mechanism in our bijective proofs will be to decompose a tableau into two parts determined by the size of their entries.
For $\se{T} \in \SYT(\la)$ with $1 \leq i \leq n$, we write
\begin{equation}
  \label{T| def}
  \se{T}|_{ \leq i} \coloneqq \text{the result obtained from $\se{T}$ by deleting all boxes except those with entries $\leq i$}.
\end{equation}
We define $\se{T}|_{<i}$, $\se{T}|_{> i}$, etc., in the same way.
Note that $\se{T}|_{\leq i} \in \SYT(\mu)$, where $\mu$ is the shape determined by those boxes of $\se{T}$ with entries $\leq i$.
Moreover, the result obtained by subtracting $i$ from all entries of $\se{T}|_{>i}$ is an element of $\SYT(\la/\mu)$.

\subsection*{The RS correspondence}

Throughout the paper, $n$ is a positive integer, and $[n] \coloneqq \{1, \ldots, n\}$.
The symmetric group on $n$ letters is denoted by
\[
  \fS_n \coloneqq \Big\{ \text{permutations $\sigma$ of $[n]$} \Big\},
\]
with the group operation given by composition.
We typically write down an element $\sigma \in \fS_n$ using the one-line notation $\sigma = (\sigma(1), \ldots, \sigma(n))$.
Given $\sigma \in \fS_n$, its $n \times n$ permutation matrix $M_\sigma$ is given by
\begin{equation}
  \label{M sigma ij}
  (M_\sigma)_{ij} =
  \begin{cases}
    1 & \text{if } \sigma(i) = j,\\ 
    0 & \text{otherwise}.
  \end{cases}
\end{equation}
Recall that the \emph{support} of a matrix $M$ is the set $\operatorname{supp} M \coloneqq \{(i,j) : M_{ij} \neq 0 \}$.
Resuming the example from~\eqref{example intro}, we have
\begin{align*}
  \sigma &= (2,7,4,6,3,1,5), \\[1ex]
  M_\sigma &= 
  \left[\begin{smallmatrix}
    \textcolor{lightgray}{0} & 1 & \textcolor{lightgray}{0} & \textcolor{lightgray}{0} & \textcolor{lightgray}{0} & \textcolor{lightgray}{0} & \textcolor{lightgray}{0} \\
    \textcolor{lightgray}{0} & \textcolor{lightgray}{0} & \textcolor{lightgray}{0} & \textcolor{lightgray}{0} & \textcolor{lightgray}{0} & \textcolor{lightgray}{0} & 1 \\
    \textcolor{lightgray}{0} & \textcolor{lightgray}{0} & \textcolor{lightgray}{0} & 1 & \textcolor{lightgray}{0} & \textcolor{lightgray}{0} & \textcolor{lightgray}{0} \\
    \textcolor{lightgray}{0} & \textcolor{lightgray}{0} & \textcolor{lightgray}{0} & \textcolor{lightgray}{0} & \textcolor{lightgray}{0} & 1 & \textcolor{lightgray}{0} \\ 
    \textcolor{lightgray}{0} & \textcolor{lightgray}{0} & 1 & \textcolor{lightgray}{0} & \textcolor{lightgray}{0} & \textcolor{lightgray}{0} & \textcolor{lightgray}{0} \\
    1 & \textcolor{lightgray}{0} & \textcolor{lightgray}{0} & \textcolor{lightgray}{0} & \textcolor{lightgray}{0} & \textcolor{lightgray}{0} & \textcolor{lightgray}{0} \\ 
    \textcolor{lightgray}{0} & \textcolor{lightgray}{0} & \textcolor{lightgray}{0} & \textcolor{lightgray}{0} & 1 & \textcolor{lightgray}{0} & \textcolor{lightgray}{0}
    \end{smallmatrix}\right],
\end{align*}
where $\operatorname{supp} M_\sigma$ is depicted by the first diagram in Figure~\ref{fig: ex of RSK}.

The \emph{Robinson--Schensted~\textup{(}RS\textup{)} correspondence} is a bijection 
\begin{equation}
\label{RSK bijection}
\RS : \fS_n \longrightarrow \coprod_{\lambda \vdash n} \SYT(\lambda) \times \SYT(\lambda),
\end{equation}
to be defined explicitly below in Algorithm~\ref{alg:Viennot}.
In particular, we give the geometric ``light and shadow'' construction due to Viennot~\cite{Viennot};
the RS correspondence is also commonly defined via Schensted's row insertion algorithm~\cite{Schensted}, Sch\"utzenberger's jeu de taquin~\cite{Schutzenberger}, or Fomin's growth diagrams~\cites{Fomin86,Fomin95}. 
In the following algorithm, we follow the exposition by Sagan~\cite{Sagan}*{\S3.6}, with the exception that we reverse the roles of the tableaux $\se{P}$ and $\se{Q}$.
(Thus our $\se{P}$ is the ``recording tableau'' and our $\se{Q}$ the ``insertion tableau'';
we do this so that $\sigma(i) = j$ corresponds to the entry $i$ in $\se{P}$ and the entry $j$ in $\se{Q}$.
In light of Lemma~\ref{lemma:symmetries}, this change in convention has no effect on our main result.)
See Figure \ref{fig: ex of RSK} for an illustrated example of the following algorithm.

\noindent \begin{minipage}{\linewidth}

\begin{algorithm}[Viennot's construction of the RS correspondence~\cite{Viennot}]\
\label{alg:Viennot}

\medskip

\hspace{3ex} \textbf{Input:} $\sigma \in \fS_n$.

\hspace{3ex} \textbf{Output:} The standard Young tableaux $\se{P}$ and $\se{Q}$ such that $\RS(\sigma) = (\se{P}, \se{Q})$.

\bigskip

\hspace{3ex} View $[n] \times [n]$ as a grid with matrix coordinates. Imagine a light source in the northwest corner.

\hspace{3ex} Thus each point casts a \emph{shadow} comprising the quarter-plane to its southeast.

\bigskip

\begin{enumerate}
\setlength{\itemsep}{1ex}

  \item Initialize $r=1$, and set $\sigma^{(1)} \coloneqq \operatorname{supp} M_\sigma$.
  (Each $\sigma^{(r)}$ is called the $r$th \emph{skeleton} of $\sigma$.)

  \item \label{L1} Define $L_1^{(r)}$ to be the set of all points in $\sigma^{(r)}$ that do not lie in the shadow of another point.
  The \emph{shadow line} of $L_1^{(r)}$ is the boundary of the combined shadow of its elements.

  \item Recursively define $L_{k+1}^{(r)}$ by applying step~\ref{L1} to $\sigma^{(r)} \setminus (L_1^{(r)} \sqcup \cdots \sqcup L_{k}^{(r)})$ for all $k \geq 1$ until this complement is empty.
  The resulting collection of shadow lines is called the \emph{shadow diagram} of $\sigma^{(r)}$.

  \item \label{fill rows} Fill in row $r$ of $\se{P}$ and $\se{Q}$ as follows:
  \[
    \se{P}_{r,k} = \min\{ i : (i,j) \in L_k^{(r)} \}, \qquad \se{Q}_{r,k} = \min\{ j : (i,j) \in L_k^{(r)} \}.
  \]
  
  \item
  Increment $r$ by 1, and recursively define $\sigma^{(r+1)}$ to be the set of points in $[n] \times [n]$ corresponding to the \scalebox{2}{$\lrcorner$}-shapes in the shadow diagram of $\sigma^{(r)}$.
  Repeat steps~\ref{L1}--\ref{fill rows}.
  The algorithm terminates when $\sigma^{(r)} = \varnothing$.
\end{enumerate}

\end{algorithm}

\end{minipage}

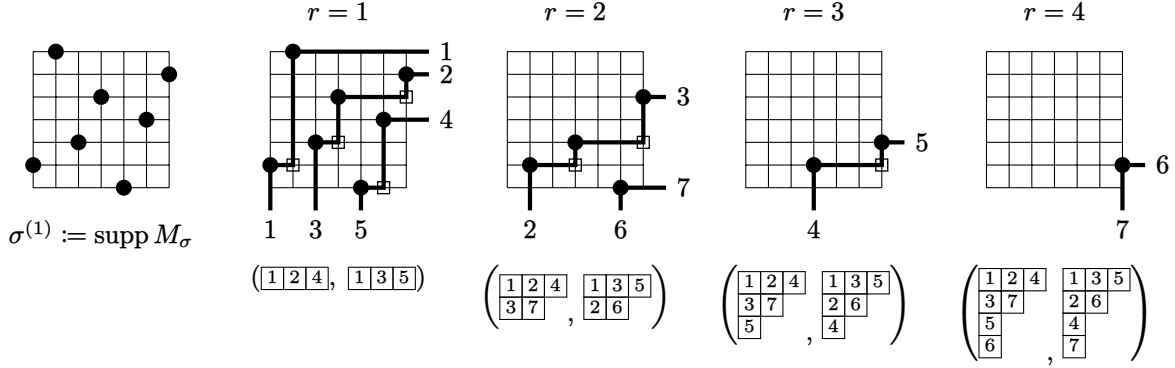
\begin{figure}[t] 
    \centering

  \tikzstyle{dot}=[circle,fill=black, minimum size = 6pt, inner sep=0pt]
\tikzstyle{box}=[rectangle,draw,line width=.5pt, minimum size = 4.5pt, inner sep=0pt]

\begin{tikzpicture}[scale=.3]
    \draw (1,1) grid (7,7);
    \node at (2,7) [dot] {};
    \node at (7,6) [dot] {};
    \node at (4,5) [dot] {};
    \node at (6,4) [dot] {};
    \node at (3,3) [dot] {};
    \node at (1,2) [dot] {};
    \node at (5,1) [dot] {};
    \node [below] at (4,0) {$\sigma^{(1)} \coloneqq \operatorname{supp} M_\sigma$};
    \node [right] at (9,8) {};
    \node at (1,-7) {};
\end{tikzpicture}
\begin{tikzpicture}[scale=.3]
    \draw (1,1) grid (7,7);
    %%L1
    \draw [ultra thick] (8,7) node [right] {1} -- ++(-6,0) node [dot] {} -- ++(0,-5) node [box] {} -- ++(-1,0) node [dot] {} -- ++(0,-2) node [below] {1}
    %%L2
    (8,6) node [right] {2} -- ++(-1,0) node [dot] {} -- ++ (0,-1) node [box] {} -- ++(-3,0) node [dot] {} -- ++(0,-2) node [box] {} -- ++(-1,0) node [dot] {} -- ++(0,-3) node [below] {3} 
    %%L3
    (8,4) node [right] {4} -- ++(-2,0) node [dot] {} -- ++(0,-3) node [box] {} -- ++(-1,0) node [dot] {} -- ++(0,-1) node [below] {5};
    \node [below] at (4,-2) {$\left( \ytableaushort{124}, \; \ytableaushort{135} \right)$};
    \node at (1,-7) {};
    \node [above] at (4,8) {$r=1$};
\end{tikzpicture}
\begin{tikzpicture}[scale=.3]
    \draw (1,1) grid (7,7);
    %%L1
    \draw [ultra thick] (8,5) node [right] {3} -- ++(-1,0) node [dot] {} -- ++(0,-2) node [box] {} -- ++(-3,0) node [dot] {} -- ++(0,-1) node [box] {} -- ++(-2,0) node [dot] {} -- ++(0,-2) node [below] {2}
    %%L2
    (8,1) node [right] {7} -- ++(-2,0) node [dot] {} -- ++ (0,-1) node [below] {6};
    \node [below] at (4,-2) {$\left( \ytableaushort{124,37}_{\textstyle,} \; \ytableaushort{135,26} \right)$};
    \node at (1,-7) {};
    \node [above] at (4,8) {$r=2$};
\end{tikzpicture}
\begin{tikzpicture}[scale=.3]
    \draw (1,1) grid (7,7);
    %%L1
    \draw [ultra thick] (8,3) node [right] {5} -- ++(-1,0) node [dot] {} -- ++(0,-1) node [box] {} -- ++(-3,0) node [dot] {} -- ++(0,-2) node [below] {4};
    \node [below] at (4,-2) {$\left( \ytableaushort{124,37,5}_{\textstyle ,} \; \ytableaushort{135,26,4} \right)$};
    \node at (1,-7) {};
    \node [above] at (4,8) {$r=3$};
\end{tikzpicture}
\begin{tikzpicture}[scale=.3]
    \draw (1,1) grid (7,7);
    %%L1
    \draw [ultra thick] (8,2) node [right] {6} -- ++(-1,0) node [dot] {} -- ++(0,-2) node [below] {7};
    \node [below] at (4,-2) {$\left( \ytableaushort{124,37,5,6}_{\textstyle ,} \; \ytableaushort{135,26,4,7} \right)$};
    \node at (1,-7) {};
    \node [above] at (4,8) {$r=4$};
\end{tikzpicture}
    
    \caption{Example of Viennot's ``light and shadow'' construction of the RS correspondence (Algorithm~\ref{alg:Viennot}), applied to the permutation $\sigma = (2,7,4,6,3,1,5)$.
    For each $r$ we display the shadow diagram of the skeleton $\sigma^{(r)}$, with the shadow lines traced in bold.
    The boxes at the \scalebox{2}{$\lrcorner$}-shapes indicate the next skeleton $\sigma^{(r+1)}$.
    Along the eastern and southern edge we label the values defined in step~\ref{fill rows}, which are the entries in row $r$ of $\se{P}$ and $\se{Q}$, respectively.
    The final pair of tableaux gives $(\se{P}, \se{Q}) = \RS(\sigma)$, revealing that $\sigma$ has shape $\lambda = (3,2,1,1)$.}
    \label{fig: ex of RSK}
\end{figure}

\medskip

It is clear from Algorithm~\ref{alg:Viennot} that the tableaux $\se{P}$ and $\se{Q}$ share a common shape, say $\la \vdash n$, said to be the \emph{shape} of $\sigma$.
This induces a decomposition of $\fS_n$ into the following classes $\fS^\la$:
\begin{equation}
  \label{S n lambda definition}
  \fS_n = \coprod_{\la \vdash n} \fS^\la, \quad \text{where } \fS^\la \coloneqq \Big\{ \sigma \in \fS_n : \text{$\sigma$ has shape $\la$} \Big\}.
\end{equation}
By Greene's theorem~\cite{Greene}*{Thm.~3.1}, if $\sigma \in \fS^\la$, then for all $1 \leq m \leq n$,
\begin{align}
\label{Greene}
\begin{split}
  \sum_{i=1}^m \la_i &= \text{the length of the longest union of $m$ increasing subsequences of $\sigma$}, \\ 
  \sum_{j=1}^m \la'_j &= \text{the length of the longest union of $m$ decreasing subsequences of $\sigma$}.
\end{split}
\end{align}
Recalling the notation $f^\la$ from~\eqref{f}, we have the following by~\eqref{RSK bijection}--\eqref{S n lambda definition}:
\begin{equation}
  \label{size S n lambda}
  \left|\fS^\la \right| = \left| \RS(\fS^\la) \right| = \big| \SYT(\la) \times \SYT(\la) \big| = (f^\la)^2.
\end{equation}
Of primary interest in this paper, for all $1 \leq i, j \leq n$ we define the subset
\begin{equation}
    \label{fS lambda ij}
    \fS^\la_{ij} \coloneqq \Big\{ \sigma \in \fS^\la : \sigma(i) = j \Big\}.
\end{equation}

\subsection*{Sch\"utzenberger involutions}

We denote the ``complement'' of $i \in [n]$ by
\begin{equation}
  \label{bar}
  \bar{\imath} \coloneqq n + 1 - i.
\end{equation}
Consider the element $\omega_0 \in \fS_n$ given by $\omega_0(i) = \bar{\imath}$, that is, $\omega_0 = (n, \ldots, 3,2,1)$.
The maps given by right- and left-multiplication by $\omega_0$ are involutions of $\fS_n$.
Along with inversion, these involutions generate the following group $G$ acting naturally on $\fS_n$:
\begin{equation}
\label{G}
  G \coloneqq \Big\langle \gT : \sigma \mapsto \sigma^{-1}, \quad \gH :  \sigma \mapsto \sigma \omega_0, \quad \gV :  \sigma \mapsto \omega_0 \sigma \Big\rangle.
\end{equation}
The reason for our notation $\gT$, $\gH$, $\gV$ is the fact that $G$ is isomorphic to the dihedral group of order $8$, which is clear from its action $g: M_\sigma \mapsto M_{g \sigma}$ on permutation matrices~\eqref{M sigma ij}:
in particular, $\gT$ is a reflection about the main diagonal (i.e., the transpose), $\gH$ is a horizontal reflection, and $\gV$ is a vertical reflection.
Note that the product 
\begin{equation*}
  \label{G R}
  \gR \coloneqq \gH \gV = \gV \gH
\end{equation*}
is another involution acting on permutations via $\gR : \sigma \mapsto \omega_0 \sigma \omega_0$, and on permutation matrices by a 180-degree rotation.
Explicitly, the following are equivalent:
\begin{align}
  \label{G action}
  \begin{split}
  \sigma(i) &= j, \\
  \gT \sigma (j) & = i, \\
  \gH \sigma (i) &= \bar{\jmath}, \\ 
  \gV \sigma ( \bar{\imath}) &= j, \\
  \gR \sigma (\bar{\imath}) &= \bar{\jmath}.
  \end{split}
\end{align}
That is, if we write $\sigma$ in one-line notation $\sigma = (\sigma(1), \ldots, \sigma(n))$, then $\gH \sigma$ is obtained by taking the pointwise complement of $\sigma$, and $\gV \sigma$ is obtained by reversing $\sigma$,  and $\gR \sigma$ is obtained by doing both.
Below we highlight the well-known interactions between the RS correspondence and the group $G$, due to Schensted~\cite{Schensted} and Sch\"utzenberger~\cites{Schutzenberger63,Schutzenberger}.
In particular, the following are equivalent~\cite{Sagan}*{Thms.~3.2.3, 3.6.6, 3.9.4}:
\begin{align}
  \label{RS involutions}
  \begin{split}
  \RS(\sigma) & = (\se{P}, \se{Q}) \text{ with } \sigma \in \fS^\la, \\
  \RS(\gT \sigma) &= (\se{Q}, \se{P}) \text{ with } \sigma \in \fS^\la, \\ 
  \RS(\gH \sigma ) &= (\se{P}', \epsilon(\se{Q}')) \text{ with } \sigma \in \fS^{\la'}, \\
  \RS( \gV \sigma) &= ( \epsilon(\se{P}'), \se{Q}') \text{ with } \sigma \in \fS^{\la'}, \\
  \RS(\gR \sigma ) &= (\epsilon(\se{P}), \epsilon(\se{Q})) \text{ with } \sigma \in \fS^\la, \\ 
  \end{split}
\end{align}
where $\epsilon$ denotes Sch\"utzenberger's evacuation operator (itself an involution on $\SYT(\la)$; see~\cite{Schutzenberger63} or~\cite{Sagan}*{\S3.9}), and where $\se{P}'$ denotes the transpose of a tableau $\se{P}$.
Combining each line in~\eqref{G action} with its corresponding line in~\eqref{RS involutions}, we obtain the following bijections with the set $\fS^\la_{ij}$ defined in~\eqref{fS lambda ij}:
\begin{align}
\label{S ij bijections}
\begin{split}
  \fS^\la_{ij} & \xleftrightarrow{\quad \gT \quad} \fS^\la_{ji}, \\ 
  \fS^\la_{ij} & \xleftrightarrow{\quad \gH \quad} \fS^{\la'}_{i \bar{\jmath}}, \\
  \fS^\la_{ij} & \xleftrightarrow{\quad \gV \quad} \fS^{\la'}_{\bar{\imath} j}, \\ 
  \fS^\la_{ij} & \xleftrightarrow{\quad \gR \quad} \fS^{\la}_{\bar{\imath} \bar{\jmath}}. \\
  \end{split}
\end{align}

\subsection*{The matrix \texorpdfstring{$P^\la$}{P lambda}}

Our main object of interest is the $n \times n$ matrix $P^\la$ whose $(i,j)$ entry gives the conditional probability
\begin{equation}
  \label{P lambda}
  P^\la_{ij} \coloneqq \frac{| \fS^\la_{ij} |}{| \fS^\la |} = \mathbb{P}\Big( \sigma(i) = j \; \Big| \; \sigma \in \fS^\la \Big),
\end{equation}
where $\sigma \in \fS_n$ is chosen uniformly at random.
Note that $P^\la$ is doubly stochastic.
Due to our combinatorial approach, we will also work with the nonnegative integer matrix $\check{P}^\la$ defined via
\begin{equation}
  \label{P lambda tilde ij}
  \check{P}^\la_{ij} = \left| \fS^\la_{ij} \right|.
\end{equation}
(The notation is meant to suggest that $\check{P}^\la$ is the ``un-normalized'' form of $P^\la$.)
In other words,
\begin{equation}
\label{P lambda tilde}
  \check{P}^\la = \left| \fS^\la \right| \cdot P^\la = \sum_{\mathclap{\sigma \in \fS^\la}} M_\sigma,
\end{equation}
where the second equality follows from the definitions~\eqref{M sigma ij} and~\eqref{fS lambda ij}.

A priori, the matrix $P^\la$ inherits the symmetries from the action of the group $G$ given in~\eqref{G} and~\eqref{S ij bijections}, which we record in the following lemma.
Recall that a square matrix is said to be \emph{symmetric} if it is symmetric about its main diagonal;
\emph{persymmetric} if it is symmetric about its antidiagonal;
and \emph{bisymmetric} if it is both symmetric and persymmetric.

\begin{lemma}
  \label{lemma:symmetries}
  Let $n$ be a positive integer, with $\la \vdash n$.
  The matrix $P^\la$ in~\eqref{P lambda} is bisymmetric, and is the vertical (equivalently, horizontal) reflection of the matrix $P^{\la'}$.
  More precisely, for all $1 \leq i,j \leq n$,
  \[ 
    P^\la_{ij} = P^\la_{ji} = P^{\la'}_{i \bar{\jmath}} = P^{\la'}_{\bar{\imath} j} = P^\la_{\bar{\imath} \bar{\jmath}}.
  \]
\end{lemma}

\begin{proof} 
  By~\eqref{P lambda tilde ij} we have $\check{P}^\la = | \fS^\la_{ij} |$, and thus the bijections in~\eqref{S ij bijections} yield $\check{P}^\la_{ij} = \check{P}^\la_{ji} = \check{P}^{\la'}_{i \bar{\jmath}} = \check{P}^{\la'}_{\bar{\imath} j} = \check{P}^\la_{\bar{\imath} \bar{\jmath}}$.
  (Note that bisymmetry means $\check{P}^\la_{ij} = \check{P}^\la_{\bar{\jmath} \, \bar{\imath}}$, which follows from combining the first and last equality.)
  Since $P^\la$ is a scalar multiple of $\check{P}^\la$, the result now follows.
\end{proof}

\section{Main result for hook shapes}
\label{sec:hooks}

Throughout this section, we fix a positive integer $n$, and we assume that $\la = (\ell, 1^{n-\ell})$, where $1 \leq \ell \leq n$.
By Greene's theorem~\eqref{Greene}, we have $\sigma \in \fS^\la$ if and only if $\sigma$ contains an increasing subsequence of length $\ell$ and a decreasing subsequence of length $\bar{\ell} \coloneqq n + 1 - \ell$.
Note that these two subsequences are not necessarily uniquely determined;
moreover, since $\ell + \bar{\ell} = n+1$, any two such subsequences intersect in exactly one term.
With this in mind, for all $\sigma \in \fS^\la$ we define
\begin{align}
  \mathrm{X}(\sigma) &\coloneqq \left\{ (\ell, i, j) : 
  \begin{array}{l} 
  \text{$\sigma(i) = j$ is contained in both} \\ 
  \text{an increasing subsequence of length $\ell$} \\ 
  \text{and a decreasing subsequence of length $\bar{\ell}$}
  \end{array}
   \right\}, \label{X} \\[1ex]
   \mathrm{Y}(\sigma) &\coloneqq \left\{ (\ell, i, j) : 
  \begin{array}{l} 
  \text{$\sigma(i) = j$, and} \\ 
  \text{$i > i'$ and $j > j'$ for all $(\ell, i', j') \in \mathrm{X}(\sigma)$}
  \end{array}
   \right\}. \label{Y}
\end{align}
In the triples $(\ell,i,j)$ above, the $\ell$ simply keeps track of the fact that $\sigma$ has shape $(\ell, 1^{n-\ell})$.
If we visualize the support of the permutation matrix $M_\sigma$, then we see that $(\ell,i,j) \in \mathrm{X}(\sigma)$ precisely when $(i,j)$ lies in the ``crossroads'' of this support, between a northwest-to-southeast pattern of length $\ell$ and a northeast-to-southwest pattern of length $\bar{\ell}$.
(Hence the ``X'' notation, to evoke this crossroads.)
Meanwhile, $(\ell,i,j) \in \mathrm{Y}(\sigma)$ precisely when $(i,j)$ lies in the support of $M_\sigma$ to the southeast of the crossroads.

Throughout this section, we adopt the shorthand
\begin{align}
  \fS(\ell, i, j) &\coloneqq \fS^{(\ell, 1^{n - \ell})}_{ij}, \label{Slij definition} \\ 
  \fS_{\mathrm{X}}(\ell, i, j) &\coloneqq \Big\{ \sigma \in \fS( \ell, i, j) : 
  (\ell,i,j) \in \mathrm{X}(\sigma)
    \Big\}, \label{SX} \\
  \fS_{\mathrm{Y}}(\ell, i, j) &\coloneqq \Big\{ \sigma \in \fS( \ell, i, j) : 
  (\ell,i,j) \in \mathrm{Y}(\sigma)
    \Big\}. \label{SY}
\end{align}
Recall the 8-element group $G$ from~\eqref{G}.
Since $\la = (\ell, 1^{n - \ell})$ if and only if $\la' = (\bar{\ell}, 1^{n - \bar{\ell}})$, with the bar notation as in~\eqref{bar}, we can rewrite the bijections~\eqref{S ij bijections} as follows:
\begin{align}
  \label{S ij bijections hook notation}
  \begin{split}
  \fS (\ell, i, j) & \xleftrightarrow{\quad \gT \quad} \fS ( \ell, j, i), \\ 
  \fS (\ell, i, j) & \xleftrightarrow{\quad \gH \quad} \fS ( \bar{\ell}, i, \bar{\jmath}), \\
  \fS (\ell, i, j) & \xleftrightarrow{\quad \gV \quad} \fS ( \bar{\ell}, \bar{\imath}, j), \\
  \fS ( \ell, i, j)  & \xleftrightarrow{\quad \gR \quad} \fS(\ell, \bar{\imath}, \bar{\jmath}).
  \end{split}
\end{align}
This induces an action by $G$ on the set $\{ \ell, \bar{\ell} \} \times \{ i, \bar{\imath}\} \times \{ j, \bar{\jmath}\}$, defined so that for all $g \in G$, 
\begin{equation}
  \label{G action on lij}
  \fS( \ell, i, j) \xleftrightarrow{\quad g \quad} \fS( g \cdot (\ell, i, j)).
\end{equation}
From now on we deal primarily with the subgroup of involutions
\begin{equation}
  \label{I}
  I \coloneqq \Big\{ e, \: \gH, \: \gV, \: \gR \Big\} \subset G.
\end{equation}

\begin{lemma}
  \label{lemma:size Slij hooks}
  Let $1 \leq \ell, i, j \leq n$, with all notation as in~\eqref{Slij definition}--\eqref{I}.
  We have
  \[
    \big| \fS(\ell, i, j) \big| = \left| \fS_{\mathrm{X}}(\ell, i, j) \right| + \sum_{g \in I} \big| \fS_{\mathrm{Y}} (g \cdot (\ell, i, j)) \big|.
  \]
\end{lemma}

\begin{proof} 
First observe from~\eqref{G action} and~\eqref{X} that for all $g \in I$,
\begin{equation}
  \label{X equivariant}
  g \cdot \mathrm{X}(\sigma) = \mathrm{X}(g \sigma),
\end{equation}
where the action on the left-hand side is the induced action in~\eqref{G action on lij}.
Next, observe from~\eqref{Y} and~\eqref{SY} and~\eqref{X equivariant} that for each $g \in I$, the set $g \: \fS_{\rm{Y}}(g \cdot (\ell,i,j))$ takes the form
\begin{align*}
  e \: \fS_{\rm{Y}}\big(e \cdot(\ell,i,j) \big) &= \Big\{ \sigma \in \fS( \ell, i, j) : i > i' \text{ and } j > j' \text{ for all } (\ell, i', j') \in \mathrm{X}(\sigma) \Big\},\\
  \gH \: \fS_{\mathrm{Y}} \big( \gH \cdot (\ell, i, j) \big) &= \Big\{ \sigma \in \fS( \ell, i, j) : i > i' \text{ and } j < j' \text{ for all } (\ell, i', j') \in \mathrm{X}(\sigma) \Big\}, \\ 
  \gV \: \fS_{\mathrm{Y}} \big( \gV \cdot (\ell, i, j) \big) &= \Big\{ \sigma \in \fS( \ell, i, j) : i < i' \text{ and } j > j' \text{ for all } (\ell, i', j') \in \mathrm{X}(\sigma) \Big\}, \\ 
  \gR \: \fS_{\mathrm{Y}} \big( \gR \cdot (\ell, i, j) \big) &= \Big\{ \sigma \in \fS( \ell, i, j) : i < i' \text{ and } j < j' \text{ for all } (\ell, i', j') \in \mathrm{X}(\sigma) \Big\}.
\end{align*}
In words, each set above contains those $\sigma$'s for which $(i,j)$ lies (respectively) southeast, southwest, northeast, and northwest of the ``crossroads'' of the support of $M_\sigma$.
Recalling from~\eqref{SX} that $\fS_{\mathrm{X}}(\ell,i,j)$ contains those $\sigma$'s for which $(i,j)$ lies \emph{in} the crossroads, it follows that every $\sigma \in \fS(\ell,i,j)$ lies in exactly one of these five sets:
  \begin{equation}
    \label{S lij disjoint union}
    \fS(\ell, i, j) = \fS_{\mathrm{X}}(\ell, i, j) \sqcup \coprod_{g \in I} g \: \fS_{\mathrm{Y}} \big( g \cdot (\ell, i, j) \big).
  \end{equation}
  Since by~\eqref{G action on lij} translation by $g$ is a bijection, we have 
  \[
    \left| g \: \fS_{\mathrm{Y}}\big( g \cdot (\ell,i,j) \big) \right| = \left| \fS_{\mathrm{Y}}\big(g \cdot (\ell,i,j) \big) \right|,
  \]
  and so the result follows from~\eqref{S lij disjoint union}.
\end{proof}

\begin{lemma}
  \label{lemma:size SX SY}
  Let $1 \leq \ell \leq n$.
  For all $1 \leq i,j \leq n$,
  \begin{align*}
    X(\ell,i,j) \coloneqq \big| \fS_{\mathrm{X}}(\ell, i, j) \big| &= \textstyle 
    \tallmset{K}{i - K} \tallmset{K}{j - K} \tallmset{\ell-K+1}{n - \ell + K - i} \tallmset{\ell - K + 1}{n - \ell + K - j},
    \\[1ex]
    Y(\ell,i,j) \coloneqq \big| \fS_{\mathrm{Y}}(\ell, i, j) \big| &= \sum_{1 < k < K} \textstyle 
    \tallmset{k-1}{i-k} \tallmset{k-1}{j-k} \tallmset{\ell - k+1}{n - \ell + k - i} \tallmset{\ell - k+1}{n - \ell + k - j},
  \end{align*}
  where
  \[
    K \coloneqq K(\ell,i,j) \coloneqq \frac{i + j + \ell - n}{2},
  \]
  and where $\shortmset{a}{b}$ is the multiset number defined in~\eqref{multiset definition}.
\end{lemma}

\begin{proof} 

We will prove each formula by exhibiting a bijection with a certain subset of $\SYT(\la) \times \SYT(\la)$.
First we prove the $X(\ell,i,j)$ formula.
Let $\sigma \in \fS(\ell,i,j)$, and suppose that $\RS(\sigma) = (\se{P}, \se{Q})$.
By~\eqref{X} and~\eqref{SX}, we have $\sigma \in \fS_{\mathrm{X}}(\ell,i,j)$ if and only if there exist indices $1 \leq a_1 < \cdots < a_\ell \leq n$ and $1 \leq b_1 < \cdots < b_{\bar{\ell}} \leq n$ such that
  \begin{equation}
    \label{a's b's proof lemma}
    \sigma(a_1) < \cdots < \sigma(a_\ell) \quad \text{and} \quad \sigma(b_1) > \cdots > \sigma(b_{\bar{\ell}}),
  \end{equation}
  along with indices $K$ and $R$ such that
  \[
    a_K = b_{R} = i, \text{ and thus } \sigma(a_K) = \sigma(b_{R}) = j.
  \]
  It follows that
  \begin{align}
  \label{two sets}
  \begin{split}
     \{b_1, \ldots, b_{R-1} \} &= \{1, \ldots, i - 1 \} \setminus \{a_1, \ldots, a_{K-1}\}, \\ 
    \{\sigma(b_1), \ldots, \sigma(b_{R-1})\} & = \{j+1, \ldots, n \} \setminus \{\sigma(a_{K+1}), \ldots, \sigma(a_\ell)\}.
  \end{split}
  \end{align}
  Since the left-hand sides in~\eqref{two sets} have the same cardinality $R-1$, the right-hand sets must as well, yielding
  \begin{equation}
    \label{K L equation}
    R - 1 = 
    (i - 1) - (K-1) = (n - j) - (\ell - K).
  \end{equation}
  Solving for $K$, we obtain
  $K = (i + j + \ell - n)/2$,
  which coincides with the $K$ defined in the statement of the lemma.
  Hence
  \begin{equation}
    \label{sigma K fact}
    \sigma \in \fS_{\mathrm{X}}(\ell,i,j) \text{ if and only if $\sigma(i) = j$ is the $K$th term in an increasing subsequence of length $\ell$}.    
  \end{equation}
  Hence if $K \notin [ \ell ]$, then automatically $\fS_{\mathrm{X}}(\ell,i,j) = \varnothing$;
  in this case, the $X(\ell,i,j)$ formula in the lemma returns~0 by definition~\eqref{multiset definition}, so we are done.
  We therefore may assume that $K \in [\ell]$.
  By~\eqref{sigma K fact}, we have $\sigma \in \fS_{\mathrm{X}}(\ell,i,j)$ if and only if $(i,j) \in L_K^{(1)}$ in Algorithm~\ref{alg:Viennot}.
  Thus we must have $\se{P}_{1,K-1} < \se{P}_{1,K} \leq i < \se{P}_{1,K+1}$, assuming the entries in question exist.
  Due to the hook shape of $\se{P}$, it follows that the entry $i$ occupies either box $K$ in the first row, or box $i-K+1$ in the first column.
  By the same argument, in $\se{Q}$ the entry $j$ occupies either box $K$ in the first row, or box $j-K+1$ in the first column.
  Hence
  \[
    \RS\left( \fS_{\mathrm{X}}(\ell,i,j) \right) = \left\{ (\se{P}, \se{Q}) \in \SYT(\la) \times \SYT(\la) : 
    \begin{aligned}
    i &= \, \se{P}_{1,K} \text{ or } \se{P}_{i-K+1, 1} \, , \\[-.5ex] 
    j &= \se{Q}_{1,K} \text{ or } \se{Q}_{j-K+1,1}
    \end{aligned} 
    \right\},
  \]
  and, recalling the notation $\se{P}|_{\leq i}$ from~\eqref{T| def}, we have a bijection
  \begin{align}
      \Phi_{\mathrm{X}} : \RS\left( \fS_{\mathrm{X}}(\ell,i,j) \right) &\longrightarrow 
    \left( 
      \begin{array}{r}
      \SYT(K, 1^{i-K}) \\ 
      {} \times \SYT(K, 1^{j-K})
      \end{array}
     \right) 
     \times 
     \left(
      \begin{array}{r}
       \SYT(\ell - K + 1, \: 1^{n - \ell + K - i}) \\  
       {} \times \SYT(\ell - K + 1, \: 1^{n - \ell + K - j})
       \end{array} 
       \right), \label{SX bijection} \\[1ex] 
    ( \se{P}, \se{Q} ) & \longmapsto \big( ( \se{P}|_{\leq i}, \se{Q}|_{\leq j} ), (\se{R}, \se{S}) \big) \nonumber
  \end{align}
  where $\se{R}$ (resp., $\se{S}$) is formed from $\se{P}|_{>i}$ (resp., $\se{Q}|_{>j}$) by decreasing all entries by $i-1$ (resp., $j-1$), and then sliding the two pieces together to meet at a new upper-left box with entry 1.
  This process is clearly invertible, so $\Phi_{\mathrm{X}}^{-1}$ is well defined.
  Since $\RS$ is injective, the $X(\ell,i,j)$ formula now follows by using Lemma~\ref{lemma:f special} to take the cardinality of the right-hand side of~\eqref{SX bijection}.

  Now we prove the $Y(\ell,i,j)$ formula.
  By~\eqref{SY}, we have $\sigma \in \fS_{\mathrm{Y}}(\ell,i,j)$ if and only if there exist (not necessarily unique) sequences~\eqref{a's b's proof lemma} along with indices $k$ and $h$ such that
  \[
    a_k = i, \;\; \sigma(a_k) = j, \qquad b_h < i, \;\; \sigma(b_h) < j, \qquad b_h \notin \{a_1, \ldots, a_\ell\}.
  \]
  It follows that
  \[
    \{ b_m : \sigma(b_m) > j \} \subsetneq \Big( \{ b_m : b_m < i \} \setminus \{a_1, \ldots, a_\ell \} \Big),
  \]
since the $\sigma(b_m)$'s are decreasing, and since $b_h$ is contained in the right-hand set but not the left-hand set.
Rewriting these two sets, we have
\[
  \sigma^{-1}\Big(\{j+1, \ldots, n\} \setminus \{\sigma(a_{k+1}), \ldots, \sigma(a_\ell) \} \Big) \subsetneq \Big( \{1, \ldots, i-1 \} \setminus \{a_1, \ldots, a_{k-1 \}} \Big),
\]
and taking cardinalities we obtain
\[
  (n-j) - (\ell - k ) < (i-1) - (k-1),
\]
and solving for $k$ we have
\[
  k < (i+j+\ell-n)/2 = K.
\]
Thus $\sigma \in \fS_{\mathrm{Y}}(\ell,i,j)$ if and only if $\sigma(i) = j$ is the $k$th term in an increasing subsequence of length $\ell$, for some $k < K$.
Equivalently, $\sigma \in \fS_{\mathrm{Y}}(\ell,i,j)$ if and only if $L_k^{(1)} = \{(i,j)\}$ for some $k < K$ in Algorithm~\ref{alg:Viennot};
note that $L_k^{(1)}$ must be a singleton due to the fact~\eqref{SY} that $(i,j)$ is southeast of every point in $\mathrm{X}(\sigma)$.
Consequently we have $\se{P}_{1,k} = i$ and $\se{Q}_{1,k} = j$.
Note that $k=1$ would force $i=j=1$ and thus $k < K = (1+1+\ell-n)/2 = 1 + (\ell-n)/2 \leq 1$, which is a contradiction; therefore we may take $1 < k < K$.
Hence
  \[
    \RS\left( \fS_{\mathrm{Y}}(\ell,i,j) \right) = \coprod_{1 < k < K} \Big\{ (\se{P}, \se{Q}) \in \SYT(\la) \times \SYT(\la) : 
    i = \se{P}_{1,k} \text{ and } j = \se{Q}_{1,k}
    \Big\}
  \]
  and we have a bijection
  \begin{align}
      \Phi_{\mathrm{Y}} : \RS\left( \fS_{\mathrm{Y}}(\ell,i,j) \right) &\longrightarrow 
    \coprod_{1 < k < K} \left( 
      \begin{array}{r}
      \SYT(k-1, 1^{i-k}) \\ 
      {} \times \SYT(k-1, 1^{j-k})
      \end{array}
     \right) 
     \times 
     \left(
      \begin{array}{r}
       \SYT(\ell - k+1, \: 1^{n - \ell + k - i}) \\  
       {} \times \SYT(\ell - k+1, \: 1^{n - \ell + k - j})
       \end{array} 
       \right), \label{SY bijection} \\[1ex] 
    ( \se{P}, \se{Q} ) & \longmapsto \big( (\se{P}|_{<i}, \se{Q}|_{<j}), (\se{R}, \se{S}) \big) \nonumber
  \end{align}
  where $\se{R}$ (resp., $\se{S}$) is formed from $\se{P}|_{\geq i}$ (resp., $\se{Q}|_{\geq j}$) by decreasing all entries by $i-1$ (resp., $j-1$), and then sliding the two pieces together so that the upper-left corner is the box now labeled 1, that is, the box originally labeled $i$ (resp., $j$).
  The rest of the proof is the same as the $X(\ell,i,j)$ case above.
\end{proof}

\begin{theorem}
  \label{thm:hooks}
  Let $\la = (\ell, 1^{n-\ell})$ be a hook shape, and recall $P^\la_{ij}$ from~\eqref{P lambda}. 
  For all $1 \leq i,j \leq n$,
  \[
   P^\la_{ij} = \frac{X(\ell,i,j) + Y(\ell, i, j) +  Y(\bar{\ell}, i, \bar{\jmath}) + Y(\bar{\ell}, \bar{\imath}, j) + Y(\ell, \bar{\imath}, \bar{\jmath})}{\tallmset{\ell}{n-\ell}^2},
  \]
  where the functions $X$ and $Y$ are those defined in Lemma~\ref{lemma:size SX SY}, where $\bar{\imath} \coloneqq n + 1 - i$, and where $\shortmset{a}{b}$ is the multiset number defined in~\eqref{multiset definition}.
\end{theorem}

\begin{proof}

We have
\begin{align*}
  \check{P}^\la_{ij} & = \left| \fS(\ell, i, j) \right| & \text{by~\eqref{P lambda tilde ij} and~\eqref{Slij definition}} \\ 
&= \left| \fS_{\mathrm{X}}(\ell,i,j) \right| + \sum_{g \in I} \left| \fS_{\mathrm{Y}} \big( g \cdot (\ell,i,j) \big) \right| & \text{by Lemma~\ref{lemma:size Slij hooks}} \\ 
&= X(\ell,i,j) + \sum_{g \in I} Y \big( g \cdot (\ell,i,j) \big) & \text{by Lemma~\ref{lemma:size SX SY}},
\end{align*}
and thus by~\eqref{S ij bijections hook notation}--\eqref{G action on lij}, the numerator in the theorem equals $\check{P}^\la_{ij}$.
The result now follows from~\eqref{size S n lambda} and Lemma~\ref{lemma:f special}.
\end{proof}

\begin{rem}[Diagrammatic interpretation]
\label{rem:visualize hooks}
  The following construction gives a simple way to visualize the contribution to $\operatorname{supp} P^\la$ coming from each of the five terms in Theorem~\ref{thm:hooks}.
  Throughout this remark, when we write ``$\la$,'' we refer to the result of rotating the Young diagram of $\la$ clockwise by 45 degrees, then superimposing it to fit snugly inside a generic $n \times n$ matrix.
  (Note that $\la$ is depicted in red in the equation at the end of this remark.)
  In particular, the upper-left box of $\la$ occupies matrix position $(1, \bar{\ell})$, the upper-right box occupies position $(\ell, n)$, and the lower-left box occupies position $(\bar{\ell}, 1)$.
  By \emph{column $k$ of $\la$}, we mean its $k$th column from the left;
  to account for cases where $k$ is not between $1$ and $\ell$, we imagine the columns continuing in both directions past the edges of $\la$.
  Similarly, by \emph{row $k$ of $\la$}, we mean its $k$th row from the top (imagining the same conventions as described above when $k$ is not between $1$ and $\bar{\ell}$).
  In the proof of Lemma~\ref{lemma:size SX SY}, by solving the equation~\eqref{K L equation} for $K \coloneqq K(\ell,i,j)$ and $R \coloneqq R(\ell,i,j)$ we obtained
    \begin{align*}
      K(\ell,i,j) & = (i + j - \bar{\ell} +1) / 2, \\ 
      R(\ell,i,j) &= (i + \bar{\jmath} - \ell + 1)/2 = K \big( \gH \cdot (\ell,i,j) \big) = i - K(\ell,i,j) + 1.
    \end{align*}
  The position $(i,j)$ occupies column $K(\ell, i, j)$ and row $R(\ell,i,j)$ of $\la$.
  To collect all the $(i,j)$'s together, define five $n \times n$ matrices $P^\la_{\mathrm{X}}$ and $P^\la_g$ via
    \begin{align*}
    \label{five matrices}
      (P^\la_{\mathrm{X}})_{ij} & = 
      X(\ell,i,j) \big/ |\fS^\la|,\\
      (P^\la_g)_{ij} & = 
      Y\big( g \cdot (\ell,i,j) \big) \big/ |\fS^\la | \text{ for all $g \in I$}.
    \end{align*}
  Then Theorem~\ref{thm:hooks} can be rewritten as the matrix equation
    \begin{equation}
      \label{matrix equation hooks}
      P^\la = P^\la_{\mathrm{X}} + \sum_{g \in I} P^\la_g.
    \end{equation}
  The matrix $P^\la_{\mathrm{X}}$ is bisymmetric; $P^\la_{e}$ and $P^\la_{\gR}$ are symmetric only; and $P^\la_{\gH}$ and $P^\la_{\gV}$ are persymmetric only.
  Moreover, upon inspecting where the multiset numbers vanish in Theorem~\ref{thm:hooks}, one can determine the support of each matrix above:
    \begin{itemize}
      \item The support of $P^\la_{\mathrm{X}}$ is its central ``checkerboard'' region
      \[
        \operatorname{supp} P^\la_{\mathrm{X}} = \left\{ (i,j) : \begin{array}{l} 
        \text{$K(\ell,i,j)$ and $R(\ell,i,j)$ are integers, and} \\[.5ex] 
        1 \leq K(\ell,i,j) \leq \ell \text{ and } 1 \leq R(\ell,i,j) \leq \bar{\ell}
        \end{array}
        \right\}.
      \]
            
      \item The support of $P^\la_e$ is its southeastern region
      \[
        \operatorname{supp} P^\la_e = \Big\{ (i,j) : K(\ell,i,j) > 2 \text{ and } 1 < R(\ell,i,j) < \bar{\ell}  \Big\}.
      \]
      
      \item The support of $P^\la_{\gH}$ is its southwestern region
      \[
        \operatorname{supp} P^\la_{\gH} = \Big\{ (i,j) : 1 < K(\ell,i,j) < \ell \text{ and } R(\ell,i,j) > 2 \Big\}.
      \]
      
      \item The matrix $P^\la_{\gV}$ is the 180-degree rotation of $P^\la_{\gH}$, so its support is in its northeastern region.
      
      \item The matrix $P^\la_{\gR}$ is the 180-degree rotation of $P^\la_e$, so its support is in its northwestern region.
      
    \end{itemize}
Below is an example of~\eqref{matrix equation hooks} where $\la = (9, 1^6)$, that is, where $n = 15$ and $\ell = 9$:
\[
  \ytableausetup{boxsize=.47em}
  \begin{tikzpicture}[baseline=(current bounding box.center)]
    \node at (0,0) {\includegraphics[width=2cm]{Heatmaps15/a09.pdf}};
    \node [below] at (0,-1) {$P^\la$};
  \end{tikzpicture}
  =
  \begin{tikzpicture}[baseline=(current bounding box.center)]
    \node at (0,0) {\includegraphics[width=2cm]{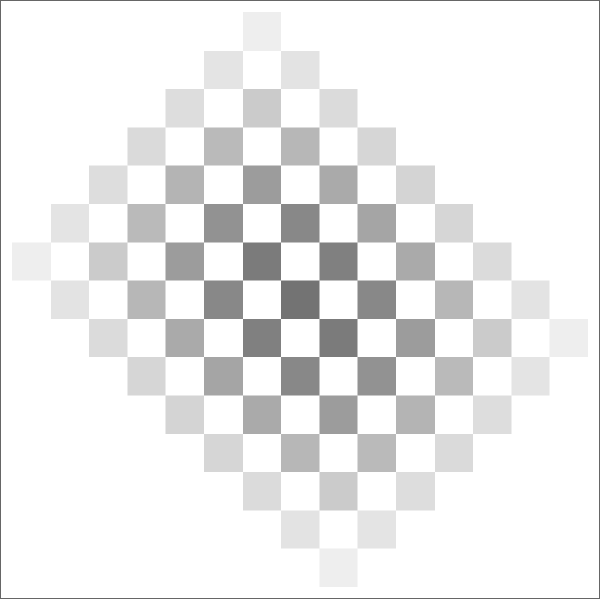}};
    \node [below] at (0,-1) {$P^\la_{\mathrm{X}}$};
    \node [rotate=-45, red] at (0,0) {\ydiagram{9,1,1,1,1,1,1}};
  \end{tikzpicture}
  +
  \begin{tikzpicture}[baseline=(current bounding box.center)]
    \node at (0,0) {\includegraphics[width=2cm]{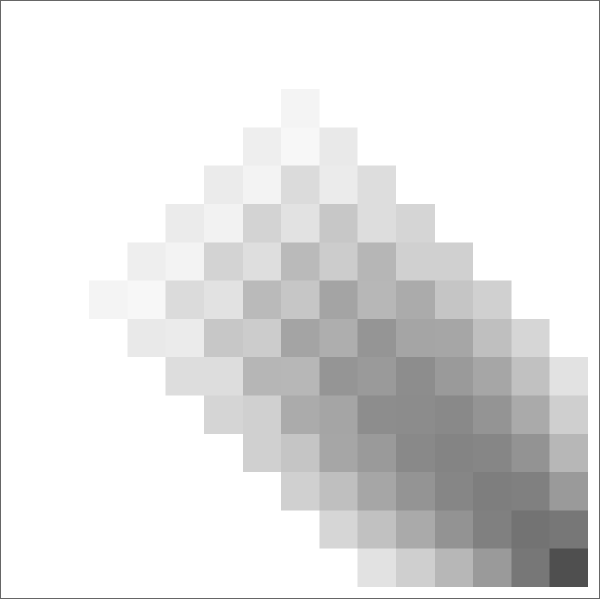}};
    \node [below] at (0,-1) {$P^\la_{e}$};
    \node [rotate=-45, red] at (0,0) {\ydiagram{9,1,1,1,1,1,1}};
  \end{tikzpicture}
  +
  \begin{tikzpicture}[baseline=(current bounding box.center)]
    \node at (0,0) {\includegraphics[width=2cm]{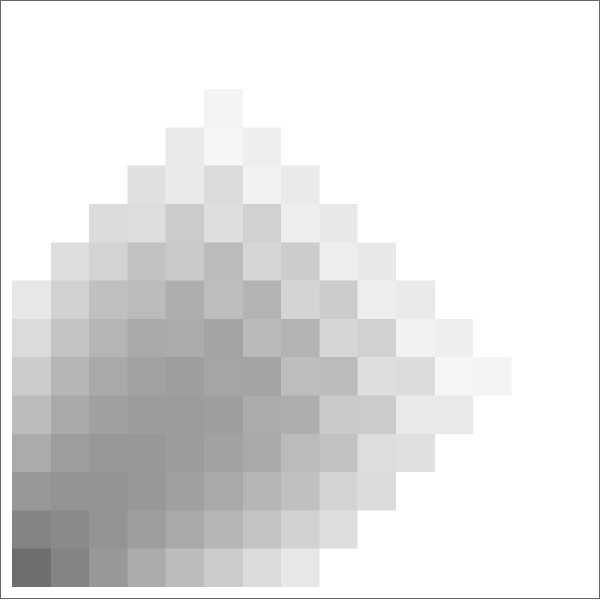}};
    \node [below] at (0,-1) {$P^\la_{\gH}$};
    \node [rotate=-45, red] at (0,0) {\ydiagram{9,1,1,1,1,1,1}};
  \end{tikzpicture}
  +
  \begin{tikzpicture}[baseline=(current bounding box.center)]
    \node at (0,0) {\includegraphics[width=2cm]{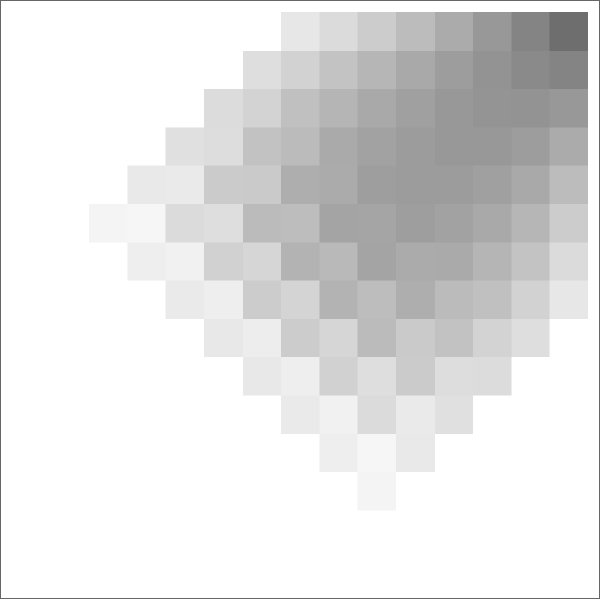}};
    \node [below] at (0,-1) {$P^\la_{\gV}$};
    \node [rotate=-45, red] at (0,0) {\ydiagram{9,1,1,1,1,1,1}};
  \end{tikzpicture}
  +
  \begin{tikzpicture}[baseline=(current bounding box.center)]
    \node at (0,0) {\includegraphics[width=2cm]{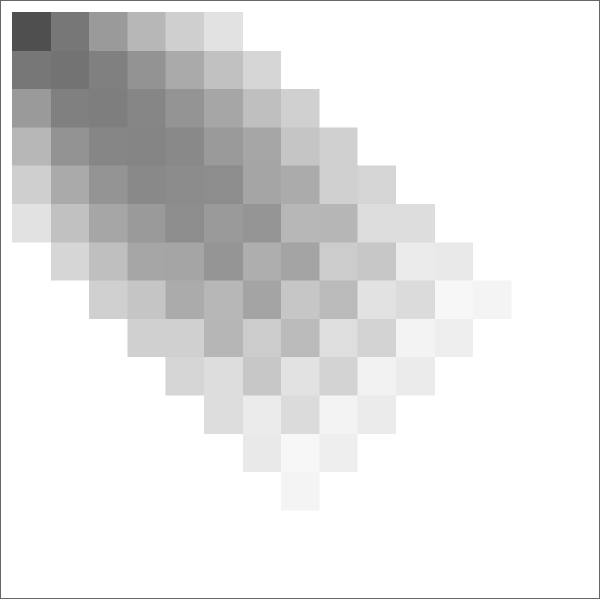}};
    \node [below] at (0,-1) {$P^\la_{\gR}$};
    \node [rotate=-45, red] at (0,0) {\ydiagram{9,1,1,1,1,1,1}};
  \end{tikzpicture}
\]
From the observations above, it follows that if $\la$ is a hook shape, then the support of $P^\la$ is all of $[n] \times [n]$, except in the edge case $\la = (2,1)$ where $P^\la_{2,2} = 0$.
Thus (with the exception of that edge case), prescribing a hook shape $\la$ never prevents $\sigma$ from sending any $i$ to any $j$.
This is in stark contrast with the two-row and rectangular cases, where prescribing $\la$ forces the bandwidth of $\sigma$ to equal $\ell$ or $n - \ell$, respectively; see Remarks~\ref{rem:visualize two-row} and~\ref{rem:rectangles}.
\end{rem}

% \begin{rem}[Hypergeometric series]
%   In Theorem~\ref{thm:hooks}, 
%   % if one wishes to eliminate the summands contributing zero to $Y(\ell,i,j)$, then one can replace the range $1 < k < K$ by $1 < k \leq \min\{i,j,\ell, K - 1/2 \}$; this is clear from inspecting the binomial coefficients in the definition of $Y(\ell,i,j)$.
%   % On the other hand, 
%   when $K > \ell$ (that is, when $i+j > n + \ell$), one can replace the range $1 < k < K$ by $1 < k \leq \ell$, in which case $Y(\ell,i,j)$ can be expressed in terms of a  ${}_4 {F}_3$ hypergeometric series:
%   \begin{equation}
%     \label{4F3 equation}
%     \text{If $i+j > n + \ell$, then } Y(\ell,i,j) = \lim_{(x,y) \rightarrow (i,j)} \textstyle \binom{n-x}{\ell-2} \binom{n-y}{\ell-2} \cdot {}_4 {F}_3 \!\! \left[ 
%       \begin{array}{c} 
%       2 - x, \; 2 - y, \; 2 - \ell, \; 2 - \ell \\[.5ex] 
%       1, \; 2 - x + \bar{\ell} , \; 2 - y + \bar{\ell}   
%       \end{array}
%       ; \; 1 \;
%      \right],
%   \end{equation}
%   where the binomial coefficients carry their fully general definition.
%   \textcolor{red}{Explain; cite Concrete Math.}
%   The limit is required because the hypergeometric series is undefined at $(x,y) = (i,j)$ for sufficiently large values $i$ and~$j$.
%   \textcolor{red}{This nice hypergeo range is precisely outside the $\la$ box in the previous remark.
%   Use $K'$ fact.  KEEP THIS ONLY IF X=0 HERE.}
% \end{rem}

\section{Main result for two-row shapes}
\label{sec:two-row}

Throughout this section, we fix a positive integer $n$, and we assume that $\la = (\ell, n- \ell)$ is a two-row shape (which forces $n/2 \leq \ell \leq n$).
For all $\sigma \in \fS^\la$ we define
\begin{align}
\label{VW}
\begin{split}
  \mathrm{V}(\sigma) & \coloneqq \Big\{ (i,j) \in \operatorname{supp} M_\sigma : 
  \text{$L_k^{(1)} = \{ (i,j) \}$ for some $1 \leq k \leq \ell$}
  \Big\},\\
\mathrm{W}(\sigma) & \coloneqq \Big\{ (i,j) \in \operatorname{supp} M_\sigma : 
  \text{$L_k^{(1)} = \{(i,j), (i',j') \}$ for some $1 \leq k \leq \ell$}
  \Big\},
  \end{split}
\end{align}
with $L_k^{(1)}$ as defined in Algorithm~\ref{alg:Viennot}, where $(i',j') \neq (i,j)$.
(The notation is meant to suggest the fact that, upon rotating the page 135 degrees counterclockwise, the points in $\mathrm{V}(\sigma)$ lie on a V-shaped shadow line, while the points in $\mathrm{W}(\sigma)$ lie on a W-shaped shadow line.)
Throughout this section, we adopt the shorthand
\begin{align}
\label{Slij SV SW}
\begin{split}
  \fS(\ell, i, j) &\coloneqq \fS^{(\ell, n-\ell)}_{ij},\\ 
  \fS_{\mathrm{V}}(\ell, i, j) &\coloneqq \Big\{ \sigma \in \fS( \ell, i, j) : 
  (i,j) \in \mathrm{V}(\sigma)
    \Big\},\\ 
  \fS_{\mathrm{W}}(\ell, i, j) &\coloneqq \Big\{ \sigma \in \fS( \ell, i, j) : 
  (i,j) \in \mathrm{W}(\sigma)
    \Big\}.
  \end{split}
\end{align}

\begin{lemma}
  \label{lemma:Slij count two-row}
  Let $(\ell, n-\ell)$ be a two-row shape, with notation as in~\eqref{Slij SV SW}.
  For all $1 \leq i,j \leq n$,
  \[
    \left| \fS(\ell,i,j) \right| = \left| \fS_{\mathrm{V}}(\ell,i,j) \right| + \left| \fS_{\mathrm{W}}(\ell,i,j) \right|.
  \]
\end{lemma}

\begin{proof}
Let $\sigma \in \fS(\ell,i,j)$.
Since $(\ell, \: n-\ell)$ has at most two rows, Algorithm~\ref{alg:Viennot} terminates at $r \leq 2$.
Hence each $| L_k^{(1)} | \leq 2$, and so $\operatorname{supp} M_\sigma = \mathrm{V}(\sigma) \sqcup \mathrm{W}(\sigma)$.
Upon fixing $(i,j)$, we have $\fS(\ell,i,j) = \fS_{\mathrm{V}}(\ell,i,j) \sqcup \fS_{\mathrm{W}}(\ell,i,j)$, which completes the proof.
\end{proof}

In the following lemma, we use the Iverson bracket $[P]$, defined to be 1 if $P$ is true and 0 if $P$ is false.
We also use the shorthand $[a]_+ \coloneqq \max\{a, \: 0 \}$.

\begin{lemma}
  \label{lemma:size SV SW}
  Let $(\ell, \: n-\ell)$ be a two-row shape.
  For all $1 \leq i,j \leq n$,
  \begin{align*}
    V(\ell,i,j) \coloneqq \left| \fS_{\mathrm{V}}(\ell,i,j) \right| & = 
    \sum_{k=1}^{\ell} \textstyle \tallballot{k-1}{i-k} \tallballot{k-1}{j-k} \tallballot{\ell - k \: \rightarrow [j - i]_{_+}}{n - \ell + k - i} \tallballot{\ell - k \: \rightarrow [i-j]_{_+}}{n - \ell + k - j} , \\
    W(\ell,i,j) \coloneqq \left| \fS_{\mathrm{W}}(\ell,i,j) \right| & = \sum_{\mathclap{(p,q) \in \mathscr{L}}} \textstyle \; \tallballot{p - [i<j]}{i - p - [i>j]} \! \tallballot{q-[i>j]}{j - q - [i<j]} \!
    \left( \!\!\!
      \begin{array}{l}
      \phantom{-} \tallballot{\ell - p \: \rightarrow \: [j-i]_{_+} + p - q}{n - \ell + p - i} \tallballot{\ell - q \: \rightarrow \: [i-j]_{_+} + q - p}{n - \ell + q - j} \\[2ex]
      - \tallballot{\ell - p \: \rightarrow \: [j-i]_{_+} + p - q - 1}{n - \ell + p - i} \tallballot{\ell - q \: \rightarrow \: [i-j]_{_+} + q - p - 1}{n - \ell + q - j}
  \end{array} \!\!
  \right),
  \end{align*}
  where
  \[
    \mathscr{L} \coloneqq  \left\{ (p,q) \in [\ell] \times [\ell] : 
    \begin{array}{l}
      p \leq i, \;\; q \leq j, \\ 
      p+q \geq \max\{i,j\}, \\ 
      0 \leq (p-q)(i-j) < (i-j)^2
    \end{array}    
     \right\},
  \]
  and where $\shortballot{a \rightarrow c}{b}$ and $\shortballot{a}{b}$ are the ballot numbers defined in~\eqref{skew ballot definition}--\eqref{ballot definition}.
\end{lemma}

\begin{proof}

  Let $\sigma \in \fS(\ell,i,j)$.
  By the argument in the proof of Lemma~\ref{lemma:Slij count two-row}, for all $1 \leq k \leq \ell$ in Algorithm~\ref{alg:Viennot}, we have $| L_k^{(1)} | = 1$ or $2$. 
  Moreover, there are exactly $\ell$ many $k$-values such that $| L_k^{(1)}| = 1$, and thus $n - \ell$ many $k$-values such that $| L_k^{(1)}| = 2$.
  Let $a_k$, $b_k$, $A_k$, $B_k$ be the indices uniquely determined so that 
  \begin{equation}
    \label{two possibilities}
    \text{either } L^{(1)}_k = \{ (a_k, b_k) \} \text{ or } L^{(1)}_k = \{ (a_k, B_k), \: (A_k, b_k) \},
  \end{equation}
  where in the second case $a_k < A_k$ and $b_k < B_k$.
  Note that $A_k$ and $B_k$ are defined only for $n-\ell$ many values of $k$, and for these $k$-values, $(A_k, B_k)$ is the location of the unique \scalebox{2}{$\lrcorner$}-shape in the shadow line of $L_k^{(1)}$.
  Passing to the $r=2$ iteration of Algorithm~\ref{alg:Viennot}, then, for all $1 \leq k' \leq n - \ell$ there is a unique $k \geq k'$ such that
  \begin{equation}
    \label{L2}
    \{ (a'_{k'}, b'_{k'}) \} \coloneqq L_{k'}^{(2)} = \{ (A_k, B_k) \}.
  \end{equation}
  Thus $(a'_1, b'_1), \ldots, (a'_{n-\ell}, b'_{n-\ell})$ are the locations of the \scalebox{2}{$\lrcorner$}-shapes in the shadow diagram of $\sigma^{(1)}$, and we have $a'_1 < \cdots < a'_{n-\ell}$ and $b'_1 < \cdots < b'_{n-\ell}$.
  By step~\ref{L1} in Algorithm~\ref{alg:Viennot}, the unique $k$ determined by $k'$ in~\eqref{L2} is given by taking
  \begin{equation}
    \label{k from k'}
    k \text{ maximal such that } a_{k+s} < a'_{k'+s} \text{ and } b_{k+s} < b'_{k'+s} \text{ for all } 0 \leq s \leq n - \ell - k',
  \end{equation}
  since the shadow lines of the $L^{(1)}_{k}$'s do not intersect each other.
  En route to characterizing the two cases in~\eqref{two possibilities}, define
  \begin{equation}
    \label{u(k)}
    u(k) \coloneqq \big| \{ k' : a_k < a'_{k'} \text{ and } b_k < b'_{k'} \} \big|,
  \end{equation}
  and for $1 \leq t \leq u(k)$ put
  \begin{equation}
    \label{a'' b''}
    a''_t \coloneqq a'_{n - \ell - u(k) + t}, \qquad b''_t \coloneqq b'_{n - \ell - u(k) + t}.
  \end{equation}
  Thus $(a''_1, b''_1), \ldots, (a''_{u(k)}, b''_{u(k)})$ are the locations of the \scalebox{2}{$\lrcorner$}-shapes in the shadow diagram of $\sigma^{(1)}$ that lie strictly southeast of $(a_k,b_k)$, and we have $a_k < a''_1 < \cdots < a''_{u(k)}$ and $b_k < b''_1 < \cdots < b''_{u(k)}$.
  Note that $L_k^{(1)} = \{(a_k, b_k)\}$ if and only if $(a''_1, b''_1) = (A_K, B_K)$ for some $K > k$.
  Equivalently, by~\eqref{k from k'},
  \begin{equation}
    \label{size 1 prelim}
    L_k^{(1)} = \{ (a_k, b_k) \} \text{ if and only if $a_{k+t} < a''_t$ and $b_{k+t} < b''_t$ for all $1 \leq t \leq u(k)$}.
  \end{equation}
  On the other hand, by~\eqref{two possibilities}--\eqref{k from k'}, we have $L_k^{(1)} = \{(a_k, B_k), \: (A_k, b_k)\}$ if and only if there is some (necessarily unique) $k' \leq k$ such that $A_k = a'_{k'}$ and $B_k = b'_k$;
  in turn,
  \begin{equation}
    \label{size 2 prelim}
    L_k^{(1)} = \{(a_k, b'_{k'}), \: (a'_{k'}, b_k) \} \text{ if and only if $k' \leq k$ such that $k$ and $k'$ satisfy~\eqref{k from k'}}.
  \end{equation}
  
  Next we translate the preceding observations into tableaux.
  Let $\RS(\sigma) = (\se{P}, \se{Q})$.  
  Note that in either case of~\eqref{two possibilities} we have $a_k = \min\{ i : (i,j) \in L^{(1)}_k \}$ and $b_k = \min\{ j : (i,j) \in L^{(1)}_k \}$;
  likewise in~\eqref{L2} we have $a'_{k'} = \min\{i : (i,j) \in L^{(2)}_{k'} \}$ and $b'_{k'} = \min\{j : (i,j) \in L^{(2)}_{k'} \}$.
  Thus by step~\ref{fill rows} of Algorithm~\ref{alg:Viennot}, along with~\eqref{a'' b''},
  \begin{equation}
    \label{a b to P Q}
    \begin{alignedat}{4}
    & a_k && = \se{P}_{1,k} \qquad && b_k && = \se{Q}_{1, k} \\ 
    & a'_{k'} && = \se{P}_{2,k'} \qquad && b'_{k'} && = \se{Q}_{2, k'} \\
    & a''_{t} && = \se{P}_{2,n-\ell-u(k)+t} \qquad && b''_{t} && = \se{Q}_{2, n - \ell - u(k) + t}
  \end{alignedat}
  \end{equation}
  for all $1 \leq k \leq \ell$ and $1 \leq k' \leq n - \ell$ and $1 \leq t \leq u(k)$.
  Upon combining~\eqref{size 1 prelim} and~\eqref{a b to P Q}, we obtain the following for all $1 \leq k \leq \ell$:
  \begin{equation}
  \label{size 1 iff}
  L^{(1)}_k = \{ (a_k, b_k) \} \text{ if and only if } 
  \left( 
    \begin{array}{l}
    \se{P}_{1,k} = a_k, \; \se{Q}_{1,k} = b_k, \\ 
    \text{and for all $1 \leq t \leq u(k)$}, \\ 
    \se{P}_{1, k+t} < \se{P}_{2, n - \ell - u(k) + t} \text{ and } \se{Q}_{1, k+t} < \se{Q}_{2, n - \ell - u(k) + t}
    \end{array}  
  \right).
  \end{equation}
  Likewise, upon combining~\eqref{size 2 prelim} and~\eqref{a b to P Q}, we obtain the following for all $1 \leq k \leq \ell$:
  \begin{equation}
    \label{size 2 iff}
    L^{(1)}_{k} = \{ (a_k, b'_{k'}), \: (a'_{k'}, b_k) \} \text{ if and only if } \left( 
    \begin{array}{l}
    \se{P}_{1,k} = a_k, \; \se{Q}_{1,k} = b_k, \: \se{P}_{2, k'} = a'_{k'}, \; \se{Q}_{2, k'} = b'_{k'}, \\ 
    \text{$k' \leq k$, and $k$ is maximal such that,} \\ 
    \text{for all $0 \leq s \leq n - \ell - k'$}, \\ 
    \se{P}_{1, k+s} < \se{P}_{2, k'+s} \text{ and } \se{Q}_{1, k+s} < \se{Q}_{2, k'+s}
    \end{array}  
  \right).
  \end{equation}

  We will now use~\eqref{size 1 iff} and~\eqref{size 2 iff} to prove the $V(\ell,i,j)$ and $W(\ell,i,j)$ formulas, respectively, as stated in this lemma.
  We begin with the $V(\ell,i,j)$ formula.
  By the definitions~\eqref{VW}--\eqref{Slij SV SW},
  \begin{equation} 
    \label{Lk singeton in V proof}
    \sigma \in \fS_{\mathrm{V}}(\ell,i,j) \text{  if and only if } L_k^{(1)} = \{(i,j)\} \text{ for some } 1 \leq k \leq \ell.
  \end{equation}
  In turn, by~\eqref{size 1 iff},
  \begin{equation}
    \label{size 1 iff ij}
    L_k^{(1)} = \{ (i,j) \} \text{ if and only if } \left( 
    \begin{array}{l}
    \se{P}_{1,k} = i, \; \se{Q}_{1,k} = j, \\ 
    \text{and for all $1 \leq t \leq u(k)$}, \\ 
    \se{P}_{1, k+t} < \se{P}_{2, n - \ell - u(k) + t} \text{ and } \se{Q}_{1, k+t} < \se{Q}_{2, n - \ell - u(k) + t}
    \end{array}  
  \right).
  \end{equation}
  Combining~\eqref{Lk singeton in V proof} and~\eqref{size 1 iff ij}, we obtain
  \begin{equation}
    \label{size 1 RS}
    \RS( \fS_{\mathrm{V}}(\ell,i,j)) = \coprod_{k=1}^\ell 
    \Big\{ ( \se{P}, \se{Q} ) \in \SYT(\la) \times \SYT(\la) : 
    (\se{P}, \se{Q}) \text{ satisfies right-hand side of~\eqref{size 1 iff ij}}
    \Big\}.
  \end{equation}
  Observe that $(\se{P}, \se{Q})$ satisfies the right-hand side of~\eqref{size 1 iff ij} if and only if the following two conditions hold:
  \begin{itemize} 
  \item $\se{P}_{1,k} = i$ and $\se{Q}_{1,k} = j$;
  \item each of the two-row arrangements below is column strict (where ellipses denote consecutive entries from a row of $\se{P}$ or $\se{Q}$, respectively):
  \begin{equation}
  \label{subtableaux}
  \begin{array}{r} 
  \boxed{
  \begin{array}{lllllll}
  & & \se{P}_{1, k+1} & \cdots\cdots\cdots & \se{P}_{1,k+u(k)} & \cdots & \se{P}_{1, \ell} \\ 
  \se{P}_{2, i-k+1} & \cdots\cdots\cdots & \se{P}_{2, n - \ell - u(k) + 1} & \cdots\cdots\cdots & \se{P}_{2, n - \ell} & &
  \end{array}
}
\\[4ex] 
\boxed{
\begin{array}{lllllll}
  & & \se{Q}_{1, k+1} & \cdots\cdots\cdots & \se{Q}_{1,k+u(k)} & \cdots & \se{Q}_{1, \ell} \\ 
  \se{Q}_{2,j-k+1} & \cdots & \se{Q}_{2, n - \ell - u(k) + 1} & \cdots\cdots\cdots & \se{Q}_{2, n - \ell} & &
\end{array}
}
\end{array}
\end{equation}
\end{itemize}
In each arrangement in~\eqref{subtableaux}, the top row consists of the last $\ell-k$ entries from the top row of $\se{P}$ (resp., $\se{Q}$);
likewise, the bottom row ends with the last $u(k)$ entries from the bottom row of $\se{P}$ (resp., $\se{Q}$), which are now left-justified with the top row.
By~\eqref{u(k)} we have
  \begin{equation}
    \label{u(k) ij}
    u(k) = n - \ell - \max\{i,j\} + k,
  \end{equation}
since the number of entries less than $i$ (resp., $j$) in the bottom row of $\se{P}$ (resp., $\se{Q}$) equals $i-k$ (resp., $j-k$);
thus by~\eqref{u(k) ij} the lengths of the left-hand ``tails'' in~\eqref{subtableaux} are given by
  \begin{equation}
    \label{overhang}
    \begin{split}
    (n - \ell - u(k) + 1) - ( i - k + 1 ) &= \max\{i,j\} - i = [j - i]_+ , \\
    (n - \ell - u(k) + 1 ) - ( j - k + 1 ) &= \max\{i,j\} - j = [i - j]_+.
    \end{split}
  \end{equation}
  Further observe that the entries in each arrangement in~\eqref{subtableaux} are precisely the entries from $\se{P}$ (resp., $\se{Q}$) that are greater than $i$ (resp., $j$);
  therefore the full bottom row has length $n - \ell + k - i$ (resp., $n - \ell + k - j$).
  Thus by~\eqref{overhang},
  \begin{equation}
  \label{P Q shapes for V}
  \begin{array}{l}
  \text{the $\se{P}$ arrangement in~\eqref{subtableaux} has skew shape $(\ell - k + [j-i]_+, \: n - \ell + k-i) / ([j-i]_+)$}, \\ 
  \text{the $\se{Q}$ arrangement in~\eqref{subtableaux} has skew shape $(\ell - k + [i-j]_+, \: n - \ell + k - j) / ([i-j]_+)$}.
  \end{array}
  \end{equation}
  Upon subtracting $i$ (resp., $j$) from every entry, each arrangement in~\eqref{subtableaux} becomes an SYT with the skew shape given in~\eqref{P Q shapes for V}.
  Hence starting from~\eqref{size 1 RS}, and recalling the notation $\se{P}|_{<i}$ from~\eqref{T| def}, we have a bijection
  \begin{align}
    \Phi_{\mathrm{V}} : \RS(\fS_{\mathrm{V}}(\ell,i,j)) & \longrightarrow 
    \coprod_{k=1}^\ell \left( \!\!\!\!
      \begin{array}{r}
      \SYT(k \! - \!1, i \!- \!k) \\ 
      {} \times \SYT(k \! - \! 1, j \!-\!k)
      \end{array}
     \!\!\! \right) 
     \times 
     \left( \!\!\!\!
      \begin{array}{l}
       \phantom{{} \times {}} \SYT((\ell \! - \! k + \! [j \! - \! i]_+, \: n \! - \! \ell \! + \! k \! - \! i) / ([j \! - \! i]_+)) \\  
       {} \times \SYT((\ell \! - \! k \! + \! [i \! - \! j]_+, \: n \! - \! \ell \! + \! k \! - \! j) / ([i \! - \! j]_+))
       \end{array} 
       \!\! \right), \label{SV bijection} \\[1ex] 
    ( \se{P}, \se{Q} ) & \longmapsto \big( (\se{P}|_{<i}, \se{Q}|_{<j}), (\se{R}, \se{S}) \big) \nonumber
  \end{align}
  where $\se{R}$ (resp., $\se{S}$) is formed from $\se{P}|_{> i}$ (resp., $\se{Q}|_{> j}$) by sliding the top row to the left to obtain the alignment in~\eqref{subtableaux}, and then decreasing all entries by $i$ (resp., $j$).
  This process is clearly invertible, so $\Phi_{\mathrm{V}}^{-1}$ is well defined.
  Since RS is injective, the $V(\ell,i,j)$ formula follows by using Lemma~\ref{lemma:f special} to take the cardinality of the right-hand side of~\eqref{SV bijection}.

  It remains to prove the $W(\ell,i,j)$ formula.
  By the definitions~\eqref{VW}--\eqref{Slij SV SW},
  \begin{equation}
    \label{doubleton}
    \sigma \in \fS_{\mathrm{W}}(\ell,i,j) \text{ if and only if } L_k^{(1)} = \{ (i,j), \: (i', j') \} \text{ for some } 1 \leq k \leq \ell,
  \end{equation}
  for some $(i', j') \neq (i,j)$.
  Note that if $(i,j)$ is northeast of $(i', j')$, then $i < j$;
  to see this, observe that $i \geq j$ forces at least one value $i'' < i$ such that $\sigma(i'') \coloneqq j'' > j$, but then $\sigma$ would contain the decreasing subsequence $(j'', j, j')$ which contradicts Greene's theorem~\eqref{Greene} since $\sigma$ has two-row shape.
  By the same argument, if $(i,j)$ is southwest of $(i', j')$, then $i > j$.
  Hence automatically we have $W(\ell,i,i) = 0$, and moreover, we may as well assume that $i \leq j$ and that $(i,j)$ is northeast of $(i',j')$:
  \begin{equation}
    \label{i < j assumption}
    \text{Throughout this proof, we assume in~\eqref{doubleton} that $i \leq j$ and $i < i'$ and $j > j'$}.
  \end{equation}
  At the end of the proof we will easily use the symmetry of Viennot's construction to extend our result to all $(i,j)$.
  By~\eqref{size 2 iff} and~\eqref{i < j assumption},
  \begin{equation}
    \label{size 2 iff ij}
    L_k^{(1)} = \{ (i,j), \: (i', j') \} \text{ if and only if } 
    \left( 
    \begin{array}{l}
    \se{P}_{1,k} = i, \; \se{Q}_{1,k} = j', \: \se{P}_{2, k'} = i', \; \se{Q}_{2, k'} = j, \\ 
    \text{and $k$ is maximal such that,} \\ 
    \text{for all $0 \leq s \leq n - \ell - k'$}, \\ 
    \se{P}_{1, k+s} < \se{P}_{2, k'+s} \text{ and } \se{Q}_{1, k+s} < \se{Q}_{2, k'+s}
    \end{array}  
  \right)
  \end{equation}
  for some $1 \leq k' \leq \min\{k, \: n-\ell\}$.
  Combining~\eqref{doubleton} and~\eqref{size 2 iff ij}, we obtain
  \begin{equation}
    \label{size 2 RS}
    \RS( \fS_{\mathrm{W}}(\ell,i,j)) = \coprod_{k=1}^\ell \; \coprod_{k'=1}^
    k
    \left\{ ( \se{P}, \se{Q} ) \in \SYT(\la) \times \SYT(\la) : 
    \begin{array}{l} 
    (\se{P}, \se{Q}) \text{ satisfies} \\ 
    \text{right-hand side of~\eqref{size 2 iff ij}} \\
    \text{for some $i' > i$ and $j' < j$}
    \end{array}
    \right\}.
  \end{equation}
  (Note that neither $\se{P}_{2,k'}$ nor $\se{Q}_{2,k'}$ exists when $k' > n - \ell$, so in this range the corresponding set in~\eqref{size 2 RS} is automatically empty; hence there is no harm in letting $k'$ range all the way up to $k$.)
  Observe that $(\se{P}, \se{Q})$ satisfies the right-hand side of~\eqref{size 2 iff ij} for some $i' > i$ and $j' < j$ if and only if the following four conditions hold (where the fourth condition is required for the maximality of $k$ in~\eqref{size 2 iff ij}):
  \begin{itemize}
    \item $\se{P}_{1,k} = i$ and $\se{Q}_{2,k'} = j$;
    \item $k + k' > i$ (in order to satisfy $\se{P}_{1,k} < \se{P}_{2,k'}$, which is the $s=0$ case in~\eqref{size 2 iff ij});
    \item each of the two-row arrangements below is column strict (where ellipses denote consecutive entries from a row of $\se{P}$ or $\se{Q}$, respectively):
    \begin{equation}
  \label{subtableauxW}
  \begin{array}{r} 
  \boxed{
  \begin{array}{lllllll}
  & & \se{P}_{1, k+1} & \cdots\cdots \cdots\cdots & \se{P}_{1,n - \ell + k - k'} & \cdots & \se{P}_{1, \ell} \\ 
  \se{P}_{2, i-k+1} & \cdots\cdots & \se{P}_{2, k' + 1} & \cdots \cdots\cdots\cdots & \se{P}_{2, n - \ell} & &
  \end{array}
}
\\[4ex] 
\boxed{
\begin{array}{lllllll}
  & & \se{Q}_{1, j-k'+1} & \cdots & \se{Q}_{1,n-\ell+k-k'} & \cdots & \se{Q}_{1, \ell} \\ 
  \se{Q}_{2,k'+1} & \cdots & \se{Q}_{2, j-k+1} & \cdots & \se{Q}_{2, n - \ell} & &
\end{array}
}
\end{array}
\end{equation}

\item it is \emph{not} the case that both arrangements in~\eqref{subtableauxW} remain column-strict when the top row is shifted left by one entry.
  \end{itemize}
Note that since $\se{Q}_{2,k'} = j$, we have $\se{Q}_{1,1} < \cdots < \se{Q}_{1,j - k'} < j < \se{Q}_{2,k'+1} < \cdots < \se{Q}_{2,n-\ell}$.
For this reason, in the top row of the $\se{Q}$ arrangement in~\eqref{subtableauxW}, we suppressed the entries $\se{Q}_{1,k+s}$ for $s \leq j-k-k'$, since these are precisely the top entries that are less than $j$, and therefore automatically less than the corresponding entries $\se{Q}_{2,k'+s}$, which are all greater than $j$.
We observe that the entries in each arrangement in~\eqref{subtableauxW} are precisely the entries from $\se{P}$ (resp., $\se{Q}$) that are greater than $i$ (resp., $j$).
Moreover, 
\begin{equation}
  \label{P Q shapes for W}
  \begin{array}{l}
  \text{the $\se{P}$ arrangement in~\eqref{subtableauxW} has skew shape $(\ell + k' - i, \: n - \ell + k - i) / (k + k' - i)$}, \\ 
  \text{the $\se{Q}$ arrangement in~\eqref{subtableauxW} has skew shape $(\ell - k', \: n - \ell - k') / (j - k - k')$}.
  \end{array}
\end{equation}
Now taking the skew shapes from~\eqref{P Q shapes for W}, define the sets
\begin{equation}
  \label{A A' B}
  \begin{split}
    \mathscr{A} & \coloneqq 
    \left(
      \begin{array}{l}
      \phantom{{} \times {}} \SYT\big((\ell + k' - i, \: n - \ell + k - i) / (k + k' - i) \big) \\[.5ex]
      {} \times \SYT \big((\ell - k', \: n - \ell - k') / (j - k - k') \big)
      \end{array}
    \right), \\[1ex] 
    \mathscr{A}' & \coloneqq \left\{ (\se{R}, \se{S}) \in \mathscr{A} : \begin{array}{l} 
    \text{both $\se{R}$ and $\se{S}$ remain column strict} \\ 
    \text{upon sliding their top row one box to the left}
    \end{array} 
    \right\},\\[1ex]
    \mathscr{B} & \coloneqq \left(
      \begin{array}{l}
      \phantom{{} \times {}} \SYT\big((\ell + k' - i - 1, \: n - \ell + k - i) / (k + k' - i - 1) \big) \\[.5ex]
      {} \times \SYT \big((\ell - k' - 1, \: n - \ell - k') / (j - k - k' - 1) \big)
      \end{array}
    \right),
  \end{split}
\end{equation}
and note the obvious bijection $\mathscr{A}' \longrightarrow \mathscr{B}$ given by sliding the top row of each tableau to the left by one box, yielding
\begin{equation}
  \label{A' B}
  \left| \mathscr{A}' \right| = \left| \mathscr{B} \right|.
\end{equation}
Hence starting from~\eqref{size 2 RS}, we have a bijection
  \begin{align}
    \Phi_{\mathrm{W}} : \RS(\fS_{\mathrm{W}}(\ell,i,j)) & \longrightarrow 
    \coprod_{k=1}^\ell \: \coprod_{k' = i-k+1}^k \left( \!\!\!\!
      \begin{array}{r}
      \SYT(k-1, \: i-k) \\ 
      {} \times \SYT(j-k', \: k'-1)
      \end{array}
     \!\!\! \right) 
     \times 
     \big( \mathscr{A} \setminus \mathscr{A}' \big), \label{SW bijection} \\[1ex] 
    ( \se{P}, \se{Q} ) & \longmapsto \big( (\se{P}|_{<i}, \se{Q}|_{<j}), (\se{R}, \se{S})) \nonumber
  \end{align}
  where $\se{R}$ (resp., $\se{S}$) is formed from $\se{P}|_{>i}$ (resp., $\se{Q}|_{>j}$) by sliding the top row to the left to obtain the alignment in~\eqref{subtableauxW}, and then decreasing all entries by $i$ (resp., $j$).
  This process is clearly invertible, so $\Phi_{\mathrm{W}}^{-1}$ is well defined.
  We now use Lemma~\ref{lemma:f special} and~\eqref{A A' B}--\eqref{A' B} to take the cardinality of the right-hand side of~\eqref{SW bijection}:
  \begin{equation}
    \label{Wlij k kprime}
    W(\ell,i,j) = \sum_{k=1}^\ell \; \sum_{k' = i-k+1}^k \textstyle \tallballot{k-1}{i-k} \tallballot{j-k'}{k'-1} \left[ 
      \begin{array}{l}\phantom{ {} - {} } \tallballot{\ell-k \rightarrow k + k' - i}{n - \ell + k - i} \tallballot{\ell + k' - j \rightarrow j - k - k'}{n - \ell - k'} \\[2ex] 
      {} - {} \tallballot{\ell-k \rightarrow k + k' - i - 1}{n - \ell + k - i} \tallballot{\ell + k' - j \rightarrow j - k - k' - 1}{n - \ell - k'}
      \end{array}
    \right].
  \end{equation}
  To make the symmetry more apparent in~\eqref{Wlij k kprime}, we reindex via the indices $p = k$ and $q = j - k'$, yielding
\begin{equation}
    \label{Wlij pq}
    W(\ell,i,j) = \sum_{p=1}^\ell \: \sum_{q=j-p}^{p - i + j - 1} \textstyle \tallballot{p-1}{i-p} \tallballot{q}{j-q-1} \left[ 
      \begin{array}{l}\phantom{ {} - {} } \tallballot{\ell-p \rightarrow j-i + p - q}{n - \ell + p - i} \tallballot{\ell - q \rightarrow q-p}{n - \ell + q -j} \\[2ex] 
      {} - {} \tallballot{\ell-p \rightarrow j-i + p - q - 1}{n - \ell + p - i} \tallballot{\ell -q \rightarrow q-p-1}{n - \ell + q -j}
      \end{array}
    \right].
  \end{equation}
  We can rewrite~\eqref{Wlij pq} as a single sum over the set $\{ (p,q) \in [\ell] \times [\ell] : p + q \geq j \text{ and } q-p < j-i \}$, but this includes many terms that vanish;
  by inspecting the ballot numbers in~\eqref{Wlij pq} to see when their arguments become negative, we see that the effective range of the sum is the index set
  \[
    \mathscr{K} \coloneqq \Big\{ (p,q) \in [\ell] \times [\ell] : p \leq i \text{ and } p + q \geq j \text{ and } 0 \leq q-p < j-i \Big\}.
  \]
  Hence we arrive at the version of our $W(\ell,i,j)$ formula presented in the introduction (see Theorem~\ref{thm:two-row in intro}), where we assumed $i \leq j$.
  Indeed, we now recall our ongoing assumption~\eqref{i < j assumption} that $i \leq j$;
  in particular, we proved then that $(i,j)$ is the northeasternmost (resp., southwesternmost) of the two points on its shadow line if and only if $i < j$ (resp., $i > j$).
  Hence the formula~\eqref{Wlij pq} is true for $i \leq j$ (correctly returning zero for $i=j$), but returns zero otherwise.
  Clearly the arguments in this proof are all identical for the symmetrical case where $(i,j)$ is the southwesternmost of the two points on its shadow line and $i > j$;
  thus one can extend~\eqref{Wlij pq} to all $(i,j)$ by stipulating that $W(\ell,i,j) = W(\ell,j,i)$ when $i > j$, as we did in the introduction.
  In order to give a unified formula, however, we simply rewrite~\eqref{Wlij pq} to account for the possibility that $i > j$, which leads to the Iverson brackets and $[ \;\; ]_+$ symbols in the statement of this lemma.
  This new sum ranges over the symmetrized set
  \[
    \mathscr{L} \coloneqq \begin{cases} 
    \{ (p,q): (p,q) \in \mathscr{K} \} & \text{if } i \leq j, \\ 
    \{ (q,p) : (p,q) \in \mathscr{K} \} & \text{if } i \geq j,
    \end{cases}
  \]
  which coincides with the index set $\mathscr{L}$ given in the statement of this lemma;
  note that there is no ambiguity between the two cases since $\mathscr{K} = \varnothing$ when $i=j$.
  This completes the proof of Lemma~\ref{lemma:size SV SW}.
\end{proof}

\begin{theorem}
  \label{thm:two-row}
  Let $\la = (\ell, n-\ell)$ be a two-row shape, and recall $P^\la_{ij}$ from~\eqref{P lambda}.
  For all $1 \leq i,j \leq n$,
  \[
    P^\la_{ij} = \frac{V(\ell,i,j) + W(\ell,i,j)}{\tallballot{\ell}{n - \ell}^2},
  \]
  where the functions $V$ and $W$ are those defined in Lemma~\ref{lemma:size SV SW}, and where $\shortballot{a}{b}$ is the ballot number defined in~\eqref{ballot definition}.
\end{theorem}

\begin{proof}
  We have $\check{P}^\la_{ij} = \left| \fS(\ell,i,j) \right|$ by~\eqref{P lambda tilde ij} and~\eqref{Slij SV SW}.
  Thus by Lemmas~\ref{lemma:Slij count two-row} and~\ref{lemma:size SV SW}, the numerator in the theorem equals $\check{P}^\la_{ij}$.
  The result now follows directly from~\eqref{P lambda tilde}, \eqref{size S n lambda}, and Lemma~\ref{lemma:f special}.
\end{proof}

\begin{rem}[Diagrammatic interpretation and bandwidth]
\label{rem:visualize two-row}
  In this remark, we exclude the trivial case $\ell = n$.
  The following construction (analogous to that in Remark~\ref{rem:visualize hooks}) gives a simple way to visualize the contribution to the support of $P^\la$ coming from the $V$ and $W$ terms in Theorem~\ref{thm:two-row}.
  This time, when we write ``$\la$,'' we refer to the result of shearing the Young diagram of $\la$ by 45 degrees toward the left, then superimposing it snugly inside the strictly upper-diagonal part of a generic $n \times n$ matrix.
  (We depict $\la$ in red in the equation at the end of this remark.)
  In particular, the upper-left box of $\la$ occupies matrix position $(\bar{\ell}-1, \: \bar{\ell})$, the upper-right box occupies position $(\bar{\ell}-1, \: n)$, and the lower-right box occupies position $(\bar{\ell},\: 2\bar{\ell}-1)$.
  To collect all the $(i,j)$'s together, define $n \times n$ matrices $P^\la_{\mathrm{V}}$ and~$P^\la_{\mathrm{W}}$ via
    \[
      (P^\la_{\mathrm{V}})_{ij} = V(\ell,i,j) \big/ | \fS^\la |, \qquad  (P^\la_{\mathrm{W}})_{ij} = W(\ell,i,j) \big/ | \fS^\la |.
    \]
    Then Theorem~\ref{thm:two-row} can be rewritten as the matrix equation
    \begin{equation}
      \label{matrix equation two-row}
      P^\la = P^\la_{\mathrm{V}} + P^\la_{\mathrm{W}}.
    \end{equation}
    Both $P^\la_{\mathrm{V}}$ and $P^\la_{\mathrm{W}}$ are symmetric but not persymmetric.
    Moreover, their support is given as follows:
    \begin{itemize}
      \item \begin{itemize} 
      \item If $\ell \neq \lfloor n/2 \rfloor$, then the support of $P^\la_{\mathrm{V}}$ is the hexagonal region
      \begin{equation}
        \label{supp PV}
        \operatorname{supp} P^\la_{\mathrm{V}} = \Big\{ (i,j) : |i-j| < \min\{ \bar{\ell}, i, j \} \Big\}.
      \end{equation}
      Note that the vertices of this hexagon are the following: $(1,1)$, $(n,n)$, $(\ell, n)$, $(n, \ell)$, the lower-right box of $\la$, and its reflection about the diagonal.

      \item If $\ell = (n+1)/2$, then in~\eqref{supp PV} we include the constraint that $\max\{i,j\}$ is odd.

      \item If $\ell = n/2$, then $\operatorname{supp} P^\la_{\mathrm{V}} = \varnothing$.

      \end{itemize}

      \item \begin{itemize} 
      
      \item If $\ell \neq n/2$, then
      \begin{equation}
        \label{supp PW}
        \operatorname{supp} P^\la_{\mathrm{W}} = \Big\{ (i,j) : 1 \leq |i-j| \leq \ell \Big\}.
      \end{equation}
      Note that this region is just the band containing $\la$, and its reflection about the diagonal.

      \item If $\ell = n/2$, then in~\eqref{supp PW} we include the constraint that if $|i-j| =1$, then $\max\{i,j\}$ is even.
    \end{itemize}
    \end{itemize}
    Below is an example of~\eqref{matrix equation two-row} where $\la = (9,6)$, that is, where $n=15$ and $\ell = 9$:
    \[
  \ytableausetup{boxsize=.32em}
  \begin{tikzpicture}[baseline=(current bounding box.center)]
    \node at (0,0) {\includegraphics[width=2cm]{Heatmaps15/b09.pdf}};
    \node [below] at (0,-1) {$P^\la$};
  \end{tikzpicture}
  =
  \begin{tikzpicture}[baseline=(current bounding box.center)]
    \node at (0,0) {\includegraphics[width=2cm]{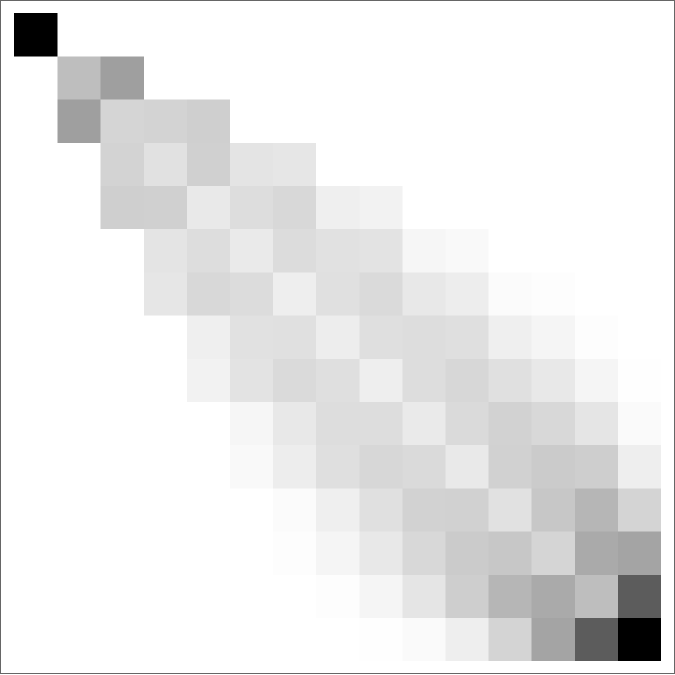}};
    \node [below] at (0,-1) {$P^\la_{\mathrm{V}}$};
    \node [xslant=-1, red] at (.45,.2) {\ydiagram{9,6}};
  \end{tikzpicture}
  +
  \begin{tikzpicture}[baseline=(current bounding box.center)]
    \node at (0,0) {\includegraphics[width=2cm]{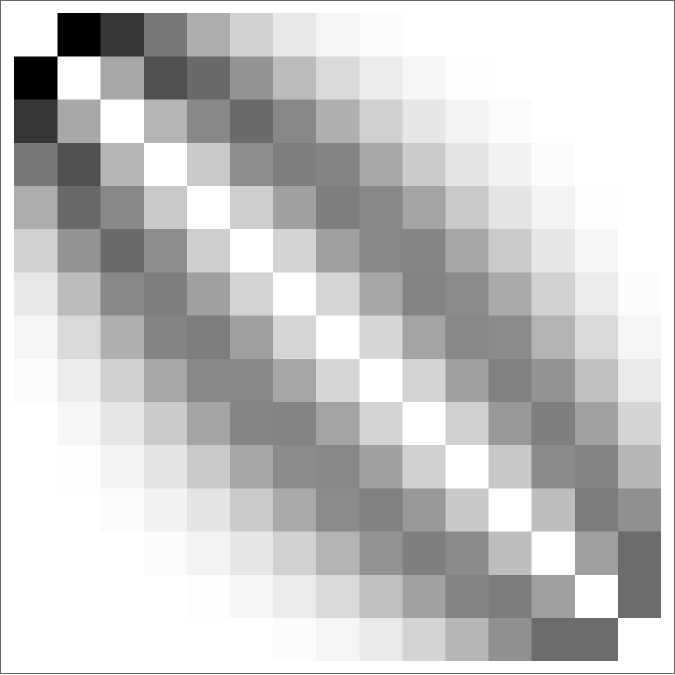}};
    \node [below] at (0,-1) {$P^\la_{\mathrm{W}}$};
    \node [xslant=-1, red] at (.45,.2) {\ydiagram{9,6}};
  \end{tikzpicture}
\]
    From the observations above, it follows that if $\sigma \in \fS^{(\ell, n-\ell)}$, then $\sigma$ has bandwidth $\ell$, where the \emph{bandwidth} of a permutation is the maximum value of $|i - \sigma(i)|$.
  As mentioned in Remark~\ref{rem:visualize hooks}, this limitation makes the two-row case very different from the hook case.
\label{rem:two-row support}

\end{rem}

\section{Main result for rectangular shapes}
\label{sec:rectangles}

Throughout this section, we fix positive integers $\ell$ and $m$, and we assume that $\la = (\ell^m)$ is a rectangular shape.
It is well known~\cite{Gorska}*{eqn.~(1.1)} that for rectangular shapes $\la$, the hook-length formula~\eqref{hook length formula} reduces to the form
\begin{equation}
  \label{f rectangles}
  f^\la = f^{(\ell^m)} = \left[ \prod_{s=0}^{m-1} \frac{s!}{(\ell + s)!} \right] (\ell m)!.
\end{equation}
For any fixed $m \geq 2$, the terms of the sequence $f^{(1^m)}, f^{(2^m)}, f^{(3^m)}, \ldots$ are called the \emph{$m$-dimensional Catalan numbers}, and the $m=2$ case recovers the classical Catalan numbers; see entry A060854 in the OEIS~\cite{OEIS}.

In the case of rectangular shapes, the RS correspondence exhibits an especially nice vertical reflection property.
Given a standard Young tableau $\se{P}$, let $\se{P}^{\uparrow}$ denote the result of reversing the order of the entries within each column of $\se{P}$.
In~\cite{EHMS}*{Def.~3.1}, $\se{P}$ is said to be \emph{admissible} if $\se{P}^{\uparrow}$ is row strict.
Moreover, by~\cite{EHMS}*{Prop.~3.2}, if $\RS(\sigma) = (\se{P}, \se{Q})$ with $\se{P}$ admissible, then
\begin{equation}
  \label{admissible fact}
  \se{Q} = \sigma \cdot \se{P}^{\uparrow},
\end{equation}
where $\sigma \cdot \se{P}^{\uparrow}$ denotes the result obtained from $\se{P}^{\uparrow}$ by replacing each entry $i$ with $\sigma(i)$.

\begin{lemma}
  \label{lemma:RS rectangles}
  Let $\sigma \in \fS^{\la}_{ij}$, where $\la = (\ell^m)$ is a rectangular shape, and suppose $\RS(\sigma) = (\se{P}, \se{Q})$.
  Then
  \[
    i = \se{P}_{r,k} \textup{ if and only if } j = \se{Q}_{m+1-r, k}.
  \]
\end{lemma}

\begin{proof}
  Any tableau of rectangular shape is admissible, since reversing entries in columns is the same as vertically reflecting the entire tableau.
  Thus in this lemma, $\se{P}$ is admissible and hence by~\eqref{admissible fact} we have $\se{Q} = \sigma \cdot \se{P}^{\uparrow}$.
  Since $\sigma \in \fS^\la_{ij}$, we have $\sigma(i) = j$, and thus the entry $j$ in $\se{Q}$ occupies the same position as the entry $i$ in $\se{P}^{\uparrow}$, which is the vertical reflection of the position of $i$ in $\se{P}$.
\end{proof}

% One can, of course, also prove Lemma~\ref{lemma:RS rectangles} via Algorithm~\ref{alg:Viennot}, by showing that for $\la = (\ell^m)$, the point $(i,j)$ is the $r$th northeasternmost point on the shadow line of $L^{(1)}_{k}$ if and only if the shadow line of $L^{(r)}_k$ has its northern boundary at $i$ and the shadow line of $L^{(m+1-r)}_k$ has its western boundary at $j$.
We define the set
\[
  \operatorname{Par}(a \times b) \coloneqq \Big\{ (\mu_1, \ldots, \mu_a) : \mu \text{ is a partition with } \mu'_1 \leq a \text{ and } \mu_1 \leq b  \Big\},
\]
which is essentially the set of all partitions whose Young diagram fits inside a rectangle with $a$ rows and $b$ columns;
note, however, that if $\mu \in \operatorname{Par}(a \times b)$, then technically $\mu$ is an $a$-tuple (possibly with trailing 0's).
Given $\mu \in \operatorname{Par}(a \times b)$, we will use the shorthand
\begin{align*}
  |\mu| & \coloneqq \mu_1 + \cdots + \mu_a, \\ 
  k \pm \mu & \coloneqq (k \pm \mu_1, \ldots, k \pm \mu_a), \\ 
  \overscriptleftarrow{\mu} & \coloneqq (\mu_a, \ldots, \mu_1), \\
  \operatorname{cor}(\mu) & \coloneqq \Big\{ (r,k) : \mu_r = k \text{ and } \mu'_k = r \Big\}, \\
  \mu \setminus (r,k) &\coloneqq \text{the partition obtained from $\mu$ by removing} \\[-.75ex] 
   & \phantom{{} \coloneqq {} } \text{the box in row $r$ and column $k$, assuming $(r,k) \in \operatorname{cor}(\mu)$}.
\end{align*}
In particular, note that $b - \overscriptleftarrow{\mu} \in \operatorname{Par}(a \times b)$ is obtained by taking the complement of $\mu$ inside the $a \times b$ rectangle, and then rotating 180 degrees.
Moreover, $\operatorname{cor}(\mu)$ is the set of \emph{corners} of $\mu$
, meaning the positions where a box can be removed from the Young diagram of $\mu$ while maintaining a valid Young diagram.
We will use a semicolon to denote concatenation of tuples $\mu = (\mu_1, \ldots, \mu_a)$ and $\nu = (\nu_1, \ldots, \nu_c)$ with a positive integer $k$ in between them:
\[
  (\mu; k; \nu) \coloneqq (\mu_1, \ldots, \mu_a, k, \nu_1, \ldots, \nu_c).
\]

\begin{lemma}
  \label{lemma:size S lambda rectangles}
  Let $\la = (\ell^m)$ be a rectangular shape, with $n \coloneqq \ell m$.
  For all $1 \leq i,j \leq n$,
  \[
    Z(\ell,m,i,j) \coloneqq \left| \fS^\la_{ij} \right| = 
    \sum_{\substack{1 \leq r \leq m, \\ 1 \leq k \leq \ell} \phantom{,}}
    \Bigg[
      \sum_{\substack{(\mu,\nu) \in \mathscr{C}_i(r,k), \\ (\pi, \rho) \in \mathscr{D}_j(r,k) \phantom{,}}} \hspace{-3ex}
      f^{(k + \mu; k-1; \nu)} f^{(k + \pi; k-1; \rho)} f^{(\ell - \overscriptleftarrow{\nu}; \ell-k; \ell - k - \overscriptleftarrow{\mu})} f^{ ( \ell - \overscriptleftarrow{\rho}; \ell-k; \ell - k - \overscriptleftarrow{\pi})}
    \Bigg],
  \]
  where
  \begin{align*}
    \mathscr{C}(r,k) &\coloneqq \operatorname{Par}\big((r-1) \times (\ell-k)\big) \times \operatorname{Par}\big((m-r) \times (k-1) \big), \\ 
  \mathscr{D}(r,k) & \coloneqq \operatorname{Par}\big( (m-r) \times (\ell-k)\big) \times \operatorname{Par}\big( (r-1) \times (k-1) \big),\\
    \mathscr{C}_i(r,k) &\coloneqq \Big\{ (\mu, \nu) \in \mathscr{C}(r,k) : 
    |\mu| + |\nu| + kr = i
    \Big\}, \\
    \mathscr{D}_j(r,k) &\coloneqq \Big\{ (\pi, \rho) \in \mathscr{D}(r,k) : 
    |\pi| + |\rho| + k(m+1-r) = j
    \Big\}.
  \end{align*}
\end{lemma}

\begin{proof}
  Let $1 \leq i,j \leq n$.
  For all $1 \leq r \leq m$ and $1 \leq k \leq \ell$, define the sets
  \begin{equation*}
    \label{Ai Bj}
    \begin{split}
      \mathscr{A}_i(r,k) & \coloneqq \Big\{ \alpha \in \operatorname{Par}(m \times \ell) : |\alpha| = i \text{ and } (r,k) \in \operatorname{cor}(\alpha) \Big\}, \\ 
      \mathscr{B}_j(r,k) & \coloneqq \Big\{ \beta \in \operatorname{Par}(m \times \ell) : |\beta| = j \text{ and } (m+1-r,k) \in \operatorname{cor}(\beta) \Big\}.
    \end{split}
  \end{equation*}
  By Lemma~\ref{lemma:RS rectangles},
  \[
    \RS(\fS^\la_{ij}) = \coprod_{\substack{1 \leq r \leq m, \\ 1 \leq k \leq \ell \phantom{,}}} \Big\{ (\se{P}, \se{Q}) \in \SYT(\la) \times \SYT(\la) : \se{P}_{r,k} = i \text{ and } \se{Q}_{m+1-r, k} = j \Big\},
  \]
  and thus we have a bijection
  \begin{align}
    \Phi_{\mathrm{Z}} : \RS(\fS^\la_{ij}) & \longrightarrow \coprod_{\substack{1 \leq r \leq m, \\ 1 \leq k \leq \ell \phantom{,}}} 
    \Bigg[ \coprod_{\substack{\alpha \in \mathscr{A}_i(r,k), \\ \beta \in \mathscr{B}_j(r,k) \phantom{,}}} 
    \left(
      \begin{array}{l}
        \phantom{\times} \SYT\big(\alpha \setminus (r,k) \big) \\ 
        \times \SYT\big(\beta \setminus (m+1-r,k) \big)
      \end{array}
    \right) 
    \times
    \left( 
      \begin{array}{l}
        \phantom{\times} \SYT(\ell - \overscriptleftarrow{\alpha}) \\ 
        \times \SYT(\ell - \overscriptleftarrow{\beta})
      \end{array}
    \right) \Bigg], \label{SZ bijection} \\[1.5ex] 
    (\se{P}, \se{Q}) & \longmapsto \big( (\se{P}|_{<i}, \se{Q}|_{<j}), (\se{R}, \se{S})\big) \nonumber
  \end{align}
  where $\se{R}$ (resp., $\se{S}$) is obtained from $\se{P}|_{>i}$ (resp., $\se{Q}|_{>j}$) by replacing each entry $a$ by $n+1-a$ and then rotating 180 degrees.
  Taking cardinalities of both sides in~\eqref{SZ bijection}, we obtain
  \begin{equation}
    \label{P check with alpha beta}
    Z(\ell, m, i, j) = \sum_{\substack{1 \leq r \leq m, \\ 1 \leq k \leq \ell \phantom{,}}} 
    \Bigg[ \sum_{\substack{\alpha \in \mathscr{A}_i(r,k), \\ \beta \in \mathscr{B}_j(r,k) \phantom{,}}} f^{\alpha \setminus (r,k)} f^{\beta \setminus (m+1-r,k)} f^{\ell - \overscriptleftarrow{\alpha}} f^{\ell - \overscriptleftarrow{\beta}} \Bigg].
  \end{equation}
  To sum out the corners, we observe that
  \begin{align*}
    \mathscr{A}_i(r,k) &= \Big\{ (k + \mu; k; \nu) : (\mu,\nu) \in \mathscr{C}_i(r,k) \Big\}, \\ 
    \mathscr{B}_j(r,k) &= \Big\{ (k + \pi; k; \rho) : (\pi,\rho) \in \mathscr{D}_j(r,k) \Big\},
  \end{align*}
  which allows us to rewrite the inner sum in~\eqref{P check with alpha beta} as
  \[
    \sum_{\substack{(\mu,\nu) \in \mathscr{C}_i(r,k), \\ (\pi, \rho) \in \mathscr{D}_j(r,k) \phantom{,}}}
      f^{(k + \mu; k-1; \nu)} f^{(k + \pi; k-1; \rho)} f^{\ell - \overscriptleftarrow{(k + \mu; k; \nu)}} f^{\ell - \overscriptleftarrow{(k + \pi; k; \rho)}}.
  \]
  Upon performing the simplifications
  \begin{align*}
    \ell - \overscriptleftarrow{(k + \mu; k; \nu)} &= (\ell - \overscriptleftarrow{\nu}; \ell-k; \ell - k - \overscriptleftarrow{\mu}), \\ 
    \ell - \overscriptleftarrow{(k + \pi; k; \rho)} &= (\ell - \overscriptleftarrow{\rho}; \ell-k; \ell - k - \overscriptleftarrow{\pi}),
  \end{align*}
  we obtain the result in the form given in the lemma.
\end{proof}

\begin{theorem}
  \label{thm:rectangles}
  Let $\la = (\ell^m)$ be a rectangular shape, with $n \coloneqq \ell m$, and recall $P^\la_{ij}$ from~\eqref{P lambda}.
  For all $1 \leq i,j \leq n$,
  \[
    P^\la_{ij} = \frac{Z(\ell,m,i,j)}{\left( \left[ \prod_{s=0}^{m-1} \frac{s!}{(\ell+s)!} \right] n! \right)^{\!2}},
  \]
  where $Z(\ell,m,i,j)$ is the function defined in Lemma~\ref{lemma:size S lambda rectangles}.
\end{theorem}

\begin{proof}
  This follows immediately from~\eqref{P lambda tilde ij}, \eqref{f rectangles}, and Lemma~\ref{lemma:size S lambda rectangles}.
\end{proof}

\begin{rem}[Diagrammatic interpretation and bandwidth]
\label{rem:rectangles}
  When $\la$ is a rectangular shape, the support of $P^\la$ is contained in the octagonal region consisting of those positions $(i,j)$ whose diagonal has length at least $\ell$, and whose antidiagonal has length at least $m$.
  Therefore if $\sigma \in \fS^{(\ell^m)}$, then $\sigma$ has bandwidth $n - \ell = \ell(m-1)$.
  Note that when $m=2$, this result coincides with the two-row case in Remark~\ref{rem:visualize two-row}.
\end{rem}

Theorem~\ref{thm:rectangles} can be restated more cleanly as follows, by collecting all of the $P^\la_{ij}$'s into a single generating function.
(The same is true of Theorems~\ref{thm:hooks} and~\ref{thm:two-row}, although there is less of an aesthetic improvement in doing so.)
Define the following polynomials in $x$ and $y$:
\begin{align*}
  \check{p}(\ell,m;x,y) & \coloneqq \sum_{\mathclap{\substack{1 \leq r \leq m, \\ 1 \leq k \leq \ell}}} (x^r y^{m+1-r})^k \Bigg[\sum_{\substack{(\mu,\nu) \in \mathscr{C}(r,k), \\ (\pi, \rho) \in \mathscr{D}(r,k) \phantom{,}}} 
  \left( \!\!\!
    \begin{array}{l}
    \phantom{{} \cdot {}} f^{(k+\mu;k-1; \nu)} \cdot f^{(\ell - \overscriptleftarrow{\nu}; \ell-k; \ell - k - \overscriptleftarrow{\mu})} \\ 
    {} \cdot f^{(k + \pi ; k-1; \rho)} \cdot f^{(\ell - \overscriptleftarrow{\rho}; \ell-k; \ell - k - \overscriptleftarrow{\pi})}
    \end{array}
  \right)
      x^{|\mu|+|\nu|} y^{|\pi|+|\rho|} 
      \Bigg], \\[1ex]
  p(\ell, m; x,y) & \coloneqq \check{p}(\ell, m; x,y) \Big/ (f^{(\ell^m)})^2,
\end{align*}
where $f^{(\ell^m)}$ is given explicitly in~\eqref{f rectangles}.
It is straightforward to check that Theorem~\ref{thm:rectangles} can now be rewritten as follows:
  \[
    p(\ell, m ; x, y) = \sum_{i,j = 1}^n P^\la_{ij} \: x^i y^j.
  \]
For example, let $\ell = m = 3$, which is the smallest example outside the two-row case.
Upon programming the polynomial $\check{p}(\ell,m; x,y)$, we immediately compute that $\check{p}(3,3;x,y)$ equals
\begin{center}
\scalebox{.9}
{$\left( 
 \begin{array}{rrrrrrrrr}
& & 210xy^3 & + 420 x y^4 & + 504 xy^5 & + 420 xy^6 & + 210 xy^7 & & \\
& + 441x^2y^2 & + 336 x^2 y^3 & + 105 x^2 y^4 & & + 105 x^2 y^6 & + 336 x^2 y^7 & + 441x^2 y^8 & \\ 
+210x^3y &+ 336x^3y^2 &+ 256 x^3y^3 &+ 80x^3 y^4 & & + 80x^3y^6 & + 256x^3 y^7 & + 336x^3 y^8 & + 210x^3 y^9\\
+420 x^4 y & + 105 x^4 y^2 & + 80 x^4 y^3 & + 169 x^4 y^4 & + 216 x^4 y^5 & + 169 x^4 y^6 & + 80 x^4 y^7 & + 105 x^4 y^8 & + 420 x^4 y^9 \\
+504x^5 y & & & + 216 x^5y^4 & +324 x^5y^5 & + 216 x^5 y^6 & & & + 504x^5y^9\\
+420 x^6 y & + 105 x^6 y^2 & + 80 x^6 y^3 & + 169 x^6 y^4 & + 216 x^6 y^5 & + 169 x^6 y^6 & + 80 x^6 y^7 & + 105 x^6 y^8 & + 420 x^6 y^9 \\
+210x^7 y &+ 336x^7 y^2 &+ 256 x^7 y^3 &+ 80x^7 y^4 & & + 80x^7y^6 & + 256x^7 y^7 & + 336x^7 y^8 & + 210x^7 y^9\\
& + 441x^8 y^2 & + 336 x^8 y^3 & + 105 x^8 y^4 & & + 105 x^8 y^6 & + 336 x^8 y^7 & + 441x^8 y^8 & \\ 
& & 210x^9y^3 & + 420 x^9 y^4 & + 504 x^9 y^5 & + 420 x^9 y^6 & + 210 x^9 y^7 & &
\end{array} 
\right)$},
\end{center}
\noindent where we have arranged each $x^i y^j$ term in the $(i,j)$ position, so that $\check{P}^\la$ is obtained by ignoring the $x$'s and $y$'s.
Then $P^\la$ is obtained from $\check{P}^\la$ by dividing each entry by $(f^{(3,3,3)})^2 = 1764$.

\section{Proof of Theorem~\ref{thm:fix in intro}}
\label{sec:fixed}

Below we reproduce the statement of Theorem~\ref{thm:fix in intro}, an asymptotic result concerning the average proportion of fixed points for permutations of hook or two-row shape.
Recall that the value of the $k$th Wallis integral $W_k \coloneqq \int_0^{\pi/2} \sin^k x \: dx$ depends on the parity of $k$, as follows:
\begin{align}
W_{2m} & = \frac{(2m-1)!!}{(2m)!!} \cdot \frac{\pi}{2} = \frac{(2m)!}{4^m (m!)^2} \cdot \frac{\pi}{2}, \nonumber \\[1ex]
W_{2m+1} & = \frac{(2m)!!}{(2m+1)!!}  = \frac{4^m (m!)^2}{(2m+1)!}, \label{Wallis odd}
\end{align}
where as usual $(2m)!! \coloneqq 2m(2m-2) \cdots 4 \cdot 2$ and $(2m+1)!! \coloneqq (2m+1)(2m-1) \cdots 3 \cdot 1$.

\begingroup 

\renewcommand\thetheorem{1.5}
  \begin{theorem} 
  Let $\la = (\ell, 1^m)$ or $(\ell, m)$, with $n \coloneqq \ell+m$, and let $\sigma$ be chosen uniformly at random from~$\fS^{\la}$. 
  For fixed $m \geq 0$,
    \[
      \lim_{\ell \rightarrow \infty} \mathbb{E} \left[ \frac{\left|\operatorname{Fix}(\sigma) \right|}{n} \right] = 
      \frac{(2m)!!}{(2m+1)!!} = W_{2m+1}.
    \]
  \end{theorem}
\endgroup

The proof of Theorem~\ref{thm:fix in intro} will begin with the observation that the expected proportion of fixed points in a uniform random element of $\fS^\la$ equals $1/n$ times the trace of $P^\la$, which can be obtained from the explicit $P^\la_{ij}$ formulas in Theorems~\ref{thm:hooks} and~\ref{thm:two-row}.
As an example before the proof, the following table compares the limiting values $W_{2m+1}$ against the actual values of $\operatorname{tr} P^\la / n$, for both hook and two-row shapes, for $n=2000$:
\[
  \renewcommand{\arraystretch}{1.5}
  \begin{array}{r|ccccccccccc}
    m & 0 & 1 & 2 & 3 & 4 & 5 & 6 & 7 & 8 & 9 & 10 \\ \hline 
    \text{exact } W_{2m+1} & 1 & \frac{2}{3} & \frac{8}{15} & \frac{16}{35} & \frac{128}{315} & \frac{256}{693} & \frac{1024}{3003} & \frac{2048}{6435} & \frac{32768}{109395} & \frac{65536}{230945} & \frac{262144}{969969} \\ 
    \text{approx. }W_{2m+1} & 1 & 0.667 & 0.533 & 0.457 & 0.406 & 0.369 & 0.341 & 0.318 & 0.300 & 0.284 & 0.270 \\ 
    \operatorname{tr} P^{(n-m, 1^{m})}/n & 1 & 0.666 & 0.533 & 0.456 & 0.405 & 0.368 & 0.340 & 0.316 & 0.298 & 0.282 & 0.268 \\ 
    \operatorname{tr} P^{(n-m, \: m)}/n & 1 & 0.666 & 0.532 & 0.455 & 0.404 & 0.367 & 0.338 & 0.315 & 0.296 & 0.280 & 0.266
  \end{array}
\]
(All approximations have been rounded to three decimal places.)
Before proving Theorem~\ref{thm:fix in intro}, we record the following well-known binomial coefficient identities:
\begin{align}
  \textstyle \binom{a}{b}^2 &= \sum_{t=0}^b \textstyle \binom{a}{b+t} \binom{b+t}{t} \binom{b}{t} & \text{(see bijective argument below)}, \label{ID squaring binomial}\\ 
  \sum_{h=0}^a \textstyle \binom{h}{b} \binom{a-h}{c} & = \textstyle \binom{a+1}{b+c+1} & \text{special case of equation (5.26) in \cite{GKP}*{p.~169}}, \label{ID upper convolution} \\ 
  \sum_{p=0}^m \textstyle \binom{2p}{p} \binom{2m - 2p}{m-p} &= 4^m & \text{equation (5.39) in~\cite{GKP}*{p.~187}}, \label{ID 4m} \\ 
  {\textstyle \binom{a+c}{b}} & = \frac{a^b}{b!} + O(a^{b-1}), \label{binomial order} \\[1ex]
  {\textstyle \binom{a+c}{b} } & \sim \frac{a^b}{b!} \text{ as } a \rightarrow \infty. \label{ID asymptotic}
\end{align}
The identity~\eqref{ID squaring binomial} follows immediately from the following bijective argument.
In particular, the left-hand side of~\eqref{ID squaring binomial} is the number of ordered pairs $(X,Y)$ of $b$-element subsets of an $a$-element set, while the right-hand side refines these ordered pairs by the quantity $t \coloneqq |X \setminus Y| = |Y \setminus X|$;
indeed, for fixed $t$, the number of such pairs $(X,Y)$ is the number $\binom{a}{b+t}$ of choices of $X \cup Y$, times the number $\binom{b+t}{t}$ of choices of $X \setminus Y$, times the number $\binom{b}{t}$ of remaining choices of $Y \setminus X$.
The fact~\eqref{binomial order} follows from the fact that $\binom{a+c}{b} = (a+c)!/(b! \cdot (a-b+c)!)$ is a polynomial in~$a$ with leading term $a^b / b!$, and this in turn yields the asymptotic~\eqref{ID asymptotic}.

\begin{proof}[Proof of Theorem~\ref{thm:fix in intro}]

For a uniform random element $\sigma \in \fS^\la$, the definitions imply that the expected value of $\left| \operatorname{Fix}(\sigma) \right|$ equals the trace of $P^\la$:
\begin{align}
  \mathbb{E} \Big[ \left| \operatorname{Fix}(\sigma) \right| \Big] & = \frac{1}{\left| \fS^\la \right|} \sum_{\sigma \in \fS^\la} \left| \{ i : \sigma(i) = i \} \right| \nonumber \\ 
  &= \frac{1}{\left| \fS^\la \right|} \sum_{i=1}^n \left| \{ \sigma \in \fS^\la : \sigma(i) = i \} \right| = \sum_{i=1}^n \frac{\left| \fS^\la_{ii} \right|}{\left| \fS^\la \right| } = \sum_{i=1}^n P^\la_{ii} = \operatorname{tr} P^\la. \label{E fix equals tr P}
\end{align}

  \textbf{Hook case: $\la = (\ell, 1^m)$}.
  We adopt the notation of Section~\ref{sec:hooks}.
  Since $m = n- \ell$, note that $K(\ell,i,i) = i - (m/2)$.
  Thus by Theorem~\ref{thm:hooks}, upon converting the multiset numbers into binomial coefficients via~\eqref{multiset definition}, the trace of $\check{P}^\la$ equals
  \[
    \sum_{i=1}^n \Bigg[ \overbrace{\textstyle \binom{i-1}{m/2}^2 \binom{n-i}{m/2}^2}^{X(\ell,i,i)} + \displaystyle \overbrace{\sum_{\mathclap{k < i - (m/2)}} \textstyle \binom{i-2}{i-k}^2 \binom{n-i}{m -i + k}^2}^{Y(\ell, i, i)} + \displaystyle \overbrace{\sum_{\mathclap{k < \bar{\imath} - (m/2)}} \textstyle \binom{\bar{\imath} - 2}{\bar{\imath} - 2}^2 \binom{n - \bar{\imath}}{m - \bar{\imath} + k}^2}^{Y(\ell, \bar{\imath}, \bar{\imath})} + \displaystyle \overbrace{2\sum_{\mathclap{k < (m/2)+1}} \textstyle \binom{i-2}{k-2} \binom{\bar{\imath} - 2}{k-2} \binom{n - i}{m + 1 - k} \binom{n - \bar{\imath}}{m + 1 - k}}^{Y(\bar{\ell}, i, \bar{\imath}) + Y(\bar{\ell}, \bar{\imath}, i)}  \Bigg],
  \]
  which equals the following upon reindexing:
  \begin{equation}
    \label{big sum over h fix m}
    \sum_{h=0}^{n-1} \Bigg[ \overbrace{\textstyle \binom{h}{m/2}^2 \binom{n-h-1}{m/2}^2}^{h = i-1} + \displaystyle \overbrace{\sum_{\mathclap{0 \leq p < m/2}} \textstyle \binom{h}{m-p}^2 \binom{n-h-2}{p}^2}^{\substack{h=i-2, \\ p = m-i+k}} + \displaystyle \overbrace{\sum_{\mathclap{m/2 < p \leq m}} \textstyle \binom{h}{m-p}^2 \binom{n-h-2}{p}^2}^{\substack{h = i-1, \\ p = \bar{\imath} - k}} + \displaystyle \overbrace{2\sum_{\mathclap{0 \leq p < m/2}} \textstyle \binom{h-1}{p-1} \binom{n-h-2}{p-1} \binom{n-h-1}{m - p} \binom{h}{m-p}}^{\substack{h = i-1, \\ p = k-1}}  \Bigg].
  \end{equation}
  Note that among the four terms in~\eqref{big sum over h fix m}, term 1 vanishes if $m$ is odd, and otherwise can be rewritten as follows:
  \begin{align}
    \text{term 1} & = \sum_{h=0}^{n-1} \textstyle \binom{h}{m/2}^2 \binom{n-h-1}{m/2}^2 \nonumber \\
    & = \left(\sum_{t=0}^{m/2} {\textstyle  \binom{i-1}{(m/2) + t} \binom{(m/2)+t}{t} \binom{m/2}{t} } \right) \left( \sum_{u=0}^{m/2} {\textstyle \binom{n-i}{(m/2) + u} \binom{(m/2)+u}{u} \binom{m/2}{u} } \right) & \text{by~\eqref{ID squaring binomial}} \nonumber \\
    &= \sum_{t,u = 0}^{m/2} {\textstyle \binom{m/2}{t} \binom{(m/2)+t}{t} \binom{m/2}{u} \binom{(m/2) + u}{u}} \left( \sum_{h=0}^{n-1} {\textstyle \binom{i-1}{(m/2)+t} \binom{n-i}{(m/2)+u}}\right) & \text{rearranging} \nonumber \\ 
    &= \sum_{t,u = 0}^{m/2} {\textstyle \binom{m/2}{t} \binom{(m/2)+t}{t} \binom{m/2}{u} \binom{(m/2) + u}{u}} \cdot \textstyle \binom{n}{m + t+ u + 1} & \text{by~\eqref{ID upper convolution}} \nonumber \\ 
    & \sim {\textstyle \binom{m}{m/2}^2 } \cdot \frac{n^{2m+1}}{(2m+1)!} \text{ as $n \rightarrow \infty$} & \text{by~\eqref{ID asymptotic}}, \label{term 1 asymptotic}
  \end{align}
  where the last line is obtained by observing that $m + t + u + 1$ obtains its maximum if and only if $t = u = m/2$.
  Next, note that terms 2 and 3 in~\eqref{big sum over h fix m} both vanish at $h=n-1$, and can be rewritten in the same way as term 1 above, upon putting $\mathscr{P} \coloneqq \{0, \ldots, m\} \setminus \{m/2\}$:
  \begin{align*}
    \text{term 2 + term 3} &= \sum_{h=0}^{n-2} \: \sum_{p \in \mathscr{P}} \textstyle \binom{h}{m-p}^2 \binom{n-h-2}{p}^2 \\ 
     &= \sum_{h=0}^{n-2} \: \sum_{p \in \mathscr{P}} \left( \sum_{t=0}^{m-p} {\textstyle   \binom{h}{m-p+t} \binom{m-p+t}{t} \binom{m-p}{t}} \right) \left( \sum_{u=0}^{p} {\textstyle \binom{n - h - 2}{p + u} \binom{p+u}{u} \binom{p}{u}  } \right) & \text{by~\eqref{ID squaring binomial}} \\ 
    &= \sum_{p \in \mathscr{P}}  \sum_{t=0}^{m-p} \sum_{u=0}^{p} {\textstyle \binom{m-p}{t} \binom{m-p+t}{t} \binom{p}{u} \binom{p+u}{u}} \left(\sum_{h=0}^{n-2} {\textstyle \binom{h}{m-p+t} \binom{n-h-2}{p+u} } \right) & \text{rearranging} \\
    &= \sum_{p \in \mathscr{P}}  \sum_{t=0}^{m-p} \sum_{u=0}^{p} {\textstyle \binom{m-p}{t} \binom{m-p+t}{t} \binom{p}{u} \binom{p+u}{u}} \cdot \textstyle \binom{n-1}{m+t+u+1} & \text{by~\eqref{ID upper convolution}} \\
    & \sim \left(\sum_{p \in \mathscr{P}} {\textstyle \binom{2p}{p} \binom{2m -2p}{m-p}} \right) \cdot \frac{n^{2m+1}}{(2m+1)!} \text{ as $n \rightarrow \infty$} & \text{by~\eqref{ID asymptotic}}.
  \end{align*}
  Observe that if $m$ is odd, then term 1 vanishes and $\mathscr{P} = \{0, \ldots, m\}$, so there is no missing term in the sum over $p$.
  On the other hand, if $m$ is even, then the $\binom{m}{m/2}^2$ in the term 1 asymptotic~\eqref{term 1 asymptotic} is precisely the missing $p = m/2$ term in the sum over $p$.
  Hence in either case, in~\eqref{big sum over h fix m} as $n \rightarrow \infty$,
  \begin{align}
    \text{term 1 + term 2 + term 3} & = \sum_{h=0}^{n-2} \: \sum_{p=0}^m \textstyle \binom{h}{m-p}^2 \binom{n - h - 2}{p}^2 \nonumber \\
    & \sim \left( \sum_{p = 0}^m { \textstyle \binom{2p}{p} \binom{2m - 2p}{m-p} } \right) \cdot \frac{n^{2m+1}}{(2m+1)!} \label{main hook asymptotic} \\
    & = \frac{4^m}{(2m+1)!} \cdot n^{2m+1} & \text{by~\eqref{ID 4m}}. \label{123 asymptotic}
  \end{align}
  Finally, by performing the same analysis, we find that term 4 in~\eqref{big sum over h fix m} is asymptotic to a constant multiple of $n^{2m-1}$, which can be seen immediately by observing that the sum of the four lower binomial coefficient indices equals $2m - 2$, rather than $2m$ (as in terms 1, 2, and 3).
  Hence term 4 contributes nothing to the asymptotic behavior of~\eqref{big sum over h fix m} as $n \rightarrow \infty$.
  Thus by~\eqref{123 asymptotic} we have
  \[
    \eqref{big sum over h fix m} = \operatorname{tr} \check{P}^{(n-m, \: 1^m)} \sim \frac{4^m}{(2m+1)!} \cdot n^{2m+1} \text{ as $n \rightarrow \infty$},
  \]
  and therefore since $\left| \fS^\la \right| = (f^\la)^2 = \shortmset{n-m}{m}^2 = \binom{n-1}{m}^2 \sim n^{2m} / (m!)^2$, we obtain
  \begin{equation}
    \label{penultimate fixed m equation}
    \operatorname{tr} P^{(n-m, \: 1^m)} = \frac{\operatorname{tr} \check{P}^{(n-m, \: 1^m)}}{\binom{n-1}{m}^2} \sim \frac{4^m (m!)^2}{(2m+1)!} \cdot n = W_{2m+1} \cdot n,
  \end{equation}
  where $W_{2m+1}$ is the Wallis integral in~\eqref{Wallis odd}.
  Finally, applying~\eqref{E fix equals tr P} to~\eqref{penultimate fixed m equation}, we obtain
  \[
    \mathbb{E} \Big[ \left| \operatorname{Fix}(\sigma) \right| \Big] = \operatorname{tr} P^{(n-m, \: 1^m)} \sim W_{2m+1} \cdot n \text{ as $n \rightarrow \infty$},
  \]
  and so dividing by $n$, and letting $\ell = n - m$ approach infinity, we obtain
  \[
    \lim_{\ell \rightarrow \infty} \mathbb{E} \left[ \frac{\left| \operatorname{Fix}(\sigma) \right|}{n} \right] = \lim_{n \rightarrow \infty} \mathbb{E} \left[ \frac{\left| \operatorname{Fix}(\sigma) \right|}{n} \right] = W_{2m+1},
  \]
  which completes the proof in the hook case.

  \textbf{Two-row case: $\la = (\ell, m)$}. 
  We adopt the notation of Section~\ref{sec:two-row}.
  Recall from the proof of Lemma~\ref{lemma:size SV SW} that for all $1 \leq i \leq n$, we have $W(\ell,i,i) = 0$.
  Therefore by Theorem~\ref{thm:two-row}, the trace of $\check{P}^\la$ equals $\sum_i V(\ell,i,i)$.
  Substituting $m = \ell - n$ and converting ballot numbers to binomial coefficients via~\eqref{skew ballot definition}--\eqref{ballot definition}, we obtain
  \begin{align*}
    \operatorname{tr} \check{P}^\la = \sum_{i=1}^n V(n-m,i,i) &= \sum_{i=1}^n \sum_{k=1}^{n-m} \textstyle \tallballot{k-1}{i-k}^2 \tallballot{n - m - k}{m + k - i}^2 \\ 
    &= \sum_{i=1}^n \left( 
      \sum_{k=i-m}^{i}
    \textstyle \left[\binom{i-1}{i-k} - \binom{i-1}{i-k-1} \right]^2 \left[ \binom{n-i}{m + k - i} - \binom{n-i}{m + k - i - 1} \right]^2 \right ).
  \end{align*}
Making the substitutions $h = i-1$ and $p = m+k-i$, we obtain
\begin{align*}
  \operatorname{tr} \check{P}^\la &= \sum_{h=0}^{n-1} \; \sum_{p=0}^m \textstyle \left[ \binom{h}{m-p} - \binom{h}{m-p-1} \right]^2 \left[ \binom{n-h-1}{p} - \binom{n-i}{p-1} \right]^2 \\ 
  & = \sum_{h=0}^{n-1} \; \sum_{p=0}^m \textstyle \left[ \binom{h}{m-p}^2 \binom{n-h-1}{p}^2 + \text{lower-order terms in $n$}\right] & \text{by~\eqref{binomial order}} \\ 
  & \sim \sum_{h=0}^{n-1} \; \sum_{p=0}^m \textstyle \binom{h}{m-p}^2 \binom{n-h-1}{p}^2 \quad \text{as $n \rightarrow \infty$} \\ 
  & \sim \frac{4^m}{(2m+1)!} \cdot n^{2m+1} & \text{by~\eqref{main hook asymptotic} and~\eqref{123 asymptotic}}.
\end{align*}
The rest of the proof is now identical to the hook case above, since in the two-row case we still have
\[
  \left| \fS^\la \right| = (f^\la)^2 = \textstyle \shortballot{n-m}{m}^2 = \left[\binom{n}{m} - \binom{n}{m-1} \right]^2 \sim (n^m / m!)^2 = n^{2m} / (m!)^2 \text{ as $n \rightarrow \infty$},
\]
just as in the hook case~\eqref{penultimate fixed m equation}.
\end{proof}

\bibliographystyle{amsplain}
\bibliography{references}

\end{document}